\documentclass{amsart}
\usepackage{amssymb,amscd}  

\newcommand{\etalchar}[1]{$^{#1}$}

\providecommand{\MR}{\relax\ifhmode\unskip\space\fi MR }
\newtheorem{theorem}{Theorem}[section]
\newtheorem{lemma}[theorem]{Lemma}
\newtheorem{corollary}[theorem]{Corollary}

\newtheorem{proposition}[theorem]{Proposition}

\theoremstyle{definition}
\newtheorem{definition}[theorem]{Definition}

\newtheorem{remark}[theorem]{Remark} 
\numberwithin{equation}{section}

\newcommand\C{\mathbf{C}} 

\newcommand\Q{\mathbf{Q}}
\newcommand\Z{\mathbf{Z}}
\newcommand\F{\mathbf{F}}
\newcommand\twomatr[4]{\begin{pmatrix}#1&#2\\ #3&#4 \end{pmatrix}}

\newcommand\tensor{\otimes}
\newcommand\isomorphic{\cong}

\DeclareMathOperator{\Hom}{Hom}

\newcommand\directsum{\oplus}
\newcommand\union{\cup}

\newcommand\intersect{\cap}

\DeclareMathOperator{\image}{image}

\newcommand\compose{\circ}

\newcommand\Projective{{\bf P}} 

\newcommand\Half{\mathcal{H}} 
\newcommand\LL{\mathcal{L}} 
\newcommand\MM{\mathcal{M}} 
\newcommand\OO{\mathcal{O}} 

\newcommand\AV{\J}           
\newcommand\Rbar{\overline{R}} 
\newcommand\RR{\mathfrak{R}} 
\newcommand\Kbar{\overline{K}} 
\newcommand\Kbarstar{{\Kbar\,^*}} 
\newcommand\calG{\mathcal{G}} 
\newcommand\boxtensor{\boxtimes} 
\newcommand\LLhat{\widehat{\LL}} 

\newcommand\ua{\underline a}
\newcommand\uc{\underline c}

\newcommand{\calC}{\mathcal{C}}

\newcommand{\J}{\mathfrak{J}}

\newcommand{\fa}{\omega}
\newcommand{\kk}{\varkappa}

\begin{document}

\title{An analog of the Edwards model for Jacobians of genus 2 curves}

\author{E. V. Flynn}
\address{Mathematical Institute, University of Oxford, Andrew Wiles
Building, Radcliffe Observatory Quarter, Woodstock Road, Oxford
OX2 6GG, United Kingdom}
\email{flynn@maths.ox.ac.uk}

\author{Kamal Khuri-Makdisi}
\address{Mathematics Department,
American University of Beirut, Bliss Street, Beirut, Lebanon}
\email{kmakdisi@aub.edu.lb}

\thanks{October 24, 2023}
\thanks{Both authors thank Michael Stoll for 
organizing the Rational Points 2019 workshop in Schney, where they
began working together  on this problem.  The second named author (KKM)
gratefully acknowledges generous funding and a supportive research
environment during long scientific visits at both the Max Planck
Institute for Mathematics in Bonn (2021) and the Institute for
Advanced Study in Princeton (2022, with funding from the Charles
Simonyi Endowment).}

\begin{abstract}
We give the explicit equations for a $\Projective^3 \times \Projective^3$
embedding of the Jacobian of a curve of genus 2, which gives a
natural analog for abelian surfaces of the Edwards curve model
of elliptic curves. This gives a much more succinct description
of the Jacobian variety than the standard version in~$\Projective^{15}$.
We also give a condition under which, as for the Edwards curve,
the abelian surfaces have a universal group law.
\end{abstract}

\maketitle

\section{Introduction}
\label{section1}

In~\cite{BernsteinLangeEC} (generalising the form given 
in~\cite{EdwardsNormalForm}) a version of the model of an elliptic
curve and its group law are given, which have a particularly
elegant explicit description, subject to the existence of~$D_1$, a
point of order~$4$, defined over the ground field.
An interpretation of this embedding is to say that, if~$D$
is any point on the standard Weierstrass model of an elliptic curve, 
then we
map $D$
into $\Projective^1 \times \Projective^1$, via
the projective $x$-coordinate of~$D$, together with the
projective $x$-coordinate of~$D + D_1$. This model becomes
more succinct if we diagonalise
the coordinates on $\Projective^1$
with respect to addition
by~$E_1$, where $E_1 = 2 D_1$ is of order~$2$.

In this article, we
give an analog for Jacobians of
curves of genus~$2$, which have~$D_1$, a point of order~$4$ defined
over the ground field, as well as~$E_2$, a point of order~$2$
independent from~$E_1 = 2 D_1$. We make use of the
embedding of the Kummer surface
in $\Projective^3$,
given on p.18 of~\cite{CasselsFlynnBook}. Our embedding of
a point $D$ on
the
Jacobian variety into $\Projective^3 \times \Projective^3$ will be
via the image
of $D$, together with the image of $D + D_1$
on the Kummer surface.
This model becomes
more succinct if we diagonalise
the coordinates on $\Projective^3$
with respect to addition
by both~$E_1$ and~$E_2$.

In Section~\ref{section2}, we develop these ideas geometrically;
the main results will be Theorem~\ref{generatorsInEachBidegree},
which describes how many independent defining equations there 
are of each bidegree,
and Theorem~\ref{PQAddition} which describes the degree of
the equations in the matrices which give the group law.
In Section~\ref{section3}, we give a brief derivation of the Edwards curve,
in the above style, explaining how the group law is universal
when a specified point is not defined over the ground field;
when we say that the group law is universal, we are referring to
its application to rational points over the base field.
In Section~\ref{section4}, we derive our $\Projective^3 \times \Projective^3$
embedding of the Jacobian variety of our genus~$2$ curve,
giving explicitly a set of defining equations for the variety
in Theorem~\ref{theoremdefeqns} (using Theorem~\ref{generatorsInEachBidegree}
to know that we have a complete set of defining equations),
and its group law in Theorem~\ref{theoremGroupLawGenus2}. 
These are considerably
more succinct than the standard versions in~$\Projective^{15}$,
such as those described in~\cite{CasselsFlynnBook}.
We also give a twisted version of the abelian surface,
analogous to the twist performed on Edwards curves.
In Section~\ref{section5}, we also give in Corollary~\ref{genus2universal} 
(a consequence of Theorem~\ref{alwaysnonzero})
a condition on the parameters under which, as for the Edwards curve, 
the abelian surfaces have a universal group law. The situation here
is more subtle since the degenerate locus is geometrically a
possibly reducible curve (rather than a pair of points,
as in the elliptic curve case), and so we need to 
construct a condition on the parameters that prevents this curve from 
having any points over the ground field.

\section{Generators of the ideal of relations}
\label{section2}

Our intention in this section is to describe a
$\Projective^3 \times \Projective^3$ embedding of the
Jacobian of a genus two curve;
the main results will be Theorem~\ref{generatorsInEachBidegree},
which describes how many independent defining equations there 
are of each bidegree, and Theorem~\ref{PQAddition} which describes the degree of
the equations in the matrices which give the group law. 
Note that the statements of
Theorem~\ref{generatorsInEachBidegree} 
and Theorem~\ref{PQAddition} are what will be
used later in Section~4; otherwise, the notation and objects
in Section~2 will not be required later. Any reader who
is primarily interested in the results of Sections~4 and~5
(and not interested in the justification of the theorems
in this section) is welcome just
to read the statements of Theorem~\ref{generatorsInEachBidegree}
and Theorem~\ref{PQAddition},
and otherwise proceed to Section~\ref{section3}.

We work with the Jacobian $\AV$ of a genus two curve $\calC$
throughout, viewing $\AV$ as $\text{Pic}^0 \calC$.  Then $\AV$ is a
principally polarized abelian surface. 
The article~\cite{LubiczRobert}, which does not limit itself to
Jacobians (or for that matter to dimension~$2$), already introduces the
idea of using the Kummer coordinates of a point $p \in \AV$, along
with the Kummer coordinates of $p+p_0$ for a fixed $p_0$ that is not a
2-torsion point.  They consider a general $p_0$, and observe that the
resulting map $\AV \to \Projective^3 \times \Projective^3$ given by
$p \mapsto (\kappa(p),\kappa(p+p_0))$ for the Kummer map $\kappa$
gives an embedding of $\AV$ into $\Projective^3 \times \Projective^3$,
provided $p_0$  
is not a point of order~$2$.  They also prove a number of results and
give effective methods to work with $\AV$, using
differential additions and other constructions that involve viewing
the fibre of $\kappa$ (or its translate) over a rational point of
$\AV$ as the spectrum of a two-dimensional algebra over the ground
field, so as to capture the two preimages in $\AV$ even if they are
not individually rational.

In this article, we restrict to $p_0 = D_1$, a
point of exact order~$4$, and will impose in Section~\ref{section4}
certain rationality conditions on 
the subgroup $H \subset \AV$ generated by $D_1$ and a second point
$D_2$ of exact order~$4$.  We alert the reader that $H$ is not
isotropic under the Weil pairing on $\AV[4]$: in fact, $e_4(D_1,D_2) =
-1$.  On the other hand, the subgroup $2H \isomorphic \Z/2\Z \times
\Z/2\Z$ is isotropic in $\AV[2]$, and corresponds to the kernel of
a Richelot isogeny.

We first describe the Edwards-like construction over an algebraically
closed field $\Kbar$, without worrying about rationality over the
ground field $K$.  In this section, we routinely identify $\calC$ and
$\AV$ with 
their set of $\Kbar$-valued points.  We require $\Kbar$ not to be of
characteristic~$2$, so as to invoke the basic addition formula of
Mumford (p.~324 in Section~3 of~\cite{MumfordEqnsI}; see also
Theorem~8 of~\cite{KempfLinearSystems}).  This addition formula, which
follows from the isogeny theorem for algebraic theta functions, is an
algebraic generalization of the analytic addition
formula~\eqref{eqnThetaAddition} for complex theta series.
The algebraic formulas we write down,
Theorem~\ref{thmAlgebraicAdditionFormula} and Corollary~\ref{corFdalpha},
are direct consequences of the basic addition formula, but it is
helpful for us to write them down explicitly and concretely, with a
careful note of isomorphisms such as $h = h_{p,q}$
in~\eqref{eqAlgebraicThetaAddition} and $\phi \compose j$
in~\eqref{eqFdalphaFormula}. 

If the reader is willing to accept that the structure of our results
for general~$\Kbar$ is completely mirrored by the structure
when $\Kbar = \C$, then it may be helpful to consult an earlier
version of this article, available at
\begin{center}
  \texttt{https://arxiv.org/abs/2211.01450v2}
\end{center}
where a shorter analytic version of the argument in this section is given,
using explicit computations with complex theta functions.


\subsection*{Preliminaries}

For a (geometric) point $x \in \AV$, we denote by $T_x: \AV \to \AV$ the
translation map: thus $T_x(p) = p+x$.
For $k \in \Z$, we denote by $[k]: \AV \to \AV$ the multiplication by
$k$.  For $k \geq 2$, with $k$ invertible in $\Kbar$ (mostly $k = 2$
or $k = 4$), we write $\AV[k]$ for the kernel of $[k]$.  Then $\AV[k]$
is isomorphic to $(\Z/k\Z)^4$.

The abelian surface~$\AV$ is principally polarized; let $\LL$ be a
symmetric line bundle on~$\AV$ (meaning $[-1]^* \LL \isomorphic \LL$)
giving rise to the principal polarization.  Then the isomorphism class
of $\LL$ is unique up to replacing $\LL$ by $T_x^*\LL$ for some
$x \in \AV[2]$.  Nonetheless, the isomorphism class of $\LL^2$ (hence
also of any even power of $\LL$) is unique; in the terminology
of~\cite{MumfordEqnsI}, the bundle $\LL^2$ is totally symmetric.

When $\Kbar = \C$, we can view $\AV$ analytically as the 
complex torus $\C^2/(\Z^2 + \Omega \Z^2)$, where $\Z^2$ and $\C^2$ are
spaces of column vectors, and $\Omega \in \Half_2$ is a point in the
Siegel upper half-space; that is, $\Omega$ is a symmetric matrix in
$M_2(\C)$ whose imaginary part is positive definite.  In that setting,
one can use complex-analytic 
theta functions to describe global sections of the powers $\LL^k$ of
(one choice of) the line bundle $\LL$.
We normalize our complex theta functions as follows.

\begin{definition}
\label{defTheta}
Let $k \geq 1$, and let $\beta \in \Z^2$.
For $z \in \C^2$ we define  
\begin{equation}
\label{eqDefTheta}
\theta_{k,\beta}(z)
= \sum_{n \in \beta + k\Z^2}
       e(\frac{n\cdot\Omega n}{2k} + n \cdot z),
\end{equation}
where $e(x) = \exp(2\pi i x)$ for $x \in \C$, and $\cdot$ is the
standard bilinear product on $\C^2$, so
$n\cdot \Omega n = {}^t n \Omega n$ and $n \cdot z = {}^t n z$.
Note that $\theta_{k,\beta}$ depends only on the image of $\beta$ in
$(\Z/k\Z)^2$.
In terms of theta functions with
characteristics (see~\cite{MumfordTataI}, page~123), we have
$\theta_{k,\beta}(z)
= \theta{\genfrac{[}{]}{0pt}{}{\beta/k}{0}}(k\Omega, kz)$.
\end{definition}

In the above analytic setting, it is standard that 
the functions $\{\theta_{k,\beta}\}_{\beta \in (\Z/k\Z)^2}$ are a
basis for $H^0(\AV,\LL^k)$.  Here the space $H^0(\AV,\LL^k)$ can be
identified with the space of holomorphic functions $f: \C^2 \to \C$
which transform according to 
\begin{equation}
\label{eqAnalyticLLk}
\begin{split}
f(z + \ell)
   &= f(z), \qquad \ell \in \Z^2,\\
(f|_k m)(z) : = e(km \cdot\Omega m/2 + km \cdot z) f(z + \Omega m)
   &= f(z), \qquad m \in \Z^2.\\
\end{split}
\end{equation}
We also have the standard addition formula, for $\beta, \beta' \in
\Z^2$ and $z,w \in \C^2$:
\begin{equation}
\label{eqnThetaAddition}
\theta_{k,\beta}(z+w) \theta_{k,\beta'}(z-w)
= \sum_{c \in (\Z/2\Z)^2}
      \theta_{2k,\beta+\beta'+kc}(z)
      \theta_{2k,\beta-\beta'+kc}(w).
\end{equation}
In the above,
the sum is over representatives
$c \in \{ {}^t(0,0), {}^t(1,0), {}^t(0,1), {}^t(1,1) \}$,
so $kc$ makes sense as a $2$-torsion element in $(\Z/2k\Z)^2$. 

For general $\Kbar$, we will use an algebraic version of the above
formula, primarily
in the case where $k = 2$ and $2k = 4$.
We will also need the
analog of another formula that relates specific linear combinations of
the $\{\theta_{4,\gamma}(z)\}_\gamma$ to the shifted theta functions
$\theta_{1,0}(2z + \alpha/2 + \Omega\beta/2)$, evaluated at $2z$.
These shifted theta functions are essentially sections of $T_x^*\LL$,
where $x \in \AV[2]$ corresponds to $\alpha/2 + \Omega\beta/2 \in \C^2$.

We first pin down canonical bases of $H^0(\AV,\LL^k)$
in terms of algebraic theta functions, summarizing
what we need from Sections 1--3 of~\cite{MumfordEqnsI}.
We follow Mumford's normalization, with some notation
from~\cite{KempfLinearSystems}.
For any $k \geq 1$ that is invertible in $\Kbar$
(mainly $k \in \{2,4\}$), it is standard that the  
translates of $\LL^k$ that are isomorphic to $\LL^k$ are exactly the
translates by $k$-torsion points:
\begin{equation}
  \label{eqMumfordsKgroup}
  \{ x \in \AV \mid T_x^*\LL^k \isomorphic \LL^k \} = \AV[k].
\end{equation}
Mumford constructs the theta group $\calG_k$ of $\LL^k$ (more
generally, he constructs the theta group of any ample line bundle of
separable type) as a central extension
\begin{equation}
  \label{eqThetaGroupExactSequence}
  0 \to \Kbarstar \to \calG_k \to \AV[k] \to 0.
\end{equation}
Here an element $\tilde{x} \in \calG_k$ with image $x \in \AV[k]$ is
actually a pair $\tilde{x} = (x,\phi)$ with a choice of isomorphism
$\phi: \LL^k \to T_x^* \LL^k$.  We can view $\phi$ concretely as an
algebraic family of isomorphisms between the fibers
$\LL^k_p$ and $(T_x^* \LL^k)_p = \LL^k_{p+x}$, where $\LL^k_p$ and
$\LL^k_{p+x}$ are one-dimensional vector spaces over $\Kbar$:
\begin{equation}
  \label{eqThetaGroupElt}
  \phi_p:  \LL^k_p \to \LL^k_{p+x}, \qquad p \in \AV.
\end{equation}
We also have an important action of $\calG_k$ on the space
$H^0(\AV,\LL^k)$.  Working through the definition of $U_z(s)$ on
p.~295 of~\cite{MumfordEqnsI}, we can describe the action concretely
as follows.  View any section $s \in H^0(\AV, \LL^k)$ as an algebraic
family of ``values'' $s(p) \in \LL^k_p$ at each varying fiber.  Then the
action of $\tilde{x} = (x,\phi)$ produces a section
$\tilde{x}*s \in H^0(\AV, \LL^k)$ with values
\begin{equation}
  \label{eqThetaGroupAction}
  (\tilde{x}*s)(p) = \phi_{p-x}(s(p-x)), \qquad p \in \AV.
\end{equation}
The Weil pairing associated to $\LL^k$ is an alternating
map that takes two elements $x,y \in \AV[k]$ to a $k$th root of unity
$e_k(x,y) \in \Kbarstar$ given by the commutator
\begin{equation}
  \label{eqWeilPairing}
  e_k(x,y) = \tilde{x} \tilde{y} \tilde{x}^{-1} \tilde{y}^{-1}.
\end{equation}
Here $\tilde{x}, \tilde{y} \in \calG_k$ are any lifts of the
elements $x,y \in \AV[k]$.

Mumford shows that the group $\calG_k$ has the structure of a
Heisenberg group, and that its 
action on $H^0(\AV,\LL^k)$ has a canonical structure up to
isomorphism.  Take a symplectic decomposition
\begin{equation}
  \label{eqTorsionDecomposition}
\AV[k] = A_k \directsum B_k,
\end{equation}
where $A_k, B_k \subset \AV[k]$ are subgroups with
$A_k \isomorphic B_k \isomorphic (\Z/k\Z)^2$, and such that each of
$A_k$ and $B_k$ is 
isotropic for the Weil pairing; then $e_k$ sets up a perfect pairing
between $A_k$ and $B_k$.  One can split the extension $\calG_k$
over each of $A_k$ and $B_k$, leading to two different injective group
homomorphisms (written with the same tilde notation)
\begin{equation}
  \label{eqLagrangeLifts}
  A_k \to \calG_k, \qquad a \mapsto \tilde{a};
  \qquad\qquad\qquad
  B_k \to \calG_k, \qquad b \mapsto \tilde{b}.
\end{equation}
We denote by $\tilde{A}_k$ and $\tilde{B}_k$ the images of these two
homomorphisms.

In Mumford's treatment, one takes a theta structure on
$\calG_k$. This includes a specific choice of isomorphism between
$B_k$ and $(\Z/k\Z)^2$; then $A_k$ is isomorphic to
$\Hom(B_k,\Kbarstar)$, by identifying $a \in A_k$ with
$l_a: B_k \to \Kbarstar$ given by $l_a(b) = e_k(b,a)$.  The theta
structure also includes appropriate choices of lifts of these
identifications to $\tilde{A}_k$ and $\tilde{B}_k$.  Then Mumford
shows that the action of $\calG_k$ on $H^0(\AV,\LL^k)$ is isomorphic
to a Schr\"odinger representation on the space he calls $V(\delta)$,
which in our situation consists of $\Kbar$-valued functions on 
$(\Z/k\Z)^2$ (see p.~297 of~\cite{MumfordEqnsI} for the precise action
of a general $\calG(\delta)$ on $V(\delta)$).

For this article, it will be easier to reword this 
action (while still keeping Mumford's normalization) in the language
used on pp.~68--69 of~\cite{KempfLinearSystems}.  So we will in fact
replace Mumford's space $V(\delta)$ with the space $W_k$ of
$\Kbar$-valued functions on $B_k$.  With this modification,
the actions of elements of $\tilde{A}_k$ and $\tilde{B}_k$ on a
function $(f: B_k \to \Kbar) \in W_k$ are
given by:
\begin{equation}
  \label{eqThetaGroupActionSchrodinger}
  (\tilde{a}*f)(b') = l_a(b') f(b') = e_k(b',a) f(b'),
  \qquad\qquad
  (\tilde{b}*f)(b') = f(b+b').
\end{equation}
These correspond to the operators $U_{(1,0,l_a)}$ and $U_{(1,b,1)}$ of
Mumford.  (The operator $U_{(\lambda,b,l_a)}$ corresponds to the
product $\lambda \tilde{a} \tilde{b} \in \calG_k$.)

For our application, it is easier to write everything in
terms of the action of the lifts $\tilde{a}$ and $\tilde{b}$ on the basis
$\{\delta_c\}_{c \in B_k}$ for the space $W_k$.  Here, as usual,
$\delta_c(b') = 1$ precisely when $b' = c$, and $\delta_c(b') = 0$
otherwise.  We then have
\begin{equation}
  \label{eqThetaGroupActionOnDeltaFunctions}
  \tilde{a}*\delta_c = e_k(c,a) \delta_c,
  \qquad\qquad
  \tilde{b}*\delta_c = \delta_{c-b},
  \qquad\qquad
  a \in A_k,
  \qquad
  b, c \in B_k.
\end{equation}
We can now define the algebraic theta functions.  The key idea here is
that the representation of $\calG_k$ on $H^0(\AV,\LL^k)$ is isomorphic
to the above representation on $W_k$.  Moreover, this representation
is irreducible.  Hence there exists an isomorphism, which by
Schur's Lemma is unique up to a scalar in $\Kbarstar$, between $W_k$ and
$H^0(\AV, \LL^k)$.  With respect to this isomorphism, each basis
element $\delta_c \in W_k$ corresponds to a section $\theta_{[k],c} \in
H^0(\AV,\LL^k)$, with the same transformation properties as
in~\eqref{eqThetaGroupActionOnDeltaFunctions}.  This determines the
$\{\theta_{[k],c}\}_{c \in B_k}$ up to a common constant factor.  Rewriting
$b'$ instead of $c \in B_k$, we therefore obtain the following action
on the various $\theta_{[k],b'} \in H^0(\AV,\LL^k)$:
\begin{equation}
  \label{eqThetaGroupActionOnAlgebraicThetas}
  \tilde{a}*\theta_{[k],b'} = e_k({b'},a) \theta_{[k],b'},
  \qquad\qquad
  \tilde{b}*\theta_{[k],b'} = \theta_{[k],b'-b}.
\end{equation}

In the complex analytic setting for all the above, where
$\AV = \C^2/\Lambda$ with the lattice $\Lambda = \Z^2 + \Omega \Z^2$,
the analogs of the above constructions are
$J_k = \frac{1}{k}\Lambda/\Lambda = A_k \directsum B_k$, with $A_k$ being
(the image in $\AV$ of) $\frac{1}{k}\Z^2$, and $B_k$ being (the image of)
$\frac{1}{k}\Omega\Z^2$.  Viewing global sections of $\LL^k$ as
functions satisfying~\eqref{eqAnalyticLLk}, we have the following
action of our lifts $\tilde{a}, \tilde{b} \in \calG_k$.  Note the
minus signs, similarly to~\eqref{eqThetaGroupAction}, as well as the
notation $f|_k m$ from~\eqref{eqAnalyticLLk}, where we now allow
$m = -\beta/k \in \Q^2$.
\begin{equation}
  \label{eqThetaGroupActionAnalytic}
  (\widetilde{\Bigl[\frac{\alpha}{k}\Bigr]}*f)(z)
      = f(z-\frac{\alpha}{k}),
  \qquad\qquad
  \widetilde{\Bigl[\Omega\frac{\beta}{k}\Bigr]}*f
      = f|_k (\frac{-\beta}{k}),
  \qquad
  \alpha,\beta \in \Z^2.
\end{equation}
The algebraic theta function $\theta_{[k],\Omega\beta'/k}$ then
corresponds to our analytic theta function $\theta_{k,\beta'}$.  The
Weil pairing over~$\C$ is characterized by
$e_k(\Omega\beta/k,\alpha/k) = e(-\beta \cdot \alpha/k)$.

We now return to general~$\Kbar$, and present  Mumford's addition
formula for algebraic theta functions in terms of the concrete
formalism above.  The product abelian variety $\AV \times \AV$ is
equipped with two projection maps 
$\pi_1, \pi_2: \AV \times \AV \to \AV$
to the first and second factors.  We consider on $\AV \times \AV$ the
line bundle $\LL^k \boxtensor \LL^k = \pi_1^* \LL^k \tensor \pi_2^*
\LL^k$.  Its fiber at the point $(p,q) \in \AV \times \AV$ is
the one-dimensional $\Kbar$-vector space $\LL_p^k \tensor \LL_q^k$.
A theta structure on $\calG_k$ for $\LL^k$ gives rise to a closely
related theta structure for $\LL^k \boxtensor \LL^k$.  We note in
particular that $H^0(\AV \times \AV, \LL^k \boxtensor \LL^k)$ has
a canonical basis
$\{ \theta_{[k],b} \boxtensor \theta_{[k],b'} \}_{b,b' \in B_k}$.
Here the notation $s \boxtensor t$, for $s,t \in H^0(\AV, \LL^k)$, is
defined by
\begin{equation}
  \label{eqBoxTensorOfSections}
  s \boxtensor t = (\pi_1^* s) \tensor (\pi_2^* t).
\end{equation}
The value of the section $s \boxtensor t$ at the point $(p,q)$ is of course
\begin{equation}
  \label{eqBoxTensorValues}
  (s \boxtensor t)(p,q) = s(p) \tensor t(q)
    \in \LL_p^k \tensor \LL_q^k
    = (\LL^k \boxtensor \LL^k)_{(p,q)}.
\end{equation}
As Mumford explains in Section~3 of~\cite{MumfordEqnsI},
the algebraic addition formula comes out of the isogeny
$\xi: \AV \times \AV \to \AV \times \AV$ given by
$\xi(p,q) = (p+q, p-q)$.  Mumford shows algebraically that we
have an isomorphism
$h: \xi^*(\LL^k \boxtensor \LL^k)
\isomorphic \LL^{2k} \boxtensor \LL^{2k}$.
(When $\Kbar = \C$, this isomorphism is easy to see analytically.)
We can view $h$ as an algebraic family of isomorphisms between the
fibers of the two line bundles, similarly to
the situation in~\eqref{eqThetaGroupElt}.
We thus obtain a family of isomorphisms $\{h_{p,q}\}_{(p,q) \in \AV
  \times \AV}$
between the one-dimensional fibers
$(\xi^*(\LL^k \boxtensor \LL^k))_{(p,q)}
                 = \LL^k_{p+q} \tensor \LL^k_{p-q}$,
and
$(\LL^{2k} \boxtensor \LL^{2k})_{(p,q)}
                 = \LL^{2k}_p \tensor \LL^{2k}_q$:
\begin{equation}
  \label{eqHpq}
  h_{p,q} : \LL^k_{p+q} \tensor \LL^k_{p-q}
     \to \LL^{2k}_{p} \tensor \LL^{2k}_{q}.
\end{equation}
We can now finally invoke Mumford's fundamental addition formula on
p.~324 of~\cite{MumfordEqnsI},
to obtain
the following formula.  In contrast to the analytic situation,
this result needs $k$ to be even, unless (as in Theorem~6
of~\cite{KempfLinearSystems}) we allow a possible modification to the
original~$\LL$.
\begin{theorem}
  \label{thmAlgebraicAdditionFormula}
Suppose that $k$ is both even and invertible in $\Kbar$.
Choose a symplectic decomposition $\AV[2k] = A_{2k} \directsum B_{2k}$;
this also determines a decomposition $\AV[k] = A_k \directsum B_k$,
with $B_k = [2]B_{2k} = B_{2k} \intersect \AV[k]$, and an analogous
$A_k$.  Then there exist compatible theta structures on $\calG_k$ and
$\calG_{2k}$, which give rise to specific choices of bases
$\{\theta_{[k],b}\}_{b \in B_k}$
and $\{\theta_{[2k],d}\}_{d \in B_{2k}}$, respectively for 
$H^0(\AV,\LL^k)$ and $H^0(\AV,\LL^{2k})$, that satisfy, for all
$p,q \in \AV$:
\begin{equation}
  \label{eqAlgebraicThetaAddition}
  \begin{split}
    h_{p,q}&(\theta_{[k],b_1}(p+q) \tensor \theta_{[k],b_2}(p-q)) \\
    &=
    \sum_{c \in B_{2k} \intersect \AV[2]}
      \theta_{[2k],d_1 + d_2 + c}(p) \tensor \theta_{[2k],d_1 - d_2 + c}(q).\\
  \end{split}
\end{equation}
Here $b_1, b_2 \in B_k$, and we choose $d_1, d_2 \in B_{2k}$
satisfying $2d_1 = b_1, 2d_2 = b_2$.  (The sum on the right hand side
is independent of the choices of $d_1, d_2$.)
\end{theorem}
\begin{proof}
The only point to make here is that $\LL^k$ is totally symmetric
because $k$ is even, so we can apply Mumford's construction of
symmetric compatible theta structures.  One then translates Mumford's
fundamental addition formula, while carefully working
through all the definitions, and following the normalizations
of~\cite{MumfordEqnsI} and our concrete expressions above for the
isomorphisms.  As Mumford points out, one can choose the isomorphisms
so as to make the common constant factor $\lambda$ in his formula
equal to~$1$. 

The formula~\eqref{eqAlgebraicThetaAddition} is parallel
to the statements in Theorem~8 of~\cite{KempfLinearSystems} and
Theorem~6.5 of~\cite{KempfBook}, but the normalizations in Kempf
appear to be slightly different (for example, an element of the theta
group in~\cite{KempfBook} is defined using an isomorphism from
$T_x^*\LL^k$ to $\LL^k$ and not its inverse), so we preferred to
follow scrupulously the treatment in~\cite{MumfordEqnsI}.
\end{proof}
\begin{remark}
In the works cited above, the compatible theta structures can be set
up so that the actions of both negation $[-1]:\AV \to \AV$ and
doubling $[2]:\AV \to \AV$
become transparent.  We will not need these in full generality.  We
will however treat ``by hand'' a specific case of the action of doubling,
which does not seem to be included in the results
of~\cite{MumfordEqnsI} because the original line bundle $\LL$ is not
totally symmetric.  The result on doubling that we need is included 
in~\cite{KempfLinearSystems} and~\cite{KempfMultiplication}, but we
will make it more explicit below.
\end{remark}

From now on, we limit ourselves to
using~\eqref{eqAlgebraicThetaAddition} only when $k=2$.  We first note
the standard ``diagonalization'' of~\eqref{eqAlgebraicThetaAddition}
obtained by introducing a character $\chi: B_2 \to \Kbarstar$ and
forming the linear combination
$\sum_{c \in B_2}
  \chi(c)
  \theta_{[2], b+c}
    \boxtensor
  \theta_{[2], c}$.
The character $\chi$, which takes values in $\{\pm 1\}$, can
be written as $\chi(c) = e_4(c,\alpha)$ for some choice of $\alpha \in
A_4$; the character $\chi$ depends only on the coset of $\alpha$ in
$A_4/A_2$.  Indeed, $e_4(\cdot,\cdot)$ is isotropic on the points of
$\AV[2] = A_2 \directsum B_2 \subset \AV[4] = A_4 \directsum B_4$.
Let us also write $b = 2d$ for some choice of $d \in B_4$.  Applying
$\xi^*$ to our linear combination above, followed by the isomorphisms
$h_{p,q}$, we obtain from~\eqref{eqAlgebraicThetaAddition}
the following important identity for all $p,q \in \AV$:
\begin{equation}
  \label{eqBasicAdditionFormula}
  h_{p,q}(\sum_{c \in B_2}
     e_4(c, \alpha)
     \theta_{[2],2d+c}(p+q) \tensor
     \theta_{[2],c}(p-q)
     )
   =
     F_{d,\alpha}(p) \tensor F_{d,\alpha}(q),
\end{equation}
where for $d \in B_4$ and $\alpha \in A_4$ we have
\begin{equation}
  \label{eqDefOfF}
  F_{d,\alpha} = \sum_{c \in B_2} e_4(c,\alpha) \theta_{[4],d+c}
  \in H^0(\AV,\LL^4).
\end{equation}
As already observed, $F_{d,\alpha}$ is unchanged if we add an element
of $A_2$ to $\alpha$; note however that if we add an element
$b' \in B_2$ to $d$, then $F_{d + b', \alpha} = e_4(b',\alpha)
F_{d,\alpha}$ (remember that $e_4(b',\alpha) \in \{ \pm 1 \}$, because
$b'$ is $2$-torsion).  It is easy to see that the $F_{d,\alpha}$ form
a basis of $H^0(\AV,\LL^4)$ as $\alpha$ and $d$ run over any fixed
choice of representatives for each coset in $A_4/A_2$ and $B_4/B_2$.

We now wish to relate $F_{d,\alpha}$ to the pullback via $[2]$ of a
translate $T_x^*(\theta_{[1],0}) \in H^0(\AV,T_x^*\LL)$, for a
suitable $x \in \AV[2]$.  This is a known result, and is easy to show
analytically, when $\Kbar = \C$,  The general case is covered in the
next to last paragraph of~\cite{KempfMultiplication}, immediately
following the proof of Theorem~7 there.  We state and prove the result
in the language of this article.  Recall that our choice of symmetric
line bundle $\LL$ is only unique in the first place up to translation
by an element of $\AV[2]$.  Moreover, since $\LL$ is symmetric, we
have $[2]^*\LL \isomorphic \LL^4$.

\begin{proposition}
  \label{propositionFixLL}
Up to translating $\LL$ by a $2$-torsion point, we can identify
$F_{0,0}$ with a nonzero constant multiple of $[2]^* \theta_{[1],0}$.
To absorb the constant, we choose a specific
isomorphism $j: [2]^*\LL \to \LL^4$, giving rise to a family of
isomorphisms $j_p: ([2]^*\LL)_p = \LL_{2p} \to \LL_p^4$, for which
\begin{equation}
  \label{eqF00Formula}
  F_{0,0}(p) = j_p(\theta_{[1],0}(2p)),
  \qquad\qquad
  p \in \AV.
\end{equation}
\end{proposition}
\begin{proof}
Consider $2$-torsion elements $a \in A_2 \subset A_4$ and
$b \in B_2 \subset B_4$, as well as their lifts $\tilde{a}, \tilde{b}
\in \calG_4$.  These lifts generate a partial splitting of the extension
$\calG_4$ over $\AV[2] = A_2 \directsum B_2$, which is a maximal
isotropic subgroup of $\AV[4]$ under the Weil pairing $e_4$. 
The section $F_{0,0}$ is invariant under the action of these
lifts from $\AV[2]$. 
Now this partial splitting gives descent data for
$\LL^4$ under all translations by $\AV[2]$: see pp.~290--291
of~\cite{MumfordEqnsI}.  The descent produces a line
bundle $\tilde{\LL}$ on $\AV$, and an isomorphism $j: [2]^*
\tilde{\LL} \to \LL^4$.  We can also descend $F_{0,0}$ to a nonzero global
section of $\tilde{\LL}$.  The proof will be complete once we show
that $\tilde{\LL}$ is in fact a translate $T_x^*\LL$, for some $x \in
\AV[2]$, for we then replace $\LL$ by $\tilde{\LL}$, and we know that
(the new) $H^0(\AV,\LL)$ is one-dimensional, with basis
$\theta_{[1],0}$.  We finally adjust $j$ by a constant to
obtain~\eqref{eqF00Formula}.

To show our assertion about $\tilde{\LL}$, note first that $[2]^*\LL$
and $[2]^*\tilde{\LL}$ are both isomorphic to $\LL^4$.  Hence
$\LL^{-1} \tensor \tilde{\LL}$ is in the kernel of the homomorphism
$[2]^*: \text{Pic} \AV \to \text{Pic} \AV$.  The effect of $[2]^*$ on
the N\'eron-Severi group is to multiply by~$4$, but $NS(\AV)$ is
torsion-free.  Hence $\LL^{-1} \tensor \tilde{\LL}$ belongs to
$\text{Pic}^0 \AV$, which implies that $\tilde{\LL}\tensor \LL^{-1}$
is isomorphic to $T_x^* \LL \tensor \LL^{-1}$ for some $x \in \AV$,
via the usual isomorphism $\Phi: \AV \to \text{Pic}^0 \AV$ given
by $\Phi(x) = T_x^* \LL \tensor \LL^{-1}$.  We have thus shown that
$\tilde{\LL} \isomorphic T_x^*\LL$.  Take any $y \in \AV$ with $2y = x$.
We know that $\LL^4 \isomorphic [2]^* \tilde{\LL} \isomorphic [2]^*
T_x^*\LL$, and this last is isomorphic to $T_y^* [2]^* \LL \isomorphic
T_y^* \LL^4$.  We deduce that $4y = 0$ and hence that
$x$ is a $2$-torsion point.
\end{proof}
\begin{remark}
One can give a more conceptual proof of the above proposition.  The
different choices of 
lifts from $\AV[2] \to \calG_4$ correspond to different ways to
descend $\LL^4$ along $\AV[2]$ to varying line bundles.  One of these
descents gives $\tilde{\LL}$, while another gives $\LL$.  But any two
lifts $\AV[2] \to \calG_4$ differ by a character of $\AV[2]$ with
values in $\{\pm 1\}$, and such a character can be obtained by a Weil
pairing with a fixed $2$-torsion point which corresponds to our $x$.
Essentially the same argument about changing $\tilde{\LL}$ to $\LL$
appears in Lemma~5 of~\cite{KempfLinearSystems} (see also Theorem~6 there).
Compare also to property~(5) on p.~228
of~\cite{MumfordAV} (under ``Functorial properties of $e^L$'').
\end{remark}

\begin{corollary}
  \label{corFdalpha}
Let $d \in B_4$ and $\alpha \in A_4$.  Define the element $\tilde{x} =
(x, \phi) = (\widetilde{-d}) (\tilde{\alpha}) \in \calG_4$; hence $x = -d +
\alpha$ and $\phi: \LL^4 \to T_{-d+\alpha}^* \LL^4$ is the
corresponding isomorphism.  We then have, for all $p \in \AV$:
\begin{equation}
  \label{eqFdalphaFormula}
  F_{d,\alpha}(p)
   = \phi_{p + d - \alpha} \compose j_{p + d - \alpha}
       ( \theta_{[1],0}(2p + 2d - 2\alpha) ).
\end{equation}
\end{corollary}
\begin{proof}
By~\eqref{eqThetaGroupActionOnAlgebraicThetas}, we have
$F_{d,\alpha} = \tilde{x}*F_{0,0}$.  Now
apply~\eqref{eqThetaGroupAction} to obtain the above formula.
\end{proof}

Our main use of~\eqref{eqFdalphaFormula} will be to determine
conditions under which $F_{d,\alpha}$ vanishes or not at certain points,
by reducing the question to whether $\theta_{[1],0}$ vanishes at the
corresponding points.  We first state a mostly standard result about
$\theta_{[1],0}$.

\begin{lemma}
  \label{lemmaTheta1Nonvanishing}
  Recall that $\AV$ is the Jacobian of a genus~$2$ curve~$\calC$, and
  that $0 \neq \theta_{[1],0} \in H^0(\AV,\LL)$, where $\LL$ is a
  symmetric line bundle giving the principal polarization.
  \begin{enumerate}
  \item
    Among the $16$ points of $\AV[2]$, we have that $\theta_{[1],0}$
    vanishes at precisely~$6$ of them (and is nonzero at the remaining
    $10$ points).
  \item
    $\theta_{[1],0}$ does not vanish at any point $p \in \AV$
    whose order is precisely~$4$, in other words, for $p \in \AV[4] -
    \AV[2]$.
  \end{enumerate}
\end{lemma}
\begin{proof}
Write $\Theta$ for the vanishing locus of $\theta_{[1],0}$.  We know
that $\Theta$ is a
symmetric theta divisor on~$\AV$, so it is the image of
$\calC$ under a suitable Abel-Jacobi map into $\AV$.  Let
$\{w_0, \dots, w_5\} \subset \calC$ be the six Weierstrass points.  Then
one choice of symmetric theta divisor is the set
$\calC_0 = \calC - w_0 \subset \AV$, by which we mean the set of all divisor
classes of the form $[v-w_0] \in \AV = \text{Pic}^0 \calC$,
parametrized by $v \in \calC$.  Then
$\OO_\AV(\calC_0)$ is a symmetric line bundle in the algebraic
equivalence class of $\LL$.  Hence $\LL$, which is also
symmetric, is isomorphic to the translate of $\OO_\AV(\calC_0)$ by
some $2$-torsion point $x \in \AV[2]$.  This means that $\Theta$ is the
set of points $\{ [v-w_0] + x \mid v \in \calC \}$, where $v$ varies on
$\calC$ and $x \in \AV[2]$ is a fixed $2$-torsion point.

Let us prove the first assertion above.
For a $2$-torsion point $p \in \AV[2]$ to lie on $\Theta$, we require
$p-x$ to be a $2$-torsion point of the form $[v - w_0]$.  But the only
such $2$-torsion points
are the six points
$\{[w_i - w_0] \mid 0 \leq i \leq 5\}$.  This gives us precisely six
choices of $p$ where $\theta_{[1],0}(p) = 0$.

As for the second assertion, it says that when $p$ has exact order
$4$, then $p-x$ cannot be of the form $[v - w_0]$ for
$v \in \calC$.  Note that in fact $p-x$ also has exact order~$4$.
Suppose to the contrary that $[v - w_0]$ were a point of 
exact order~$4$.  Then $4v - 4w_0$ would be the divisor of an element
$\phi$ of the function field $\Kbar(\calC)$.  Let us take a model for
$\calC$ over $\Kbar$ given by an equation of the form $Y^2 =
(X-a_1)\cdots(X-a_5)$, where $w_0$ is the point at infinity, and where
for $i \geq 1$, $w_i = (a_i, 0)$ for distinct $a_1, \dots, a_5 \in
\Kbar$.  Then $X$ has a double pole at $w_0$, and $Y$ has a quintuple
pole at $w_0$.  Now the only singularity of $\phi$ is a quadruple pole
at $w_0$, so $\phi = c_2 X^2 + c_1 X + c_0$ with $c_2 \neq 0$; we can
harmlessly assume $c_2 = 1$.  Factor $\phi = (X-b_1)(X-b_2)$, with
$b_1, b_2 \in \Kbar$.  If $b_1 \neq b_2$,
then $\phi$ cannot have a quadruple zero at just one point
$v \in \calC$; thus $\phi = (X-b_1)^2$.  But then $v$ has coordinates
$(b_1,c)$ for some $c \in \Kbar$; if we had $c \neq 0$, then $\phi$
would also vanish at a second point $(b_1, -c) \in \calC$.  Thus the
only possibility is to have $c=0$, and $b_1 \in \{a_1, \dots, a_5\}$.
But then $v = (b_1,0)$ would be one of the Weierstrass points, and
$[v-w_0]$ would be a point of order~$2$, not of exact order~$4$.
\end{proof}

The first part of the following corollary is basically a precise
statement in our setting of the difference between even and odd
theta-characteristics for the curve~$\calC$.

\begin{corollary}
  \label{corFdalphaNonvanishing}
  Let $d \in B_4$ and let $\alpha \in A_4$.
  \begin{enumerate}
  \item
    The value $F_{d,\alpha}(0)$ is zero when $e_4(2d,\alpha) = -1$,
    and is nonzero otherwise (when $e_4(2d,\alpha) = 1$).
  \item
    Let $q \in \AV$ be a point of exact order~$8$.  Then
    $F_{d,\alpha}(q) \neq 0$. 
  \end{enumerate}
\end{corollary}
\begin{proof}
  The second assertion follows directly from~\eqref{eqFdalphaFormula}
  and the second part of Lemma~\ref{lemmaTheta1Nonvanishing}.
  Let us prove the first assertion.  Observe first that $e_4(2d,\alpha) =
  e_2(2d,2\alpha)$, with $2\alpha \in A_2$ and $2d \in B_2$.
  Moreover, there are six pairs $(b,a) \in B_2 \times A_2$ for which
  $e_2(b,a) = -1$; namely, to each of the three nonzero choices of
  $b \in B_2$, there correspond precisely two choices of $a \in A_2$
  that give a nontrivial Weil pairing with~$b$.  Now put
  $q=0$ in~\eqref{eqBasicAdditionFormula} to obtain
  \begin{equation}
    \label{eqMultiplicationL2}
      h_{p,0}(\sum_{c \in B_2}
     e_4(c, \alpha)
     \theta_{[2],2d+c}(p) \tensor \theta_{[2],c}(p)
     )
   =
     F_{d,\alpha}(p) \tensor F_{d,\alpha}(0).
  \end{equation}
  The expression $S$ within parentheses on the left hand side is a
  linear   combination of products of two sections of 
  $\LL^2$.  Multiplication of two sections of the same line bundle is
  commutative, as is familiar over $\C$
  (for example, set $w=0$ in~\eqref{eqnThetaAddition}).
  Here is a pedantic proof of
  commutativity in our algebraic context: the fiber $V = \LL_p^2$ is
  a one-dimensional $\Kbar$-vector space, so for $v,v' \in V$ we have
  $v \tensor v' = v' \tensor v$ in the one-dimensional space
  $V \tensor V = \LL_p^4$.  Using the fact that $4d = 0$, we deduce
  that 
  \begin{equation}
    \label{eqRearrangeMult}
    \begin{split}
      S &:= \sum_{c \in B_2} e_4(c, \alpha)
      \theta_{[2],2d+c}(p) \tensor \theta_{[2],c}(p)\\
      &= \sum_{c \in B_2} e_4(c, \alpha)
      \theta_{[2],c}(p) \tensor \theta_{[2],2d+c}(p)\\
      &= \sum_{c' \in B_2} e_4(2d+c', \alpha)
      \theta_{[2],2d+c'}(p) \tensor \theta_{[2],c'}(p)\\
      &= e_4(2d, \alpha) S.\\
    \end{split}
  \end{equation}
  Thus when $e_4(2d,\alpha) = -1$, we must have $S = 0$ for all $p$,
  so $F_{d,\alpha}(p) \tensor F_{d,\alpha}(0) = 0$.
  Since $F_{d,\alpha}$ is a nonzero element of $H^0(\AV,\LL^4)$, we
  deduce that $F_{d,\alpha}(0) = 0$, and hence that
  $\theta_{[1],0}(2d-2\alpha) = 0$.  Now $2d - 2\alpha$ is a
  $2$-torsion point (and is equal to $2d + 2\alpha$), and we have just
  listed six $2$-torsion points where $\theta_{[1],0}$ vanishes.
  (These correspond to the odd theta-characteristics.)  By
  our lemma, these are the only such $2$-torsion points; hence, when
  $e_4(2d,\alpha) = 1$ (an even theta-characteristic), we must have
  $\theta_{[1],0}(2d-2\alpha) \neq 0$, hence $F_{d,\alpha}(0) \neq 0$.
\end{proof}

We remark incidentally that when $\AV$ is not a Jacobian, but rather a
principally polarized abelian surface that is the product of two
elliptic curves, then $\theta_{[1],0}$ vanishes at one of the even
theta-characteristics, so vanishes at precisely seven $2$-torsion
points.  Over $\C$, we can see this by taking
$\Omega = \twomatr{\tau_1}{0}{0}{\tau_2}$, in which case the analytic
$\theta_{1,0}$ vanishes both at the six odd theta-characteristics and
at one even theta-characteristic, corresponding to
$2\alpha = \frac{1}{2}\binom{1}{1}$ and
$2d = \frac{1}{2}\Omega\binom{1}{1} = \frac{1}{2}\binom{\tau_1}{\tau_2}$.

\subsection*{The embedding into $\Projective^3 \times \Projective^3$}
We first describe the Kummer map $\kappa: \AV \to \Projective^3$ that
is associated to $\LL^2$.  In our isotropic subgroup $A_4 \subset \AV[4]$,
we fix a point $D_1$ of exact order~$4$.  Let $E_1 = 2D_1 \in A_2
\isomorphic (\Z/2\Z)^2$, let $E_2 \in A_2$ be a second $2$-torsion
point, and define $E_3 = E_1 + E_2$; hence
$A_2 = \{ 0, E_1, E_2, E_3\}$.  (In Section~\ref{section4}, we
will also write $E_0 = 0$, and take a point $D_2$ with $2D_2 = E_2$,
then write $D_3 = D_1 + D_2$; however, the $4$-torsion points $D_2$ and
$D_3$ will not belong to $A_4$.)  We also assign 
names to the points of $B_2$ as
$B_2 = \{0=b_{00}, b_{10}, b_{01}, b_{11} \}$,
according to their Weil pairing with $E_1$ and $E_2$:
\begin{equation}
  \label{eqDefbij}
  e_2(b_{ij},E_1) = (-1)^i, \qquad\qquad e_2(b_{ij},E_2) = (-1)^j.
\end{equation}
These names appear briefly here in~\eqref{eqDefQnew}, and
they will make a later appearance in Lemmas \ref{lemmaCommon2d2alpha}
and~\ref{lemmaQzeroNonvanishing}.
Using the above notation, we rename our basis
$\{ \theta_{[2],b} \}_{b \in B_2}$ for $H^0(\AV,\LL^2)$ as
\begin{equation}
\label{eqDefQnew}
Q_{00} = \theta_{[2],b_{00}}, \qquad
Q_{10} = \theta_{[2],b_{10}}, \qquad
Q_{01} = \theta_{[2],b_{01}}, \qquad
Q_{11} = \theta_{[2],b_{11}}.
\end{equation}
Hence the componentwise actions of the lifts
$\widetilde{E_1}, \widetilde{E_2}, \widetilde{E_3} \in \calG_2$ on the
$4$-tuple of sections $(Q_{00},Q_{10},Q_{01},Q_{11})$ are given by
\begin{equation}
  \label{eqEtildeAction}
  \begin{split}
    \widetilde{E_1}*(Q_{00},Q_{10},Q_{01},Q_{11})
    & = (Q_{00},-Q_{10},Q_{01},-Q_{11}), \\
    \widetilde{E_2}*(Q_{00},Q_{10},Q_{01},Q_{11})
    & = (Q_{00},Q_{10},-Q_{01},-Q_{11}),\\
    \widetilde{E_3}*(Q_{00},Q_{10},Q_{01},Q_{11})
    & = (Q_{00},-Q_{10},-Q_{01},Q_{11}).\\
  \end{split}
\end{equation}
The Kummer map is then
\begin{equation}
  \label{eqDefKappa}
  \kappa(p) = [Q_{00}(p), Q_{10}(p), Q_{01}(p), Q_{11}(p)] \in \Projective^3,
  \qquad\qquad p \in \AV.
\end{equation}
Note carefully the notation within square brackets for projective coordinates.
Here the values $Q_j(p)$ for $j \in \{00,01,10,11\}$ all belong to the
same one-dimensional space $\LL_p^2$, so strictly speaking the
projective coordinates that we wrote down are not elements of $\Kbar$.
However, their ``ratios'' give a well defined point in
$\Projective^3$, as usual.  (This uses the fact that $\LL^2$ is base point
free, so the values $\{Q_j(p)\}_j$ can never all vanish for the same $p$.)
It is standard that $\kappa(-p) = \kappa(p)$, and that the image of
$\kappa$ is in bijection with (the geometric points of)
$\AV/\{[\pm 1]\}$.
References for these facts are Chapter~3 of~\cite{CasselsFlynnBook},
Section~4.8 of~\cite{BirkenhakeLange}, and Section~10.4 of~\cite{KempfBook}.
The image of~$\kappa$ is called
the Kummer surface, and is the zero set of a
quartic equation in the $Q_j$, called the Kummer quartic. 
Our identities~\eqref{eqEtildeAction} imply the following identities
between projective points:
\begin{equation}
  \label{eqEprojectiveOnKappa}
  \begin{split}
    \kappa(p-E_1)
    &= [Q_{00}(p-E_1),Q_{10}(p-E_1),Q_{01}(p-E_1),Q_{11}(p-E_1)]\\
    &= [Q_{00}(p),-Q_{10}(p),Q_{01}(p),-Q_{11}(p)], \\
    \kappa(p-E_2)
    &= [Q_{00}(p),Q_{10}(p),-Q_{01}(p),-Q_{11}(p)], \\
    \kappa(p-E_3)
    &= [Q_{00}(p),-Q_{10}(p),-Q_{01}(p),Q_{11}(p)]. \\
  \end{split}
\end{equation}
This is because applying the same isomorphism $\phi_{p-E_i}$
(associated to $\widetilde{E_i}$) to all four coordinates does not
change the projective point.  Note of course also that $-E_i = E_i$.


We now translate everything by $D_1$.  Write
$\LL' = T_{D_1}^* \LL$ for the shifted line bundle.  Then
a basis $\{P_j\}_j$ for $H^0(\AV,(\LL')^2)$ is given by
$P_j = T_{D_1}^* Q_j$.  Concretely, for $j \in \{00,10,01,11\}$,
we have
\begin{equation}
\label{eqDefPnew}
P_j(p) = Q_j(p + D_1) \in  (\LL'_p)^2 = \LL_{p+D_1}^2.
\qquad\qquad p \in \AV.
\end{equation}
The projective map $\kappa'$ associated to $(\LL')^2$ sends $p$
to $\kappa'(p) = \kappa(p+D_1)$.  The image of $\kappa'$ is given by
the same Kummer quartic equation, this time in the $\{P_j\}$.
Combining $\kappa$ and $\kappa'$ gives our fundamental embedding
of~$\AV$ into $\Projective^3 \times \Projective^3$:
\begin{equation}
\label{eqEmbeddingP3xP3new}
p \mapsto (\kappa(p), \kappa'(p)) = ([P_i(p)]_i, [Q_j(p)]_j) \in
\Projective^3 \times \Projective^3. 
\end{equation}

For comparison with Section~\ref{section4}, where we pay attention to
rationality over a base field~$K$, we note that the $K$-rational
coordinates $([u_i]_i, [y_j]_j)$ on $\Projective^3 \times \Projective^3$
that appear in that section transform in a ``diagonalized'' way
under the $K$-rational isotropic $2$-torsion subgroup
$A_2 = \{E_0, E_1, E_2, E_3\}$, similarly to~\eqref{eqEtildeAction}
but with a slightly different order of coordinates:
see~\eqref{mapsonembedding}.  Each space where $A_2$
acts by a given character is one-dimensional; hence each $u_i$
from Section~\ref{section4} is then a different nonzero constant multiple
of a corresponding $Q_{i'}$, and similarly each $y_j$ is a multiple of
a corresponding $P_{j'}$.  These nonzero constants belong to~$\Kbar$,
but not necessarily to $K$ itself.  Moreover, the nonzero elements
$b \in B_2$ 
are in general not $K$-rational.  The lifts~$\tilde{b}$ of these  
elements to $\calG_2$ permute the $Q$ (and $P$) coordinates, but their
action on the $u_i$ (and $y_j$) combines a permutation with a
rescaling of each coordinate by a different factor in $\Kbarstar$;
essentially, the action is by an element of the normalizer of the
diagonal algebraic torus in $GL(4,\Kbar)$.

When $\Kbar = \C$, we can describe the $Q$ and $P$ coordinates
analytically by
\begin{equation}
  \label{eqDefPQ}
  \begin{split}
Q_{00} &= \theta_{2,(0,0)}(z), \qquad
Q_{10} = \theta_{2,(1,0)}(z), \\
Q_{01} &= \theta_{2,(0,1)}(z), \qquad
Q_{11} = \theta_{2,(1,1)}(z), \\
& \qquad P_j(z) = Q_j(z + (1/4,0)),\\
  \end{split}
\end{equation}
where we write vectors as rows instead of columns for
typographical convenience.

\subsection*{Relations between the $P$ and $Q$ coordinates, and
dimensions of each homogeneous component of the bigraded ideal}
Our goal is to study the equations of $\AV$ under the
embedding~\eqref{eqEmbeddingP3xP3new}, as well as formulas for the group
law.  The natural setting for all this is the bigraded polynomial
ring in eight variables
\begin{equation}
\label{eqnDefR}
R = \Kbar[\{P_i\}, \{Q_j\}] = \directsum_{d,e \geq 0} R_{d,e},
\end{equation}
where here the $P_i$ and $Q_j$ are the coordinates on $\Projective^3
\times \Projective^3$ (in particular, they are algebraically
independent), and $R_{d,e}$ is the summand consisting of bihomogeneous
polynomials $f(\{P_i\},\{Q_j\})$ of bidegree $(d,e)$ in the $P$s and $Q$s.
When we view the $P$s and $Q$s instead as sections of the line bundles
$(\LL')^2$ and $\LL^2$, and take products of
such sections, we can map a monomial such as
(to take a random example)
$P_{10}^3 Q_{00}Q_{01} \in R_{3,2}$ into
a space such as
$H^0(\AV, (\LL')^6 \tensor \LL^4)$. Let us write
\begin{equation}
\label{eqnDefMdeVde}
\MM_{d,e} = (\LL')^{2d}\tensor \LL^{2e},
\qquad
V_{d,e} = H^0(\AV, \MM_{d,e}).
\end{equation}
(Note for later use that $\MM_{d,e}$ is algebraically equivalent to
$\LL^{2(d+e)}$.)
The multiplication map that we just defined on monomials
extends to a $\Kbar$-linear map
\begin{equation}
\label{eqMultToVde}
\mu = \mu_{d,e}: R_{d,e} \to V_{d,e}
    \qquad f \mapsto \overline{f}.
\end{equation}
We will generally distinguish $f$ from its image $\mu(f) =
\overline{f}$, but we reserve the right to be occasionally sloppy,
especially when $\mu_{d,e}$ is injective.

The ideal $I$ defining the image of $\AV$
under~\eqref{eqEmbeddingP3xP3new} is a bigraded ideal of $R$, with
\begin{equation}
\label{eqIde}
I_{d,e} = \ker \mu_{d,e}.
\end{equation}
We will denote by $\Rbar$ the bigraded quotient ring $R/I$,
so $\Rbar_{d,e} \isomorphic \image \mu_{d,e} \subset V_{d,e}$.  We
view $\Rbar_{d,e}$ as a subspace of $V_{d,e}$, and point out
that for some $(d,e)$, it is a proper subset; see
Remark~\ref{remarkSmallDimension}.

We now compute, in stages, the dimensions of $I_{d,e}$ and of
$\Rbar_{d,e}$.  First note the following basic dimension counts.

\begin{proposition}
\label{propBasicDims}
For $d,e \geq 0$, we have
$\dim R_{d,e} = \binom{d+3}{3}\binom{e+3}{3}$ and
$\dim V_{d,e} = 4(d+e)^2$.
\end{proposition}
\begin{proof}
The first statement holds because there are $\binom{d+3}{3}$
monomials of degree~$d$ in the four variables $\{P_i\}$, and
$\binom{e+3}{3}$ monomials in the $\{Q_j\}$.  The second
statement amounts to Riemann-Roch for the abelian surface $\AV$,
because $\MM_{d,e}$ is algebraically equivalent to the $2(d+e)$th
power of the principal polarization bundle $\LL$.   
\end{proof}

When one of $d$ or $e$ is zero, then the dimensions are easy to
compute, as a consequence of the known structure of the Kummer
embedding.

\begin{proposition}
\label{Ik0I0k}
For $k \leq 3$, we have $I_{k,0} = 0$; for $k \geq 4$, we have
$\dim I_{k,0} = \dim R_{k-4,0}$.  Hence
\begin{equation}
\label{eqDimIk0}
\dim \Rbar_{k,0} = \begin{cases}
  \dim R_{k,0} = \binom{k+3}{3}, &\text{if $k \leq 3$,}\\
  \dim R_{k,0} - \dim R_{k-4,0} = 2k^2 + 2, &\text{if $k \geq 4$.}\\
\end{cases}
\end{equation}
An analogous result holds for $I_{0,k}$ and $\Rbar_{0,k}$.  Note in
fact that the above formulas imply that
$\dim \Rbar_{k,0} = \dim \Rbar_{0,k} = 2k^2 + 2$ for all $k \geq 1$. 
\end{proposition}
\begin{proof}
Elements of $I_{k,0}$ are the same as degree~$k$ relations between the
$\{ P_i\}$ that do not involve any of the $\{Q_j\}$.  All relations in
the $\{P_i\}$ alone are generated by the quartic equation
$r = r(P_{00},P_{10},P_{01},P_{11})$ that defines the Kummer surface.
Hence, for $k \geq 4$, we have $I_{k,0} = r R_{k-4,0}$.  The ring
$R$ is a domain, so $I_{k,0}$ has the same dimension as $R_{k-4,0}$.
Finally, when $k \in \{1,2,3\}$, then one easily checks that
$\binom{k+3}{3} = 2k^2 + 2$.
\end{proof}

\begin{remark}
\label{remarkSmallDimension}
In particular, $\Rbar_{k,0} \neq V_{k,0}$ for $k \geq 2$; for example,
$\dim \Rbar_{2,0} = 10 < 16 = \dim V_{2,0}$. 
\end{remark}

Our next goal is to show that, in fact, $\Rbar_{d,e} = V_{d,e}$
whenever $d,e \geq 1$.  We begin with the case of $\Rbar_{1,1}$.

\begin{proposition}
\label{Rbar11}
In the setting of our construction, with $\AV$ the Jacobian of a curve
$\calC$, the space $\Rbar_{1,1}$ is equal to all of $V_{1,1}$; equivalently,
$I_{1,1} = 0$, and the set of sixteen products $P_i Q_j$ is linearly
independent and hence a basis of $\Rbar_{1,1} = V_{1,1}$.
\end{proposition}
\begin{proof}
We are equivalently asserting the surjectivity of the multiplication
map
$H^0(\AV,(\LL')^2) \tensor H^0(\AV, \LL^2) \to H^0(\AV, \MM_{1,1})$.
This surjectivity (hence bijectivity, since both spaces are
$16$-dimensional) follows from
part~2 of Lemma~\ref{lemmaTheta1Nonvanishing} here, combined with
Theorem~3 in the appendix
of~\cite{KempfMultiplication}.  Note that the statement there has a
typographical mistake; 
see the statement of Theorem~2.1 of~\cite{PareschiSalvatiManni} and
the comments just preceding their Remark~2.4.  We reproduce the argument
from~\cite{KempfMultiplication} in our setting.  Take a point
$q \in \AV$ such that $2q = D_1$; hence $q$ has exact order~$8$, and
$\MM_{1,1} = (\LL')^2 \tensor \LL^2 = (T_{2q}^* \LL^2) \tensor \LL^2$
is isomorphic to $(T_q^* \LL^2) \tensor (T_q^* \LL^2)$, hence to
$T_q^* \LL^4$.  As mentioned just
after~\eqref{eqDefOfF}, the $\{F_{d,\alpha}\}$ form a basis of
$H^0(\AV,\LL^4)$ when $d$ and $\alpha$ range over coset
representatives for $B_4/B_2$ and $A_4/A_2$; thus the
$\{T_q^* F_{d,\alpha}\}$,
which are a basis for $H^0(\AV, T_q^* \LL^4)$,
can be identified with a basis for $H^0(\AV,
\MM_{1,1})$.

Now replace $p$
in~\eqref{eqBasicAdditionFormula} by $p' = p+q$.
This yields, for each choice of $d$ and 
$\alpha$, a linear combination of values
$P_i(p) \tensor Q_j(p)$, which is equal (up to the isomorphism
$h_{p+q,q}$) to $F_{d,\alpha}(p+q) \tensor F_{d,\alpha}(q)$.  We can
view $F_{d,\alpha}(p+q)$ as the value at $p$ of the section
$T_q^* F_{d,\alpha} \in H^0(\AV, T_q^* \LL^4)$.  
Moreover, at least up to isomorphism, we can
view $F_{d,\alpha}(q)$, which belongs to the fixed one-dimensional
space $\LL_q^4$, as a nonzero element of $\Kbar$, by the second
assertion of Corollary~\ref{corFdalphaNonvanishing}.
This means that the image of our multiplication map contains a nonzero
multiple of each basis element of $H^0(\AV,\MM_{1,1})$, and is hence
surjective.  
\end{proof}

We now show that all the ``larger'' $\Rbar_{d,e}$ are equal to
their corresponding $(d,e)$.  We introduce the notation
\begin{equation}
\label{eqComponentwiseOrder}
(d,e) \geq (d',e') \quad :\iff \quad d \geq d' \text{ and } e \geq e'.
\end{equation}

\begin{proposition}
\label{RbarIgeq1geq1}
Let $(d,e) \geq (1,1)$.  Then $\Rbar_{d,e} = V_{d,e}$; that is, the
multiplication map $\mu_{d,e}$ is surjective.
\end{proposition}
\begin{proof}
This is an application of a result of Kempf that first appeared
in~\cite{KempfProjectiveCoordRings}, but this reference is less
widely
available than others.  The proof of Kempf is reproduced
in Theorem~10.1 of~\cite{MumfordTataIII}, in language that works over
an arbitrary~$\Kbar$.  (Over $\C$, one can also see  
Theorem~6.8(c) of~\cite{KempfBook}, which is easily adaptable to
general $\Kbar$, and Proposition~7.3.4 
of~\cite{BirkenhakeLange}).  The result of Kempf is as follows:
let $\LL_1, \LL_2$ be ample
line bundles on $\AV$ with $\LL_i$ algebraically equivalent to
$\LL^{\ell_i}$.
Assume that 
$(\ell_1, \ell_2) \geq (2,3)$ or that $(\ell_1, \ell_2)
\geq (3,2)$, in the sense of~\eqref{eqComponentwiseOrder}.
The theorem then states that the map $H^0(\AV,\LL_1) \tensor H^0(\AV,\LL_2)
\to H^0(\AV, \LL_1 \tensor \LL_2)$ is surjective.  Applying this to
our $\MM_{d,e}$ and either $\MM_{1,0}$ or $\MM_{0,1}$, this means
that if $d+e \geq 2$ then the ``product'' $V_{d,e} \cdot V_{1,0}$ is
equal to $V_{d+1,e}$, and similarly $V_{d,e} \cdot V_{0,1} =
V_{d,e+1}$.  By a notation such as $V_{d,e} \cdot V_{1,0}$, we mean
the set of all finite sums $\sum_i a_i \cdot b_i$ with each
$a_i \in V_{d,e}$ and $b_i \in V_{1,0}$.  In other words,
$V_{d,e} \cdot V_{1,0}$ is the image of
the multiplication map $V_{d,e} \tensor V_{1,0} \to V_{d+1,e}$.

We now proceed by induction, from the base case $(d,e) =
(1,1)$, to obtain our result.  (Surjectivity of the multiplication
maps in $R$ or $\Rbar$ is immediate: for example, $R_{d,e}\cdot R_{1,0} =
R_{d+1,e}$.)
\end{proof}

\begin{corollary}
\label{dimensionsIde}
For $(d,e) \geq (1,1)$, we have
$\dim I_{d,e} = \dim R_{d,e} - \dim V_{d,e} =
\binom{d+3}{3}\binom{e+3}{3} - 4(d+e)^2$.
As a special case, and after some simple algebra,
$\dim I_{k,1} = \dim I_{1,k} = 4\binom{k+1}{3}
   = 4 \dim R_{k-2,0} = 4 \dim R_{0, k-2}$
when $k \geq 1$.
\end{corollary}

\subsection*{The result of Kempf on relations; consequences for
ideal generators}
\begin{definition}
\label{defOfR}
Let $\LL_1$ and $\LL_2$ be ample line bundles on $\AV$.  The space of
relations $\RR(\LL_1,\LL_2)$ is defined by
\begin{equation}
  \label{eqnForR}
  \RR(\LL_1,\LL_2) = \ker( H^0(\AV,\LL_1) \tensor H^0(\AV,\LL_2)
      \to H^0(\AV, \LL_1 \tensor \LL_2)).
\end{equation}
\end{definition}
The space $\RR(\MM_{d',e'},\MM_{d'',e''})$ provides information
about the bigraded component $I_{d'+d'', e'+e''}$ of our ideal of relations.

The following theorem is again due to
Kempf~\cite{KempfProjectiveCoordRings}.
See Theorem~6.14 and Proposition~6.15
of~\cite{KempfBook}; or, alternatively, Theorem~10.10
of~\cite{MumfordTataIII}; or also Proposition~7.4.3
of~\cite{BirkenhakeLange} and the subsequent method of proof of their
Theorem~7.4.1.  All these references follow the original argument of 
Kempf.  

\begin{theorem}
\label{thmKempfRelations}
Let $\LL_1$, $\LL_2$, and $\LL_3$ be ample line bundles on $\AV$, with
$\LL_i$ algebraically equivalent to $\LL^{\ell_i}$.  Assume that
$\ell_3 \geq 2$ and that either $(\ell_1,\ell_2) \geq (2,5)$ or
$(\ell_1,\ell_2) \geq (3,4)$.  Then the following map of vector
spaces is surjective:
\begin{equation}
  \label{eqnKempfRelations}
\RR(\LL_1,\LL_2) \tensor H^0(\AV,\LL_3) \to \RR(\LL_1, \LL_2 \tensor \LL_3).
\end{equation}
\end{theorem}


Let us describe the map in~\eqref{eqnKempfRelations} explicitly.
Consider a tensor $\sum_{i} s_i \tensor t_i \in
\RR(\LL_1, \LL_2)$.  This means that for each $i$, we have
$s_i \in  H^0(\AV,\LL_1)$ and $t_i \in H^0(\AV, \LL_2)$, and that furthermore
$\sum_{i} s_i \cdot t_i = 0$ in $H^0(\AV,\LL_1 \tensor \LL_2)$.  Let
$u \in H^0(\AV, \LL_3)$.  Then
$(\sum_{i} s_i \tensor t_i) \tensor u$ is mapped to the element
$\sum_{i} s_i \tensor (t_i \cdot u) 
   \in \RR(\LL_1, \LL_2 \tensor \LL_3)$,
where each $t_i \cdot u$ belongs to $H^0(\AV,\LL_2 \tensor \LL_3)$.

We will apply Theorem~\ref{thmKempfRelations} twice.
Proposition~\ref{propositionI31} below is not strictly speaking
necessary: we will later use a different method to show in
Proposition~\ref{propositionI21} that the result 
holds with the weaker assumption that $k\geq 2$.  However, the
notation here is simpler than in the proof of
Proposition~\ref{propositionI22} below, so this first proof is easier
to digest, and can serve as a guide to the proof of
Proposition~\ref{propositionI22}.  The notation 
$R_{1,0} \cdot I_{k,1}$ was defined in the proof of
Proposition~\ref{RbarIgeq1geq1}. 

\begin{proposition}
\label{propositionI31}
Let $k \geq 3$.  Then $R_{1,0} \cdot I_{k,1} = I_{k+1,1}$ and
$R_{0,1} \cdot I_{1,k} = I_{1,k+1}$.
\end{proposition}
\begin{proof}
The two statements are symmetric, so we prove only the first
assertion.  The key step will be to apply
Theorem~\ref{thmKempfRelations} with $\LL_1 = \MM_{1,0}$, $\LL_2 =
\MM_{k-1,1}$, and $\LL_3 = \MM_{1,0}$.  Hence $\ell_1 = 2$, $\ell_2 =
2k \geq 6$, and $\ell_3 = 2$, and we will invoke the surjectivity
of~\eqref{eqnKempfRelations} at an opportune moment.

Consider an element $f \in I_{k+1,1}$; our goal is to express $f$
in terms of elements of $I_{k,1}$.  Writing $f$ in terms of the
coordinates $\{P_i\}$ and $\{Q_j\}$ (recall these are respectively
bases of $R_{1,0}$ and $R_{0,1}$), we can (nonuniquely) write $f$
in the form $f = \sum_i P_i g_i$ with $g_i \in R_{k,1}$.  Let
$\overline{g}_i \in \Rbar_{k,1}$ be the image of $g_i$ when we map it
to $\Rbar_{k,1} = H^0(\AV,\MM_{k,1})$.  Since $f \in I_{k+1,1}$, we
deduce that $\sum_i P_i \tensor \overline{g}_i 
  \in \RR(\MM_{1,0},\MM_{k,1})$.

At this point, invoke the surjectivity of~\eqref{eqnKempfRelations} to
obtain a preimage of $\sum_i P_i \tensor \overline{g}_i$.  This
preimage has the form
$\sum_j r_j \tensor P_j \in \RR(\MM_{1,0},\MM_{k-1,1}) \tensor R_{1,0}$.
Each $r_j \in \RR(\MM_{1,0},\MM_{k-1,1})$ can be written as
$r_j = \sum_i P_i \tensor \overline{A}_{ij}$ with $\overline{A}_{ij} \in
\Rbar_{k-1,1}$; let $A_{ij} \in R_{k-1,1}$ be a representative of
$\overline{A}_{ij}$.  The fact that $r_j \in
\RR(\MM_{1,0},\MM_{k-1,1})$ means that it maps to zero in
$\Rbar_{k,1}$, so for each $j$ we have 
\begin{equation}
\label{eqn1InPfOfPropI31}
\sum_i P_i A_{ij} \in
I_{k,1}.
\end{equation}

We thus have a preimage 
$\sum_j (\sum_i P_i \tensor \overline{A}_{ij}) \tensor P_j$ that maps
to
$\sum_{i,j} P_i \tensor \overline{A_{ij} P_j} = \sum_i P_i \tensor
\overline{g}_i$.
This last equality takes place inside $\RR(\MM_{1,0},\MM_{k,1})$,
which is a subspace of $R_{1,0}\tensor\Rbar_{k,1}$.
Now the $\{P_i\}$ are a basis for $R_{1,0}$, so for each
$i$, $\sum_j \overline{A_{ij} P_j} = \overline{g}_i$
inside $\Rbar_{k,1}$; in other words,
$g_i \equiv \sum_j A_{ij} P_j \pmod{I_{k,1}}$ for all $i$.  Hence we obtain
\begin{equation}
\label{eqn2InPfOfPropI31}
f = \sum_i P_i g_i 
  \equiv \sum_{i,j} P_i A_{ij} P_j \pmod{R_{1,0} \cdot I_{k,1}}
  \equiv 0 \pmod {I_{k,1} \cdot R_{1,0}},
\end{equation}
where the last congruence holds by~\eqref{eqn1InPfOfPropI31}.  Hence
$f \in R_{1,0} \cdot I_{k,1}$, as desired.
\end{proof}

\begin{proposition}
\label{propositionI22}
Let $d',e' \geq 2$.  Then $R_{1,0} \cdot I_{d',e'} = I_{d'+1,e'}$, and we
similarly have $R_{0,1} \cdot I_{d',e'} = I_{d',e'+1}$.
\end{proposition}
\begin{proof}
As before, we prove only the first assertion.  Write $d' = d+1$ and
$e' = e+1$ with $d,e \geq 1$.  We will apply
Theorem~\ref{thmKempfRelations} with $\LL_1 = \MM_{1,1}$, $\LL_2 =
\MM_{d,e}$, and $\LL_3 = \MM_{1,0}$.   This time we use the basis
$\{P_i Q_j\}_{i,j}$ for the sixteen-dimensional space
$R_{1,1} = \Rbar_{1,1} = H^0(\AV,\MM_{1,1})$.  The fact that $d,e \geq
1$ also ensures that $\Rbar_{d,e} = H^0(\AV,\MM_{d,e})$.

Consider an element $f \in I_{d'+1,e'} = I_{d+2,e+1}$.  Similarly to
the proof of Proposition~\ref{propositionI31}, write $f = \sum_{i,j}
P_i Q_j g_{ij}$, with $g_{ij} \in R_{d+1,e}$.  Once again, 
$\sum_{i,j} P_i Q_j \tensor \overline{g}_{ij} 
\in \RR(\MM_{1,1},\MM_{d+1,e})$, and this comes from a tensor
$\sum_{i,j,k} (P_i Q_j \tensor \overline{A}_{ijk}) \tensor P_k$ with
each $\sum_{i,j} P_i Q_j \tensor \overline{A}_{ijk} \in
\RR(\MM_{1,1},\MM_{d,e})$, for all $k$.  By the same reasoning as
before (that is, the fact that we have a basis for $\Rbar_{1,1}$), we
deduce that for all $i,j$, $g_{ij} \equiv \sum_k A_{ijk} P_k
\pmod{I_{d+1,e}}$.  Hence $f \equiv \sum_{i,j,k} P_i Q_j A_{ijk} P_k
\pmod{R_{1,1} \cdot I_{d+1,e}}$; note here that $R_{1,1} \cdot
I_{d+1,e} \subset R_{1,0} \cdot I_{d+1,e+1}$.

Finally,  
$\sum_{i,j,k} P_i Q_j A_{ijk} P_k 
= \sum_k (\sum_{i,j} P_i Q_j A_{ijk}) P_k
\equiv 0 \pmod{I_{d+1,e+1} \cdot R_{1,0}}$
and so we have shown the desired result, that
 $f \in R_{1,0} \cdot I_{d+1,e+1} = R_{1,0} \cdot I_{d',e'}$.
\end{proof}

\subsection*{Constructing specific elements of $I_{1,2}$ and $I_{2,1}$}

Our next goal is to study more carefully certain relations
between the $\{P_i\}$ and the $\{Q_j\}$, and to deduce from
these a specific basis for each of $I_{1,2}$ and $I_{2,1}$.  These
bases will allow us to prove Proposition~\ref{propositionI21} below,
which strengthens Proposition~\ref{propositionI31} and allows us to
reach the principal result of this section,
Theorem~\ref{generatorsInEachBidegree}, giving our structural results
on the bidegrees that are enough to generate the ideal $I$.

In the setting of algebraic theta functions,
we give a construction that is closely related to the result in
\eqref{genus2biquadD1}, \eqref{genus2biquadD1var},
and~\eqref{someeqns}.
We actually discovered those relations in Section~\ref{section4} earlier
in our investigations, based on heuristics on where to expect them,
but we need to first carry out a similar computation here, so as to
make Theorem~\ref{generatorsInEachBidegree} available to us at the
correct moment when we use it later.

We first observe that~\eqref{eqMultiplicationL2} identifies
certain quadratic expressions in the $\{Q_j\}$ in terms of the
sections $F_{d,\alpha} \in H^0(\AV,\LL^4)$.  Here, we can view the
quadratic expressions in the $\{Q_j\}$ as elements of
the $10$-dimensional space $R_{0,2}$; since $I_{0,2} = 0$, we can
identify $R_{0,2}$ both with $\Rbar_{0,2}$ and with its
$10$-dimensional image inside $H^0(\AV,\LL^4)$.  It follows from the
proof of Corollary~\ref{corFdalphaNonvanishing} that this image is
spanned by those $F_{d,\alpha}$ for which $F_{d,\alpha}(0) \neq 0$,
equivalently, for which $e_4(2d,\alpha) = 1$.

We wish to obtain a similar identification for the quadratic
expressions in the $\{P_j\}$, in other words for the elements of
$R_{2,0}$.  Since the $\{P_j\}$ are translations of 
the $\{Q_j\}$ by $D_1$, this means that we need to translate the
$F_{d,\alpha}$.  This amounts to the action of
$\widetilde{(-D_1)} = (-D_1, \phi) \in \tilde{A}_4 \subset \calG_4$.
According to~\eqref{eqThetaGroupActionOnAlgebraicThetas}
and~\eqref{eqThetaGroupAction}, we have
\begin{equation}
  \label{eqD1actionFdalpha}
  \begin{split}
  &\widetilde{(-D_1)} * F_{d,\alpha}
  = \sum_{c \in B_2} e_4(c,\alpha) e_4(d+c,-D_1) \theta_{[4],d+c}
  = e_4(d,-D_1) F_{d,\alpha-D_1},\\
  &\phi_{p+D_1}( F_{d,\alpha}(p+D_1) )
     = e_4(d, -D_1) F_{d,\alpha-D_1}(p).\\
  \end{split}
\end{equation}
Replace $p$ by $p+D_1$ in~\eqref{eqMultiplicationL2}, and write
$\hat{h}_{0,p}
  = (\phi_{p+D_1} \tensor \text{id}_{\LL_0^4})\compose h_{p+D_1,0}$
to obtain
\begin{equation}
  \label{eqMultiplicationLprime2}
  \begin{split}
    \hat{h}_{p,0}
      (\sum_{c \in B_2}
    &    e_4(c, \alpha) \theta_{[2],2d+c}(p+D_1)
             \tensor \theta_{[2],c}(p+D_1)
      )\\
    & = e_4(d, -D_1) F_{d,\alpha-D_1}(p) \tensor F_{d,\alpha}(0).\\
  \end{split}
\end{equation}
Thus the image of $R_{2,0}$ in $H^0(\AV, \LL^4)$ is spanned by those
$F_{d,\alpha-D_1}$  for which $F_{d,\alpha}(0) \neq 0$.  In other
words, the image is spanned by those $F_{d,\alpha}$ for which
$F_{d,\alpha+D_1}(0) \neq 0$, equivalently,
for which $e_4(2d,\alpha+D_1) = 1$.

Let us identify those $F_{d,\alpha}$ that appear in the intersection
of the images of $R_{0,2}$ and $R_{0,2}$.
The next lemma shows that there
are essentially~$6$ such common
choices of $F_{d,\alpha}$, where we limit $d$ and $\alpha$ to one
fixed representative of each coset in $B_4/B_2$ and $A_4/A_2$,
respectively.
Equivalently,
let us list the corresponding pairs
$(2d,2\alpha) \in B_2 \times A_2$.  Recall the notation
$B_2 = \{0=b_{00}, b_{10}, b_{01}, b_{11}\}$ from~\eqref{eqDefbij}
and~\eqref{eqDefQnew}.
\begin{lemma}
  \label{lemmaCommon2d2alpha}
The set of pairs $(2d,2\alpha)$ for which $F_{d,\alpha}$ belongs to
the image of both $R_{0,2}$ and $R_{2,0}$ is
\begin{equation}
  \label{eqSixPairs}
  \{ (0,0), (0,E_1), (0,E_2), (0,E_3), (b_{01},0), (b_{01},E_1) \}.
\end{equation}
\end{lemma}
\begin{proof}
First note that if $b \in B_2$ and $\alpha \in A_4$, we have
$e_4(b,\alpha) = e_2(b,2\alpha)$; this is property~(4) on p.~228
of~\cite{MumfordAV} (and it holds more generally for all
$b \in \AV[2]$ and $\alpha \in \AV[4]$).
Our desired condition on $(2d,2\alpha)$ is therefore equivalent to
$e_2(2d,2\alpha) = e_2(2d,2\alpha + E_1) = 1$, because $2D_1 = E_1$.
In particular, $e_2(2d,E_1) = 1$, which forces
$2d \in \{0, b_{01}\}$.  The rest of the calculation is easy.
\end{proof}

From now until the end of the proof of
Proposition~\ref{propositionI21}, it is convenient to introduce 
notation for the six common sections $F_{d,\alpha}$ that we have just
identified above.  First choose any $\alpha_1, \alpha_2, \alpha_3 \in A_4$
with $2\alpha_i = E_i = -E_i$ (for example, we can take
$\alpha_1 = D_1$, but the choice of the $\alpha_i$ does not matter).
Next choose $\delta \in B_4$ with $2\delta = b_{01}$.
Note that $(e_4(\delta,-D_1))^2 = e_4(b_{01},-D_1) = e_2(b_{01},E_1) = 1$.
Hence $e_4(\delta,-D_1) \in \{ \pm 1\}$.  We can modify if necessary
our initial choice of $\delta$ by replacing it with $\delta + b_{10}$;
this does not
change $2\delta$, but it modifies $e_4(\delta,-D_1)$ by a factor of
$e_4(b_{10},-D_1) = e_2(b_{10},E_1) = -1$.  So without loss of
generality, we can arrange for $e_4(\delta,-D_1) = 1$.  We now write
\begin{equation}
\label{eqnTemporarypqTUtunew}
\begin{split}
& S_{00} = F_{0,0},           \qquad
  S_{10} = F_{0,\alpha_1},    \qquad
  S_{01} = F_{0,\alpha_2},    \qquad
  S_{11} = F_{0,\alpha_3},    \\
& s_{00} = S_{00}(0),         \quad\>
  s_{10} = S_{10}(0),         \quad\>\>
  s_{01} = S_{01}(0),         \qquad\>
  s_{11} = S_{11}(0),         \\
& T_0 = F_{\delta,0},         \qquad\>
  T_1 = F_{\delta,\alpha_1},  \qquad\>\>
  t_0 = T_0(0),               \qquad\quad
  t_1 = T_1(0).               \\
\end{split}
\end{equation}
Here $S_{00}, S_{10}, S_{01}, S_{11}, T_0, T_1 \in H^0(\AV,\LL^4)$, and
$s_{00}, s_{10}, s_{01}, s_{11}, t_0, t_1 \in \LL_0^4$.
Note that by our discussion above, the Thetanullwerte
$s_{00}, \dots, t_1$ are all nonzero.

We now have the following explicit formulas.

\begin{proposition}
\label{propositionP2Q2TUnew}
In the formulas below, we use multiplicative notation to write
tensor products of elements of fibers of the same line bundle;
see the pedantic remark about multiplication in the proof of
Corollary~\ref{corFdalphaNonvanishing}.
This means that we write, for example,
$P_{00}(p) \tensor P_{00}(p)$
(or respectively $P_{00}(p) \tensor P_{01}(p)$)
as $P_{00}(p)^2$ (or respectively $P_{00}(p) P_{01}(p)$).
We then have, for $p \in \AV$:
\begin{equation}
\label{eqnP2Q2TUnew}
\begin{split}
\hat{h}_{p,0}(P_{00}(p)^2 + P_{10}(p)^2 + P_{01}(p)^2 + P_{11}(p)^2)
  & = S_{10}(p) \tensor s_{00},\\
\hat{h}_{p,0}(P_{00}(p)^2 - P_{10}(p)^2 + P_{01}(p)^2 - P_{11}(p)^2)
  &= S_{00}(p) \tensor s_{10},\\
\hat{h}_{p,0}(P_{00}(p)^2 + P_{10}(p)^2 - P_{01}(p)^2 - P_{11}(p)^2)
  &= S_{11}(p) \tensor s_{01},\\
\hat{h}_{p,0}(P_{00}(p)^2 - P_{10}(p)^2 - P_{01}(p)^2 + P_{11}(p)^2)
  &= S_{01}(p) \tensor s_{11},\\
\hat{h}_{p,0}(2(P_{00}(p)P_{01}(p) + P_{10}(p)P_{11}(p))) \qquad
  &= T_1(p) \tensor t_0,\\
\hat{h}_{p,0}(2(P_{00}(p)P_{01}(p) - P_{10}(p)P_{11}(p))) \qquad
  &= T_0(p) \tensor t_1,\\
h_{p,0}(Q_{00}(p)^2 + Q_{10}(p)^2 + Q_{01}(p)^2 + Q_{11}(p)^2)
  &= S_{00}(p) \tensor s_{00},\\
h_{p,0}(Q_{00}(p)^2 - Q_{10}(p)^2 + Q_{01}(p)^2 - Q_{11}(p)^2)
  &= S_{10}(p) \tensor s_{10},\\
h_{p,0}(Q_{00}(p)^2 + Q_{10}(p)^2 - Q_{01}(p)^2 - Q_{11}(p)^2)
  &= S_{01}(p) \tensor s_{01},\\
h_{p,0}(Q_{00}(p)^2 - Q_{10}(p)^2 - Q_{01}(p)^2 + Q_{11}(p)^2)
  &= S_{11}(p) \tensor s_{11},\\
h_{p,0}(2(Q_{00}(p)Q_{01}(p) + Q_{10}(p)Q_{11}(p))) \qquad
  &= T_0(p) \tensor t_0,\\
h_{p,0}(2(Q_{00}(p)Q_{01}(p) - Q_{10}(p)Q_{11}(p))) \qquad
  &= T_1(p) \tensor t_1.\\
\end{split}
\end{equation}
\end{proposition}
\begin{proof}
The above is just a restatement of our formulas
\eqref{eqMultiplicationL2} and~\eqref{eqMultiplicationLprime2},
using the fact that
$P_j(p) = \theta_{[2],b_j}(p+D_1)$ and
$Q_j(p) = \theta_{[2],b_j}(p)$.  Note that in the sums over
$c \in B_2 = \{ b_{00}, b_{10}, b_{01}, b_{11} \}$ we have
$e_4(c,\alpha_i) = e_2(c, E_i)$, whose values are given in~\eqref{eqDefbij}.
\end{proof}

Our next result uses Proposition~\ref{propositionP2Q2TUnew} to deduce
a projective equality between points in $\Projective^5$.  The six
coordinates of these points
correspond to the six common sections $F_{d,\alpha}$ from above.
As in the discussion following~\eqref{eqDefKappa}, we
allow the coordinates in~$\Projective^5$ 
to be elements of the same one-dimensional vector space instead of
$\Kbar$.  Note also that if $V$ and $W$ are both one-dimensional
vector spaces, and $h: V \to W$ is an isomorphism, then we always have
$[h(v_1), \dots, h(v_6)] = [v_1, \dots, v_6]$.  We will use this
tacitly throughout.  One example is with
$V = \LL_p^2 \tensor \LL_p^2 = \LL_p^4$, a fiber of
our line bundle, and the isomorphism
$h_{p,0}: \LL_p^2 \tensor \LL_p^2 \to \LL_p^4 \tensor \LL_0^4 = W$.
We also choose an isomorphism $\LL_0^4 \isomorphic \Kbar$, and
identify the Thetanullwerte $s_{00},\dots,t_1$ with elements
$\tilde{s}_{00}, \dots, \tilde{t}_1 \in \Kbar$ (which, as we know, are
nonzero).  This allows us to replace a
tensor such as $S_{00}(p) \tensor s_{00} \in W$ by a product
$\tilde{s}_{00} S_{00}(p) \in \LL_p^4$ inside projective coordinates.

\begin{proposition}
  \label{propositionFirstDefABCDEF}
For every $p \in \AV$, the projective point (note the unusual order of
the coordinates) 
\begin{equation}
  \label{eqQ2VectorOfLen6}
 [Q_{01}(p)^2,Q_{01}(p)Q_{00}(p),Q_{00}(p)^2,
        Q_{11}(p)^2,Q_{11}(p)Q_{10}(p),Q_{10}(p)^2] \in \Projective^5
\end{equation}
is equal to
$[\tilde{A}(p),\tilde{B}(p),\tilde{C}(p),
      \tilde{D}(p),\tilde{E}(p),\tilde{F}(p)]$,
where $\tilde{A}, \dots, \tilde{F} \in H^0(\AV,\LL^4)$ are given by
\begin{equation}
\label{eqnsAFSTnew}
\begin{split}
\tilde{A} &= \tilde{s}_{00} S_{00} + \tilde{s}_{10}S_{10}
   - \tilde{s}_{01}S_{01} - \tilde{s}_{11}S_{11},\\
\tilde{B} &= \tilde{t}_0 T_0 + \tilde{t}_1 T_1,\\
\tilde{C} &= \tilde{s}_{00} S_{00} + \tilde{s}_{10}S_{10}
   + \tilde{s}_{01}S_{01} + \tilde{s}_{11}S_{11},\\  
\tilde{D} &= \tilde{s}_{00} S_{00} - \tilde{s}_{10}S_{10}
   - \tilde{s}_{01}S_{01} + \tilde{s}_{11}S_{11},\\
\tilde{E} &= \tilde{t}_0 T_0 - \tilde{t}_1 T_1,\\
\tilde{F} &= \tilde{s}_{00} S_{00} - \tilde{s}_{10}S_{10}
   + \tilde{s}_{01}S_{01} - \tilde{s}_{11}S_{11}.\\
\end{split}
\end{equation}
Moreover, there exist elements $A, \dots, F \in R_{2,0}$ for which
the projective point $[A(p), B(p), C(p), D(p), E(p), F(p)]$ is equal
to the vector in~\eqref{eqQ2VectorOfLen6}.  Specifically, define the
elements $\hat{S}_{00}, \dots, \hat{T}_1 \in R_{2,0}$ by
\begin{equation}
  \label{eqABCDEFinTermsOfPpart1}
  \begin{split}
    \hat{S}_{00}
    &= \tilde{s}_{10}^{-1}
         (P_{00}^2 - P_{10}^2 + P_{01}^2 - P_{11}^2),\\
    \hat{S}_{10}
    &= \tilde{s}_{00}^{-1}
         (P_{00}^2 + P_{10}^2 + P_{01}^2 + P_{11}^2),\\
    \hat{S}_{01}
    &= \tilde{s}_{11}^{-1}
         (P_{00}^2 - P_{10}^2 - P_{01}^2 + P_{11}^2),\\
    \hat{S}_{11}
    &= \tilde{s}_{01}^{-1}
         (P_{00}^2 + P_{10}^2 - P_{01}^2 - P_{11}^2),\\
    \hat{T}_0
    &= \tilde{t}_1^{-1}(2(P_{00}P_{01} - P_{10}P_{11})),\\
    \hat{T}_1
    &= \tilde{t}_0^{-1}(2(P_{00}P_{01} + P_{10}P_{11})).\\
  \end{split}
\end{equation}
Then define, in a way parallel to~\eqref{eqnsAFSTnew},
$A = \tilde{s}_{00} \hat{S}_{00} + \tilde{s}_{10}\hat{S}_{10}
        - \tilde{s}_{01}\hat{S}_{01} - \tilde{s}_{11}\hat{S}_{11}$,
$B = \tilde{t}_0 \hat{T}_0 + \tilde{t}_1 \hat{T}_1$,
and so forth, until
$F = \tilde{s}_{00} \hat{S}_{00} - \tilde{s}_{10}\hat{S}_{10}
        + \tilde{s}_{01}\hat{S}_{01} - \tilde{s}_{11}\hat{S}_{11}$.
\end{proposition}
\begin{proof}
The first assertion amounts to combining the equations
in~\eqref{eqnP2Q2TUnew} to obtain statements such as
$h_{p,0}(4 Q_{0,1}(p)^2)
  = S_{00}(p) \tensor s_{00} + S_{10}(p) \tensor s_{10}
    -  S_{01}(p) \tensor s_{01} - S_{11}(p) \tensor s_{11}$,
and similarly for $h_{p,0}$ applied to $4$ times the other components
of~\eqref{eqQ2VectorOfLen6}.  The isomorphism $h_{p,0}$ and the
common factor~$4$, as well as the identification of $\LL_0^4$ with
$\Kbar$, do not affect the projective point.

The second assertion holds because the $\hat{S}_j$ (similarly for
$\hat{T}_0, \hat{T}_1$) are precisely those
elements of $R_{2,0}$ that satisfy
$\hat{h}_{p,0}(\hat{S}_j(p)) = S_j(p)$, once we take into account the
identification of $\LL_0^4$ with~$\Kbar$.  Thus we have an equality of
projective points
$[\tilde{A}(p),\dots,\tilde{F}(p)] = [A(p), \dots, F(p)]$.
\end{proof}

We can finally give the promised elements of $I_{2,1}$ and a formula
for the Kummer quartic in $I_{4,0}$ (this is the quartic that defines
the image of $\kappa'$).  Essentially the same result
holds with the roles of the $P$s and $Q$s reversed, giving us elements
of $I_{1,2}$ and $I_{0,4}$.  Note however that when we replace $Q_j$
by $P_j$, which amounts to translation by $D_1$, we replace each $P_j$
(itself already a translate of $Q_j$ by $D_1$)
with
the result of translating 
$Q_j$ by $E_1 = 2D_1$;
so we replace $P_j$ by $\widetilde{E_1}*Q_j \in \{ \pm Q_j\}$.
This introduces some sign changes, but
does not affect the structure of the result.

\begin{proposition}
\label{propEltsI21I40new}
With the abovementioned elements $A,B,C,D,E,F \in R_{2,0}$ from
Proposition~\ref{propositionFirstDefABCDEF}, we have that the
following elements belong to the ideal $I$: 
\begin{equation}
\label{eqnI21PropBnew}
Q_{00} A - Q_{01} B,
\quad
Q_{00} B - Q_{01} C,
\quad
Q_{10} D - Q_{11} E,
\quad
Q_{10} E - Q_{11} F 
\in I_{2,1}.
\end{equation}
We also obtain
\begin{equation}
\label{eqnI21PropCnew}
AC - B^2, \quad DF - E^2 \in I_{4,0}.
\end{equation}
Both $AC-B^2$ and $DF - E^2$ must be multiples of the Kummer quartic;
in fact, they are equal, and
\begin{equation}
\label{eqnI21PropDnew}
\begin{split}
& AC-B^2 = DF-E^2 \\
  & = \frac{\tilde{s}_{00}^2}{\tilde{s}_{10}^2}
          (P_{00}^2 - P_{10}^2 + P_{01}^2 - P_{11}^2)^2 \\
  & \quad + \frac{\tilde{s}_{10}^2}{\tilde{s}_{00}^2}
          (P_{00}^2 + P_{10}^2 + P_{01}^2 + P_{11}^2)^2 \\
  & \quad - \frac{\tilde{s}_{01}^2}{\tilde{s}_{11}^2}
          (P_{00}^2 - P_{10}^2 - P_{01}^2 + P_{11}^2)^2 \\
  & \quad - \frac{\tilde{s}_{11}^2}{\tilde{s}_{01}^2}
          (P_{00}^2 + P_{10}^2 - P_{01}^2 - P_{11}^2)^2 \\ 
  & \quad - 4\frac{\tilde{t}_0^2}{\tilde{t}_1^2}
          (P_{00}P_{01} - P_{10}P_{11})^2 \\
  & \quad - 4\frac{\tilde{t}_1^2}{\tilde{t}_0^2}
          (P_{00}P_{01} + P_{10}P_{11})^2. \\
\end{split}
\end{equation}
\end{proposition}
\begin{proof}
Consider any of the expressions in \eqref{eqnI21PropBnew}
and~\eqref{eqnI21PropCnew}.  To show that such an expression belongs
to the ideal $I$, we must check that it vanishes at all $p \in \AV$.
This follows from the equality between the projective point
$[A(p),B(p),\dots,F(p)]$ and the projective
point $[Q_{01}(p)^2,Q_{01}(p)Q_{00}(p),\dots,Q_{10}(p)^2]$
from~\eqref{eqQ2VectorOfLen6}.
For example, 
$Q_{00}(p) A(p) - Q_{01}(p) B(p)$ (which is technically an element of
$\LL_p^2 \tensor (\LL'_p)^4$) is ``proportional'' to the element
$Q_{00}(p)(Q_{01}(p)^2) - Q_{01}(p)(Q_{01}(p)Q_{00}(p))$,
which belongs to $\LL_p^2 \tensor \LL_p^4 = \LL_p^6$.  This element
vanishes, because taking products in tensor powers of $\LL_p$ is commutative.

At this point, it is also possible to prove~\eqref{eqnI21PropDnew}
directly by expressing all of $A,\dots,F$ in terms of the $P_i$ and
expanding.  This is too large to do by hand in that form, but the
computation becomes quite approachable if we use our elements
$\hat{S}_{00}, \dots, \hat{T}_1 \in R_{2,0}$, as given
in~\eqref{eqABCDEFinTermsOfPpart1}.
We have
\begin{equation}
\label{eqnACB2DFE2new}
\begin{split}
AC-B^2
  &= (\tilde{s}_{00} \hat{S}_{00} + \tilde{s}_{10} \hat{S}_{10})^2 -
(\tilde{s}_{01} \hat{S}_{01} + \tilde{s}_{11}\hat{S}_{11})^2
    - ( \tilde{t}_0 \hat{T}_0 + \tilde{t}_1 \hat{T}_1)^2,\\
DF-E^2
  &= (\tilde{s}_{00} \hat{S}_{00} - \tilde{s}_{10} \hat{S}_{10})^2 -
(\tilde{s}_{01} \hat{S}_{01} - \tilde{s}_{11} \hat{S}_{11})^2
   - (\tilde{t}_0 \hat{T}_0 - \tilde{t}_1 \hat{T}_1)^2.\\
\end{split}
\end{equation}
One-quarter of the difference is then
\begin{equation}
\label{eqnACB2DFE2diffnew}
\begin{split}
&4^{-1}((AC-B^2) - (DF-E^2))
  = \tilde{s}_{00} \tilde{s}_{10} \hat{S}_{00} \hat{S}_{10}
        - \tilde{s}_{01} \tilde{s}_{11} \hat{S}_{01} \hat{S}_{11}
      - \tilde{t}_0 \tilde{t}_1 \hat{T}_0 \hat{T}_1\\
&= (P_{00}^2 - P_{10}^2 + P_{01}^2 - P_{11}^2)
        (P_{00}^2 + P_{10}^2 + P_{01}^2 + P_{11}^2)\\
& \qquad - (P_{00}^2 - P_{10}^2 - P_{01}^2 + P_{11}^2)
               (P_{00}^2 + P_{10}^2 - P_{01}^2 - P_{11}^2)\\
& \qquad -4(P_{00}P_{01} - P_{10}P_{11})(P_{00}P_{01} + P_{10}P_{11})\\ 
&= (P_{00}^2 + P_{01}^2)^2 - (P_{10}^2 + P_{11}^2)^2
  - (P_{00}^2 - P_{01}^2)^2 + (P_{10}^2 - P_{11}^2)^2\\
  &\qquad - 4P_{00}^2 P_{01}^2 + 4P_{10}^2 P_{11}^2\\
&= 0.\\
\end{split}
\end{equation}
The common value of $AC-B^2$ and $DF-E^2$ is the parts
of~\eqref{eqnACB2DFE2new} that are not affected by the sign changes
between the two lines.  In other words,
\begin{equation}
\label{eqnACB2DFE2commonnew}
\begin{split}
& AC-B^2 = DF-E^2 \\
  & = \tilde{s}_{00}^2 \hat{S}_{00}^2 + \tilde{s}_{10}^2 \hat{S}_{10}^2 
         - \tilde{s}_{01}^2 \hat{S}_{01}^2 - \tilde{s}_{11}^2 \hat{S}_{11}^2  
       - \tilde{t}_0^2 \hat{T}_0^2 - \tilde{t}_1^2 \hat{T}_1^2 ,\\
\end{split}
\end{equation}
and this proves~\eqref{eqnI21PropDnew}.
\end{proof}

\begin{remark}
\label{rmkOnThetaConstsnew}
We note that the constants such as $\frac{\tilde{s}_{00}^2}{\tilde{s}_{10}^2}$
and $\frac{\tilde{t}_0^2}{\tilde{t}_1^2}$ appearing
in~\eqref{eqnI21PropDnew} are independent of the identification made
between $\LL_0^4$ and $\Kbar$; they could equally well have been
written as $\frac{s_{00}^2}{s_{10}^2}$ and $\frac{t_0^2}{t_1^2}$, with
the obvious interpretation of quotients of
(nonzero)
elements of $\LL_0^4$.
\end{remark}

We still need to show that the common value of $AC-B^2$ and~$DF-F^2$
is not zero.  To do this, we need to study the values of
the Thetanullwerte $\tilde{s}_j$ and $\tilde{t}_i$, as well as their
relation to the values $q_j = Q_j(0) \in \LL_0^2$.  Analogously to our
previous definition, we choose an isomorphism between $\LL_0^2$
and~$\Kbar$, under which each $q_j$ can be identified with
$\tilde{q}_j \in \Kbar$, for $j \in \{00,10,01,10\}$.  Hence
$\kappa(0)
  = [\tilde{q}_{00}, \tilde{q}_{10}, \tilde{q}_{01}, \tilde{q}_{11}]$.

We choose our identification between the $q_j$ and the $\tilde{q}_j$ so 
as to have actual equality in the corresponding equations
from~\eqref{eqnP2Q2TUnew}, when $p=0$, after also composing with the
isomorphisms $h_{0,0}$; otherwise, we would only have an equality of
projective points.  Having done all this, we obtain the following
identities:
\begin{equation}
  \label{eqnP2Q2TUatZeroNew}
  \begin{split}
\tilde{q}_{00}^2 + \tilde{q}_{10}^2 + \tilde{q}_{01}^2 + \tilde{q}_{11}^2
  &= \tilde{s}_{00}^2,\\
\tilde{q}_{00}^2 - \tilde{q}_{10}^2 + \tilde{q}_{01}^2 - \tilde{q}_{11}^2
  &= \tilde{s}_{10}^2,\\
\tilde{q}_{00}^2 + \tilde{q}_{10}^2 - \tilde{q}_{01}^2 - \tilde{q}_{11}^2
  &= \tilde{s}_{01}^2,\\
\tilde{q}_{00}^2 - \tilde{q}_{10}^2 - \tilde{q}_{01}^2 + \tilde{q}_{11}^2
  &= \tilde{s}_{11}^2,\\
2(\tilde{q}_{00}\tilde{q}_{01} + \tilde{q}_{10}\tilde{q}_{11}) 
  &= \tilde{t}_0^2,\\
2(\tilde{q}_{00}\tilde{q}_{01} - \tilde{q}_{10}\tilde{q}_{11}) 
  &= \tilde{t}_1^2.\\
\end{split}
\end{equation}
One can proceed similarly with the $P_j(0)$, if desired; we do not
need them for the treatment here.

\begin{lemma}
  \label{lemmaQzeroNonvanishing}
  Assume (as is the case in Section~\ref{section4}) that the quotient
  abelian variety $\AV' = \AV/A_2$ is the Jacobian of a genus~$2$
  curve $\calC'$, where $\calC'$ is related to $\calC$ via a Richelot
  isogeny.  Then 
  $\tilde{q}_{00}$, $\tilde{q}_{10}$, $\tilde{q}_{01}$, and
  $\tilde{q}_{11}$, are all nonzero.
\end{lemma}
\begin{proof}
We apply Lemma~\ref{lemmaTheta1Nonvanishing} to $\calC'$ and $\AV'$.
Essentially, the $\tilde{q}_j$s are even theta characteristics for
$\AV'$.  Since we have not set up extended theta structures that
include an explicit action of $[-1]$, we prove our lemma by a slightly
different approach.

The line bundle $\LL^2$ on~$\AV$ descends to a symmetric line bundle
$\LLhat$ on $\AV'$, along the lift from $A_2$ to $\tilde{A}_2
\subset \calG_2$.  Moreover, $\LLhat$ gives a principal polarization
on $\AV'$.  This means that we can view $Q_{00}$, which is invariant
under $\tilde{A}_2$, as the (unique, up to a scalar) nonzero
element of $H^0(\AV',\LLhat)$; thus $Q_{00}$ plays the same role as
$\theta_{[1],0} \in H^0(\AV,\LL)$, so $Q_{00}$ vanishes at precisely $6$
points of $\AV'[2]$.

The points of $\AV'[2]$ correspond to points of $(A_4 \directsum
B_2)/A_2$.  Represent each such point as $\alpha + b$, where
$\alpha \in \{0, \alpha_1, \alpha_2, \alpha_3\}$ as in the discussion
immediately after Lemma~\ref{lemmaCommon2d2alpha}, and
$b \in \{0 = b_{00}, b_{10}, b_{01}, b_{11}\}$.  Since
$\widetilde{b_j}*Q_{00} = Q_j$ for $j \in \{00,10,01,11\}$, the
question of whether $Q_{00}(\alpha+b_j) = 0$ is equivalent to whether
$Q_j(\alpha) = 0$, and this vanishing occurs for precisely six of the
16 choices for the pair $(j,\alpha)$.  On the other hand, we have
$\kappa(\alpha_i) = \kappa(-\alpha_i)$, but also
$-\alpha_i = \alpha_i - E_i$.  Hence we can
use~\eqref{eqEprojectiveOnKappa} to obtain identities of
projective points, such as for example
\begin{equation}
  \label{eqQatFourTorsion}
  \begin{split}
    \kappa(\alpha_1)
    &= [Q_{00}(\alpha_1),Q_{10}(\alpha_1),Q_{01}(\alpha_1),Q_{11}(\alpha_1)]
          \\
    = \kappa(\alpha_1 - E_1)
    &= [Q_{00}(\alpha_1),-Q_{10}(\alpha_1),Q_{01}(\alpha_1),-Q_{11}(\alpha_1)].
          \\ 
  \end{split}
\end{equation}
This equality in $\Projective^3$ implies that either
$Q_{10}(\alpha_1) = Q_{11}(\alpha_1) = 0$, or that
$Q_{00}(\alpha_1) = Q_{01}(\alpha_1) = 0$.  This identifies two points
of $\AV'[2]$ where $Q_{00}$ vanishes.  Similarly, using $\alpha_2$ and
$\alpha_3$, we identify four more points where $Q_{00}$ vanishes.
This brings our total to six, which are all associated to
points of $\AV'[2]$ with $\alpha \neq 0$.  In particular, all the
$Q_j(0)$ are nonzero, as desired.
\end{proof}

\begin{lemma}
\label{lemmaACminusBsquaredNonzero}
The expression $AC - B^2 = DF - E^2$ is a nonzero element of
$I_{4,0} \subset R_{4,0}$; it is therefore the Kummer quartic equation. 
\end{lemma}
\begin{proof}
Under the hypotheses of Lemma~\ref{lemmaQzeroNonvanishing}, this follows
from the fact that the  
coefficient of $P_{00} P_{01} P_{10} P_{11}$ in~\eqref{eqnI21PropDnew}
is a constant times
$\frac{\tilde{t}_0^2}{\tilde{t}_1^2}
      - \frac{\tilde{t}_1^2}{\tilde{t}_0^2}$,
whose numerator, $\tilde{t}_0^4 - \tilde{t}_1^4$ is,
by~\eqref{eqnP2Q2TUatZeroNew}, a 
constant times
$\tilde{q}_{00}\tilde{q}_{01}\tilde{q}_{10}\tilde{q}_{11}$,
which is nonzero in this setting.

The proof in the general case takes more work.
If~\eqref{eqnI21PropDnew} is zero, then all its
coefficients must 
vanish.  Equating the coefficients of $P_{00}^2 P_{10}^2$,
$P_{00}^2 P_{11}^2$, and $P_{00}^4$ to zero,
we obtain that the only way that~\eqref{eqnI21PropDnew}
could be zero is if
\begin{equation}
\label{eqnIfACminusBsquaredVanishes} 
\frac{\tilde{s}_{00}^2}{\tilde{s}_{10}^2}
  = \frac{\tilde{s}_{10}^2}{\tilde{s}_{00}^2}
  = \frac{\tilde{s}_{01}^2}{\tilde{s}_{11}^2}
  = \frac{\tilde{s}_{11}^2}{\tilde{s}_{01}^2}.
\end{equation}
But then the common value
of~\eqref{eqnIfACminusBsquaredVanishes} is equal to its own
reciprocal, so must be $\pm 1$.

In case this common value is
$1$, then the equalities
$\tilde{s}_{00}^2 = \tilde{s}_{10}^2$ and
$\tilde{s}_{01}^2 = \tilde{s}_{11}^2$ imply
by~\eqref{eqnP2Q2TUatZeroNew} that
$\tilde{q}_{10} = \tilde{q}_{11} = 0$.
This means that $\kappa(0) = [\tilde{q}_{00}, 0, \tilde{q}_{01}, 0]$.
But then we would have $\kappa(0) = \kappa(E_1)$,
by~\eqref{eqEprojectiveOnKappa}.
This contradicts the
fact that the Kummer map sends two points $p,q \in \AV$ to the
same point $\kappa(p) = \kappa(q) \in \Projective^3$ if and only if
$p = \pm q$; see for 
example Step~I in the proof of Theorem~4.8.1
of~\cite{BirkenhakeLange}.  (Their proof works over an arbitrary
field; note that the divisor $\Theta$ is isomorphic to $\calC$, and is
therefore an irreducible variety.)
The case when the common value of~\eqref{eqnIfACminusBsquaredVanishes}
is $-1$ is analogous, since in that case we would obtain instead
$\tilde{q}_{00} = \tilde{q}_{01} = 0$.
\end{proof}

We can finally prove the stronger version of
Proposition~\ref{propositionI31}.

\begin{proposition}
\label{propositionI21}
For $k \geq 2$, we have $R_{k-2,0} \cdot I_{2,1} = I_{k,1}$ and
$R_{0,k-2} \cdot I_{1,2} = I_{1,k}$.  
\end{proposition}
\begin{proof}
As usual, we only treat the case of $I_{k,1}$.
For the purposes of this proof, we write
$R_{*,0} = \directsum_{d \geq 0} R_{d,0} = \Kbar[\{P_i\}]$, the ring of
polynomials in the $P_i$ alone.  We also identify
$R_{*,1} = \directsum_{d \geq 0} R_{d,1}$ with $(R_{*,0})^4$, using
the basis $\{Q_j\}$. We similarly identify
$I_{*,1} = \directsum_{d \geq 0} I_{d,1}$ with an $R_{*,0}$-submodule
of $(R_{*,0})^4$.
In this setting, the four elements of $I_{2,1}$ listed
in~\eqref{eqnI21PropBnew} correspond to the vectors
\begin{equation}
  \label{eqnI21vectors}
(A,-B,0,0),
\qquad
(B,-C,0,0),
\qquad
(0,0,D,-E),
\qquad
(0,0,E,-F).
\end{equation}
Our goal is to show that the $R_{*,0}$-submodule
of $(R_{*,0})^4$ generated by these four vectors corresponds precisely
to $I_{*,1}$.  Now these four
vectors are linearly independent, even over the field of fractions of
$R_{*,0}$ (that is, the field of rational functions in the $P_i$, viewed as
independent transcendentals).
Indeed, putting together these vectors into a $4\times4$ matrix with
entries in $R_{*,0}$, this matrix is block diagonal, with its
$2\times2$ subblocks having determinants $-AC+B^2$ and
$-DF+E^2$.  As we have seen, these determinants are equal and nonzero.

The above implies that our four elements of $I_{2,1}$ generate a free
$R_{*,0}$-module of rank four, and this free module is itself a
submodule of $I_{*,1}$.  Now the dimension of the bidegree $(k,1)$
component of this free module is exactly $4 \dim R_{k-2,0}$, by
viewing each element of the $(k,1)$-homogeneous component as a linear
combination of our four vectors, with coefficients in $R_{k-2,0}$.  We
also know that 
$\dim I_{k,1} = 4 \dim R_{k-2,0}$, by Corollary~\ref{dimensionsIde}.
So we have proved our desired result.

Our proof yields expressions for 
each of $(-AC+B^2,0,0,0)$, $(0,-AC+B^2,0,0)$, $(0,0,-DF+E^2,0)$, and
$(0,0,0,-DF+E^2)$ as an $R$-linear combination of our original four
vectors.  In other words, the product of the Kummer quartic by each
$Q_j$ is an easy combination of the elements from
$I_{2,1}$: for example, 
$C (Q_{00} A - Q_{01} B) - B (Q_{00} B - Q_{01} C) = Q_{00} ( AC - B^2)
 \in I_{4,1}$.
This gives a good way to see how our four elements of $I_{2,1}$
actually generate all multiples of the Kummer quartic in $I_{d,e}$
when $e \geq 1$.  On some level, the homogeneous component $I_{4,0}$
is only needed to generate the $I_{k,0}$ for $k \geq 4$.
\end{proof}

Putting all the above results together, we obtain the our basic
structural result about generators for the bigraded ideal of
relations between the $P_i$s and the $Q_j$s.  We state this slightly
more generally, to bring out the parts of the construction that will
matter in Section~\ref{section4}, when we carry out similar computations
while working carefully over a field of definition for the original
curve $\calC$.

\begin{theorem}
\label{generatorsInEachBidegree}
Let $\AV$ be the Jacobian of a genus~$2$ curve $\calC$ over a field~$K$
that is not of characteristic~$2$.  Let $\LL$ be a symmetric line
bundle on $\AV$ that gives rise to the principal polarization on
$\AV$, so that $\dim H^0(\AV,\LL^2) = 4$ and this linear series yields
the Kummer embedding of $\AV/[\pm 1]$ into $\Projective^3$.  Also
let $\LL'$ be the translate of $\LL$ by a point of order~$4$ in $\AV(K)$,
and consider the resulting embedding $\AV \hookrightarrow
\Projective^3 \times \Projective^3$ associated to the linear series of
$\LL^2$ and $(\LL')^2$.  Then the resulting bigraded ideal $I$ of
relations in this model is generated by:
\begin{itemize}
\item $I_{0,4}$ and $I_{4,0}$, which are each $1$-dimensional,
\item $I_{1,2}$ and $I_{2,1}$, which are each $4$-dimensional, and
\item $I_{2,2}$, which is $36$-dimensional.
\end{itemize}
Moreover, there are no equations of bidegree $(1,1)$, that is, $I_{1,1} = 0$.
\end{theorem}
\begin{proof}
The dimensions of the components $I_{d,e}$, and the ranks of the
multiplication maps from $R_{d,e} \times I_{d',e'} \to
I_{d+d',e+e'}$, are unaffected by passing from~$K$ to its 
algebraic closure~$\Kbar$.  We may therefore place ourselves in the
situation that was studied in this section, with the Kummer embedding
given by $\kappa$, in terms of the $\{Q_j\}$ coordinates, and the
embedding into $\Projective^3 \times \Projective^3$ being given by the
$P$s and the $Q$s.  In that case,
the relations in bidegrees $(0,4)$ and $(4,0)$ generate all relations
of bidegrees $(0,k)$ and $(k,0)$, by the proof of
Proposition~\ref{Ik0I0k}.  The relations in bidegrees $(1,2)$ and
$(2,1)$ generate all relations of bidegrees $(1,k)$ and $(k,1)$, by
Proposition~\ref{propositionI21}.
The relations in bidegree $(2,2)$ generate all the $I_{d,e}$ with $d,e
\geq 2$, by Proposition~\ref{propositionI22}.  The dimensions of all
the $I_{d,e}$ have also been calculated above, in
Proposition~\ref{Ik0I0k} and Corollary~\ref{dimensionsIde}.
\end{proof}

We will also need a result about the relations between the $P_i$ and
$Q_j$ on sums and differences of points on $\AV$.  This follows from
combining Proposition~\ref{Rbar11} with
Theorem~\ref{thmAlgebraicAdditionFormula}.

\begin{theorem}
\label{PQAddition}
There exist polynomials $A_{ij}$, with coefficients in
$\Kbar$, and of multidegree $(1,1,1,1)$ in the four sets of variables
$\{P_{i'}(p)\}, \{Q_{j'}(p)\}, \{P_{k'}(r)\}, \{Q_{\ell'}(r)\}$,
such that
for all $p,r \in \AV$, the $4\times4$ matrices below are
projectively equal:
\begin{equation}
\label{eqnPQAddition}
\Bigl[P_i(p+r) Q_j(p-r)\Bigr]_{i,j}
     = \Bigl[A_{ij}\bigl(\{P_{i'}(p)\}, \{Q_{j'}(p)\},
               \{P_{k'}(r)\}, \{Q_{\ell'}(r)\}\bigr)\Bigr]_{i,j}.
\end{equation}
The essential claim here is that the coefficients of $A_{ij}$ do not
depend on the points $p,r$.  Projective equality of matrices can be
viewed as equality in $\Projective^{15}$, if one enumerates the
entries of each matrix in the same order.
\end{theorem}
\begin{proof}
This proof uses similar ideas to the proof of
Theorem~\ref{Rbar11}.  Once again, let $q \in \AV$ satisfy $2q = D_1$.
In the setting of~\eqref{eqAlgebraicThetaAddition}, we have
\begin{equation}
  \label{eqnPQwithShift}
  \begin{split}
  h_{p+q,r+q}&(\theta_{[2],b_1}(p+r+D_1)
         \tensor \theta_{[2],b_2}(p-r))\\
  &= \sum_{c \in B_2}
      \theta_{[4],d_1+d_2+c}(p+q)
        \tensor \theta_{[4],d_1-d_2+c}(r+q).\\
  \end{split}
\end{equation}
As $b_1$ and $b_2$ vary over all of $B_2$, the left hand side of the
above equation gives the entries of the projective matrix
$[P_i(p+r) Q_j(p-r)]_{i,j}$; as usual, the isomorphism $h_{p+q,r+q}$
respects projective equality.  In the analogous matrix made out of
the right hand side of~\eqref{eqnPQwithShift}, we can view
$\theta_{[4],d_1+d_2+c}(p+q)$ as the value at $p$ of a section of
$T_q^* \LL^4$, just as in the proof of Proposition~\ref{Rbar11}, where
the section in question can be interpreted as an element of
$R_{1,1}$; so we can express each $\theta_{[4],d_1+d_2+c}(p+q)$ (up
to projective equivalence) as a polynomial of bidegree $(1,1)$ in
the variables $\{P_{i'}(p)\}, \{Q_{j'}(p)\}$.  Similarly, we can view
each expression $\theta_{[4],d_1-d_2+c}(r+q)$ as the value at
$r$ of another element of $R_{1,1}$, in other words as a polynomial
of bidegree $(1,1)$ in the variables $\{P_{i'}(r)\}, \{Q_{j'}(r)\}$.
These identifications are all made up to projective equivalence, in
other words up to a common ``constant factor'' that is independent of
any $d$ in $\theta_{[4],d}$.  Each term in the sum on the right, being
a product
$\theta_{[4],d_1+d_2+c}(p+q) \tensor \theta_{[4],d_1-d_2+c}(r+q)$,
can therefore be viewed as a certain polynomial of multidegree
$(1,1,1,1)$ in our four sets of variables.  Adding up all these
polynomials produces the desired $A_{ij}$.

For the diligent reader, here is a sketch of a more
explicit, but messier, proof of this theorem, following an approach
that is similar to that in
Propositions \ref{propositionP2Q2TUnew}
and~\ref{propositionFirstDefABCDEF}.
Each $P_i(p+r)Q_j(p-r)$ can be written as a linear combination
of expressions analogous to the left hand side
of~\eqref{eqBasicAdditionFormula}.  In somewhat loose notation, let us
write these analogous expressions as
$\sum_{c \in B_2} e_4(c,\alpha) P_{2d+c}(p+r) Q_c(p-r)$, which we can
identify (projectively) as the product 
$F_{d,\alpha}(p+q) F_{d,\alpha}(r+q)$.
Specializing to $r=0$ gives a (projective)
identity between
$[\sum_{c \in B_2} e_4(c,\alpha) P_{2d+c}(p) Q_c(p)]_{d,\alpha}$ and
$[\tilde{f}_{d,\alpha}F_{d,\alpha}(p + q)]_{d,\alpha}$; a similar
identity holds when $p=0$.
Here we have identified
the nonvanishing theta constants $F_{d,\alpha}(q)$
with field elements $\tilde{f}_{d,\alpha} \in \Kbarstar$.
Write
$U_{d,\alpha} = \sum_{c \in B_2} e_4(c,\alpha) P_{2d+c} Q_c \in R_{1,1}$.
We then projectively identify the three $16$-tuples
$\Bigl[\sum_{c \in B_2} e_4(c,\alpha) P_{2d+c}(p+r) Q_c(p-r)\Bigr]_{d,\alpha}$,
$\Bigl[F_{d,\alpha}(p+q) \tensor F_{d,\alpha}(r+q)\Bigr]_{d,\alpha}$,
and  
$\Bigl[\tilde{f}_{d,\alpha}^{-2} {U}_{d,\alpha}(p)
  {U}_{d,\alpha}(r) \Bigr]_{d,\alpha}$.
Taking corresponding linear combinations of the entries of the first
and third
$16$-tuples produces the identification in~\eqref{eqnPQAddition}.
\end{proof}

\begin{remark}
\label{remarkI22andUsualQuadrics}
Throughout our discussion, the line bundle $\MM_{1,1}$ and its sections
have figured prominently.  The reason is that if we compose our
embedding $\AV \to \Projective^3 \times \Projective^3$ with the Segre
map $\Projective^3 \times \Projective^3 \hookrightarrow
\Projective^{15}$, this gives precisely the projective embedding of
$\AV$ that is given by $\MM_{1,1}$.  Since this line bundle is
algebraically equivalent to $\LL^4$, the ideal describing the image of
$\AV$ in $\Projective^{15}$ is the usual homogeneous ideal in $16$
variables that is generated by $72$ quadrics.  These quadric
generators correspond to the $36$ basis elements of $I_{2,2}$, combined
with the $36$ quadrics describing the image of the Segre map: these
are $(P_i Q_j)(P_k Q_\ell) - (P_i Q_\ell)(P_k Q_j)$.  It follows that,
in our bigraded ring $R$, our ideal $I$ is in fact the saturation of
the ideal generated by $I_{2,2}$.  This explains the important role
played by $I_{2,2}$ in this section. 
\end{remark}


\section{An approach for an explicit derivation of the Edwards curve}
\label{section3}

Recall the model for the Edwards curve 
in~\cite{BernsteinLangeEC} (generalising the form
given in~\cite{EdwardsNormalForm}):
\begin{equation}\label{EdwardsCurve}
U^2 + Y^2 = 1 + d U^2 Y^2.
\end{equation}
This has a universal group law, provided
that~$d$ is non-square over the ground field. 
In this section, we
present a style of deriving the Edwards curve
and choices for its group law, which make use of 
a $\Projective^1 \times \Projective^1$ embedding, arising from the 
projective $x$-coordinates of a point~$D$ and $D + D_1$, where~$D_1$ is
a fixed point of order~$4$. We shall then imitate this style and
notation when we describe our genus~$2$ approach in the next section.

We first describe a general elliptic curve~$\calC$,
defined over a field~$K$, not of characteristic~$2$, which has
a point~$D_1$ of order~$4$. Then $E_1 = 2 D_1$ will have order~$2$,
and there will be a $2$-isogeny~$\phi$ from~$\calC$ to a curve~$\calC'$
of the form $y^2 = x(x-\alpha)(x-\beta)$.
(Here we have imposed the additional requirement that all the
$2$-torsion points of $\calC'$ are defined over $K$.)
Say that $(\alpha,0) = \phi(D_1)$
so that~$\alpha$ is square. Scale~$\alpha$
to~$1$, take~$\calC'$ to have the form $y^2 = x(x-1)(x-d)$,
and then we may use the standard $2$-isogeny with kernel
$\langle (0,0) \rangle$ from $y^2 = x(x^2 + h_1 x + h_2)$
to $y^2 = x(x^2 - 2 h_1 x + h_1^2 - 4 h_2)$ 
(as given on p.74 of~\cite{SilvermanBook}) to find the curve
\begin{equation}\label{eqnforellipticC}
\calC : y^2 = x(x^2 + 2(1+d)x + (1-d)^2),
\end{equation}
which has $D_1 = (1-d, 2(1-d))$ of order~$4$, and $E_1 = 2D_1 = (0,0)$
of order~$2$.

For any point~$D$ on~$\calC$, we let
$[k_1, k_2] = \bigl[ k_1(D), k_2(D) \bigr]
  \in \Projective^1$ denote the projective
$x$-coordinate, so that $x(D) = k_2(D)/k_1(D)$.
Then addition by~$E_1$ induces $[k_1,k_2] \mapsto [ k_2, (1-d)^2 k_1 ]$;
this can be described by the projective linear transformation
\begin{equation}\label{eqnaddE1elliptic}
\begin{pmatrix}
k_1\\
k_2\\
\end{pmatrix}
\mapsto
\begin{pmatrix}
0   & 1 \\
(1-d)^2 & 0 \\
\end{pmatrix}
\begin{pmatrix}
k_1\\
k_2\\
\end{pmatrix},
\end{equation}
and this matrix has eigenvalues $1-d$ and $d-1$ with eigenvectors
$\Bigl( \genfrac{}{}{0pt}{}{1}{1-d} \Bigr)$ 
and $\Bigl( \genfrac{}{}{0pt}{}{-1}{1-d} \Bigr)$, respectively.
We can perform a change of basis which diagonalises the
effect of addition by~$E_1$, namely:
\begin{equation}\label{diagelliptic1}
\begin{pmatrix}
k_1\\
k_2\\
\end{pmatrix}
=
\begin{pmatrix}
1 & -1 \\
1-d &  1-d \\
\end{pmatrix}
\begin{pmatrix}
l_1\\
l_2\\
\end{pmatrix}.
\end{equation}
Now let~$E_0$ denote the identity element, let~$E_1$ be as above,
and let $E_2,E_3$ be the other points of order~$2$ on~$\calC$.
For any point~$D$ on~$\calC$, let $l(D)$ denote
$[l_1(D), l_2(D)] \in \Projective^1$.
Then: $l(E_0) = [1,1]$, $l(E_1) = [1,-1]$, $l(E_2) = [-\sqrt{d},1]$
and $l(E_3) = [\sqrt{d},1]$. Addition by these points
maps
$[l_1,l_2]$
to, respectively: $[l_1,l_2]$, $[l_1,-l_2]$, $[-\sqrt{d}\ l_2, l_1]$
and $[\sqrt{d}\ l_2, l_1]$. Further, if~$D_1$ is the point of order~$4$
given above, then $l(D_1) = [1,0]$.

Recall the standard result (see, for example, Definition~2.1 
in~\cite{FlynnHeights}) that, 
for points $D,E$ on~$\calC$,
the $2\times 2$ matrix
$\bigl( k_i(D+E) k_j(D-E) + k_j(D+E) k_i(D-E) \bigr)$ is
projectively equal to a $2\times 2$ matrix of forms
which are biquadratic in $[ k_1(D), k_2(D) ]$ and $[ k_1(E), k_2(E) ]$.
If we perform the linear change in coordinates to the $l$-coordinates,
we see that these have a particularly simple form. Indeed,
if we let $[u_1,u_2] = [l_1(D),l_2(D)]$ and $[v_1,v_2] = [l_1(E),l_2(E)]$
and define the $B_{ij} = B_{ij}( [u_1,u_2], [v_1,v_2] )$ by:
\begin{equation}\label{ellipticbiquad}
\begin{split}
\Bigl( B_{ij} \Bigr) = 
\begin{pmatrix}
-d u_1^2 v_2^2 - d u_2^2 v_1^2 + d u_2^2 v_2^2 + u_1^2 v_1^2&
(1-d) u_1 u_2 v_1 v_2\\
(1-d) u_1 u_2 v_1 v_2&
-d u_2^2 v_2^2 - u_1^2 v_1^2 + u_1^2 v_2^2 + u_2^2 v_1^2
\end{pmatrix}
\end{split}
\end{equation}
then the $2 \times 2$ matrices $\bigl( B_{ij} \bigr)$
and $\bigl( l_i(D+E) l_j(D-E) + l_j(D+E) l_i(D-E) \bigr)$
are projectively equal.

We now embed our elliptic curve~$\calC$ into 
$\Projective^1 \times \Projective^1$, using the embedding:
\begin{equation}\label{ellipticP1xP1}
D \mapsto
\Bigl( [u_1,u_2], [y_1,y_2] \Bigr)
= \Bigl( \bigl[ l_1(D),l_2(D) \bigr], 
         \bigl[ l_1(D+D_1),l_2(D+D_1) \bigr] \Bigr).
\end{equation}

Since $l(D_1) = [1,0]$, we can substitute $v_1=1, v_2=0$
into~(\ref{ellipticbiquad}) to see that
the matrix
$\bigl( l_i(D+D_1) l_j(D-D_1) + l_j(D+D_1) l_i(D-D_1) \bigr)$
is
projectively
the same as
\begin{equation}\label{ellipticbiquadD1}
\begin{split}
\begin{pmatrix}
- d u_2^2 + u_1^2&
0\\
0&
- u_1^2 + u_2^2
\end{pmatrix}.
\end{split}
\end{equation}
Furthermore, $\bigl( l_i(D+D_1) l_j(D-D_1) + l_j(D+D_1) l_i(D-D_1) \bigr)$
is the same as 
$\bigl( l_i(D+D_1) l_j(D+D_1+E_1) + l_j(D+D_1) l_i(D+D_1+E_1) \bigr)$.
If we use that $y_i = l_i(D+D_1)$ and the fact that addition by~$E_1$
induces
$[l_1,l_2] \mapsto [l_1,-l_2]$, we see that our matrix is
also projectively equal to
\begin{equation}\label{ellipticbiquadD1var}
\begin{pmatrix}
2 y_1^2 &
0\\
0&
-2 y_2^2
\end{pmatrix}.
\end{equation}
Since~(\ref{ellipticbiquadD1}) and (\ref{ellipticbiquadD1var})
are projectively equal, we see that
\begin{equation}\label{edwardsproj}
(- d u_2^2 + u_1^2) (-y_2^2) = (- u_1^2 + u_2^2) y_1^2.
\end{equation}
If we now define the affine variables $U = u_2/u_1$
and $Y = y_2/y_1$ then $-(- d U^2 + 1) Y^2 = - 1 + U^2$,
giving $U^2 + Y^2 = 1 + d U^2 Y^2$, which is the equation
of the Edwards curve.

By an analog of Theorem~\ref{PQAddition}, the matrix
$\bigl( l_i(D+E) l_j(D-E+D_1) \bigr)_{i,j}$ is projectively the same as a
matrix whose entries are linear combinations of the terms of the form
$l_{i_1}(D) l_{i_2}(E) l_{i_3}(D+D_1) l_{i_4}(E+D_1)$.
If we let, for $i \in \{ 1,2 \}$,
\begin{equation}\label{ellipticuyvz}
u_i = l_i(D),\ 
y_i = l_i(D+D_1),\
v_i = l_i(E),\
z_i = l_i(E+D_1),
\end{equation}
then we may regard $\bigl( [u_1,u_2], [y_1,y_2] \bigr)$
and $\bigl( [v_1,v_2], [z_1,z_2] \bigr)$ as two arbitrary points
on the $\Projective^1 \times \Projective^1$ embedding of our curve.
We thus know that there is a matrix~$\bigl( A_{ij} \bigr)$,
where each 
$A_{ij} = A_{ij}\bigl( ([u_1,u_2],[y_1,y_2]),([v_1,v_2],[z_1,z_2]) \bigr)$
is a linear combination of terms of the
form $u_{i_1} v_{i_2} y_{i_3} z_{i_4}$, with the property
that $\bigl( A_{ij} \bigr) = \bigl( l_i(D+E) l_j(D-E+D_1) \bigr)$.
Each~$A_{ij}$ has~$16$ coefficients, and so there are~$64$
coefficients to be found. These can be determined just from
the linear equations in the coefficients arising from the cases when:
(i) $D$ is general and $E$ is any of the $2$-torsion points,
(ii) $E$ is general and $D$ is any of the $2$-torsion points,
and (iii) $D = E = D_1$. This gives
\begin{equation}\label{Aelliptic}
\Bigl( A_{ij} \Bigr)
=
\begin{pmatrix}
 u_1 v_1 y_1 z_1 - d u_2 v_2 y_2 z_2 &
 u_1 v_2 y_2 z_1 - u_2 v_1 y_1 z_2 \\
 -u_1 v_1 y_2 z_2 + u_2 v_2 y_1 z_1 &
 -u_1 v_2 y_1 z_2 + u_2 v_1 y_2 z_1
\end{pmatrix}.
\end{equation}
We see that any column of $\bigl( A_{ij} \bigr)$ gives 
the $u$-coordinates of~$D+E$.

There should also be a matrix
$\bigl( J_{ij} \bigr) = \bigl( l_i(D+E+D_1) l_j(D-E) \bigr)$,
and indeed we see that
\begin{equation}\label{ellipticJA}
\begin{split}
&J_{ij}\bigl( ([u_1,u_2],[y_1,y_2]),([v_1,v_2],[z_1,z_2]) \bigr)\\
&= 
J_{ij}\bigl(([l_1(D),l_2(D)],[l_1(D+D_1),l_2(D+D_1)]),\\
&\ \ \ \ \ \ \ \ \ \ ([l_1(E),l_2(E)],[l_1(E+D_1),l_2(E+D_1)])\bigr)\\
 &= l_i(D+E+D_1) l_j(D-E)\\
 &= l_j(D-E) l_i(D+E+D_1)\\
 &= 
A_{ji}\bigl( ([l_1(D),l_2(D)],[l_1(D+D_1),l_2(D+D_1)]),\\
&\ \ \ \ \ \ \ \ \ \ ([l_1(-E),l_2(-E)],[l_1(-E+D_1),l_2(-E+D_1)])\bigr)\\
 &= 
A_{ji}\bigl(([l_1(D),l_2(D)],[l_1(D+D_1),l_2(D+D_1)]),\\
&\ \ \ \ \ \ \ \ \ \ ([l_1(E),l_2(E)],[l_1(E-D_1),l_2(E-D_1)])\bigr)\\
 &= 
A_{ji}\bigl(([l_1(D),l_2(D)],[l_1(D+D_1),l_2(D+D_1)]),\\
&\ \ \ \ \ \ \ \ \ \    
([l_1(E),l_2(E)],[l_1(E+E_1+D_1),l_2(E+E_1+D_1)])\bigr)\\
 &= 
A_{ji}\bigl(([l_1(D),l_2(D)],[l_1(D+D_1),l_2(D+D_1)]),\\
&\ \ \ \ \ \ \ \ \ \ ([l_1(E),l_2(E)],[l_1(E+D_1),-l_2(E+D_1)])\bigr)\\
&= A_{ji}\bigl(([u_1,u_2],[y_1,y_2]),([v_1,v_2],[z_1,-z_2])\bigr).\\
\end{split}
\end{equation}
So, if we define
\begin{equation}\label{ellipticJdefn}
J_{ij}\bigl(([u_1,u_2],[y_1,y_2]),([v_1,v_2],[z_1,z_2])\bigr)
	= A_{ji}\bigl(([u_1,u_2],[y_1,y_2]),([v_1,v_2],[z_1,-z_2])\bigr)
\end{equation}
this gives
\begin{equation}\label{Jelliptic}
\Bigl( J_{ij} \Bigr)
=
\begin{pmatrix}
 u_1 v_1 y_1 z_1 + d u_2 v_2 y_2 z_2 &
 u_1 v_1 y_2 z_2 + u_2 v_2 y_1 z_1 \\
 u_1 v_2 y_2 z_1 + u_2 v_1 y_1 z_2 &
 u_1 v_2 y_1 z_2 + u_2 v_1 y_2 z_1
\end{pmatrix}.
\end{equation}
We see that $\bigl( J_{ij} \bigr) = \bigl( l_i(D+E+D_1) l_j(D-E) \bigr)$
and that any column of $\bigl( J_{ij} \bigr)$ gives 
the $y$-coordinates of~$D+E$.

We now have a description of the group law for our 
$\Projective^1 \times \Projective^1$ embedding; namely,
$D+E$ is given by any column of $\bigl( A_{ij} \bigr)$ 
together with any column of $\bigl( J_{ij} \bigr)$.
If we write our original points in affine coordinates
$(U,Y)$ and $(V,Z)$, where $U = u_2/u_1$, $Y = y_2/y_1$,
$V = v_2/v_1$ and $Z = z_2/z_1$, then the sum could be
given by any of
\begin{equation}\label{ellipticsum}
\begin{split}
\Bigl( A_{21}/A_{11}, J_{21}/J_{11} \Bigr)
&= \Bigl( (- Y Z + U V)/(1 - d U V Y Z), (V Y + U Z)/(1 + d U V Y Z) \Bigr),\\
\Bigl( A_{21}/A_{11}, J_{22}/J_{12} \Bigr)
&= \Bigl( (- Y Z + U V)/(1 - d U V Y Z), (V Z + U Y)/(Y Z + U V) \Bigr),\\
\Bigl( A_{22}/A_{12}, J_{21}/J_{11} \Bigr)
&= \Bigl( (- V Z + U Y)/(V Y - U Z), (V Y + U Z)/(1 + d U V Y Z) \Bigr),\\
\Bigl( A_{22}/A_{11}, J_{22}/J_{12} \Bigr)
&= \Bigl( (- V Z + U Y)/(V Y - U Z), (V Z + U Y)/(Y Z + U V) \Bigr),\\
\end{split}
\end{equation}
which the same as the group law in~\cite{BernsteinLangeEC}, after
taking account of the fact that we are taking~$(1,0)$ as the
identity, whereas~$(0,1)$ is taken to be the identity element
in~\cite{BernsteinLangeEC}. We can think of the columns of the
matrices $\bigl( A_{ij} \bigr)$ and $\bigl( J_{ij} \bigr)$
as giving all variations of the group law. 

We finally note that, if $A_{11} = A_{21} = 0$ then
\begin{equation}\label{ellipticnondegen}
y_1 z_1 A_{11} + d y_2 z_2 A_{21}
= u_1 v_1 (y_1^2 z_1^2 - d y_2^2 z_2^2 ) = 0.
\end{equation}
Suppose that~$d$ is non-square over the ground field~$K$.
Since $[u_1,u_2],[y_1,y_2] \in \Projective^1$ and
satisfy~(\ref{edwardsproj}), it follows that~$u_1$ and~$y_1$
are nonzero (which also has the consequence that the
the affine coordinate~$U = u_2/u_1$ is always well defined,
and so the above affine variety has no points at infinity). 
Similarly,~$v_1$ and~$z_1$ are nonzero.
This is inconsistent with~(\ref{ellipticnondegen}), and
so we see that, provided~$d$ is non-square, we can never
have $A_{11} = A_{21} = 0$; similarly, we can never have
any of $A_{12} = A_{22} = 0$, $J_{11} = J_{21} = 0$
or $J_{12} = J_{22} = 0$, and so there are no exceptional
cases for any of the above versions of the group law.

We can see how the existence of an always nonzero coordinate
and the existence of a universal group law follow from
the fact that, on our original elliptic curve~(\ref{eqnforellipticC}),
our condition that~$d$ is non-square forces the point
$( d-1, 2(d-1)\sqrt{d} )$ to be not defined over the ground field;
this in turn forces~$u_1$ to be always nonzero for all points
$\bigl( [u_1,u_2], [y_1,y_2] \bigr)$ defined over the ground field;
that is to say, $l_1(D)$ is nonzero for any~$D$ defined over
the ground field. For any two points $D,E$, 
since $\bigl( A_{ij} \bigr) = \bigl( l_i(D+E) l_j(D-E+D_1) \bigr)$
it follows that
$\bigl[ A_{11}, A_{21} \bigr]
= \bigl[ l_1(D+E) l_1(D-E+D_1), l_2(D+E) l_1(D-E+D_1) \bigr]
 = [l_1(D+E), l_2(D+E)]$
cannot have both entries~$0$, since $l_1(D-E+D_1)$ is nonzero.
Similarly $J_{11},J_{21}$ are not both zero.
Yet another interpretation is to note that the intersection
of our elliptic curve with the condition~$u_1 = 0$
is by~\eqref{edwardsproj} the pair of points
$\{([0,1],[1,\pm\sqrt{d}])\} \subset \Projective^1 \times \Projective^1$.
Our condition on $d$ ensures
that this pair of points is collectively but not individually
defined over the ground field. This serves to ensure that
each of $[A_{11},A_{21}]$ and $[J_{11},J_{21}]$ is a well defined projective
point with a nonzero coordinate, and we have a universal group law.

Of course, even if~$d$ is square, our matrices~$A,J$
in ~(\ref{Aelliptic}) and~(\ref{Jelliptic}) still give
a description of the group law which covers all possibilities;
it is merely that there will not be a specified column
which covers all possibilities; rather, one will need to
use one column for some cases and another column for
others, or more generally one can use a linear combination of columns, 
which gives the same projective point.

We note in passing that, since negation has the effect:
$\bigl( [u_1,u_2], [y_1,y_2] \bigr) \mapsto
\bigl( [u_1,u_2], [y_1,-y_2] \bigr)$, we may replace~$y_2$
with~$\sqrt{\kk} y_2$, for any nonsquare $\kk \in K$, and the
result will still be defined over~$K$. This is due to the
fact that the nontrivial Galois action $\sqrt{\kk} \mapsto -\sqrt{\kk}$
has the same effect as negation, so that the variety is taken
to itself under this action. After replacing~$y_2$ with~$\sqrt{\kk} y_2$, 
we obtain the affine model $U^2 + \kk Y^2 = 1 + d \kk U^2 Y^2$, which
is a quadratic twist of the original curve, and is birationally
equivalent over~$K$ to the twisted Edwards curve, described
in~\cite{BBJLPtwisted}.

\section{An analog for Jacobians of genus 2 curves}
\label{section4}

In this section, we shall derive our $\Projective^3 \times \Projective^3$
embedding of the Jacobian variety of a genus~$2$ curve,
giving explicitly a set of defining equations for the variety
in Theorem~\ref{theoremdefeqns}
and its group law in Theorem~\ref{theoremGroupLawGenus2}.

\subsection*{The standard embedding of the Kummer surface.}
We shall first describe a standard embedding of the Kummer surface,
which will be given in~(\ref{kummercoords}).
For a general curve of genus~$2$
\begin{equation}\label{genus2general}
y^2 = f_6 x^6 + f_5 x^5 + f_4 x^4 + f_3 x^3 + f_2 x^2 + f_1 x + f_0,
\end{equation}
defined over a ground field~$K$ (not of characteristic~$2$).
we may represent elements of the Jacobian variety by
$\{ (x_1,y_1), (x_2,y_2) \}$, as a shorthand for the divisor
class of $(x_1,y_1) + (x_2,y_2) - \infty^+ - \infty^-$,
where~$\infty~^+$ and~$\infty^-$ denote the points on the
non-singular curve that lie over the singular point at infinity.
The role of the projective $x$-coordinate in the previous
section will be performed by the Kummer surface, which has
an embedding (see p.18 of~\cite{CasselsFlynnBook})
in $\Projective^3$ given by $[k_1,k_2,k_3,k_4]$, where 
\begin{equation}\label{kummercoords}
k_1 = 1,\  k_2 = x_1 + x_2 ,\  k_3 = x_1 x_2,\
k_4 = (F_0 (x_1 ,x_2 )-2y_1 y_2 )/(x_1 -x_2 )^2,
\end{equation}
\noindent and where
\begin{equation}\label{F0defn}
\begin{split}
F_0 (x_1 ,x_2 )= &2 f_0 + f_1 (x_1 +x_2 ) + 2 f_2 (x_1 x_2 )
+ f_3 (x_1 x_2) (x_1 +x_2 ) \\
&+ 2 f_4 (x_1 x_2 )^2 + f_5 (x_1 x_2 )^2 (x_1 + x_2 )
+2 f_6 (x_1 x_2 )^3.
\end{split}
\end{equation}
The defining equation of the Kummer surface is given by
\begin {equation}\label {kummerequation}
 R(k_1,k_2,k_3) k_4^2 + S(k_1,k_2,k_3) k_4 + T(k_1,k_2,k_3) = 0,
\end {equation}
where $R$, $S$, $T$ are given by:
\begin{eqnarray*}
R(k_1,k_2,k_3) &=& k_2^2 - 4 k_1 k_3, \\
S(k_1,k_2,k_3) &=&
-2 \left( 2 k_1^3 f_0 + k_1^2 k_2 f_1 + 2 k_1^2 k_3 f_2 + k_1 k_2 k_3 f_3
+ 2 k_1 k_3^2 f_4  \right. \\
  &~& \left. + k_2 k_3^2 f_5 + 2 k_3^3 f_6 \right), \\
T(k_1,k_2,k_3) &=&
- 4 k_1^4 f_0 f_2 + k_1^4 f_1^2 - 4 k_1^3 k_2 f_0 f_3 - 2 k_1^3 k_3 f_1 f_3
- 4 k_1^2 k_2^2 f_0 f_4 \\
&~& + 4 k_1^2 k_2 k_3 f_0 f_5 - 4 k_1^2 k_2 k_3 f_1 f_4 - 4 k_1^2 k_3^2 f_0 f_6
+ 2 k_1^2 k_3^2 f_1 f_5 \\
&~& - 4 k_1^2 k_3^2 f_2 f_4  + k_1^2 k_3^2 f_3^2 - 4 k_1 k_2^3 f_0 f_5
+ 8 k_1 k_2^2 k_3 f_0 f_6 - 4 k_2^4 f_0 f_6 \\
&~& - 4 k_1 k_2^2 k_3 f_1 f_5 + 4 k_1 k_2 k_3^2 f_1 f_6
- 4 k_1 k_2 k_3^2 f_2 f_5 - 2 k_1 k_3^3 f_3 f_5 \\
&~& - 4 k_2^3 k_3 f_1 f_6 -4 k_2^2 k_3^2 f_2 f_6
-4 k_2 k_3^3 f_3 f_6 - 4 k_3^4 f_4 f_6 + k_3^4 f_5^2.
\end{eqnarray*}
Our aim now is to follow the style of the previous section and derive
a $\Projective^3 \times \Projective^3$ embedding of the Jacobian variety, 
arising from the images on the Kummer surface of the point~$D$ 
(on the Jacobian) and $D + D_1$, where~$D_1$ is
a fixed point of order~$4$. 

\subsection*{Forcing the existence of two independent points of order~4.} 
We shall next describe a general genus~$2$ curve~$\calC$ whose Jacobian has
independent points~$D_1, D_2$ of order~$4$; the model for such a curve
wll be given in~(\ref{eqnforgenus2C}). Then $E_1 = 2 D_1$ 
and $E_2 = 2 D_2$ will have order~$2$. We shall also insist
that $\langle E_1, E_2 \rangle$ is a maximal isotropic subgroup
which is the kernel of a Richelot isogeny, as described on p.89
of~\cite{CasselsFlynnBook}. These requirements force the Jacobian
to be isogenous to the Jacobian of a curve
$\calC'$ with full $2$-torsion.
After a linear change in variable, the curve $\calC'$ is given by
$y^2 = x (x-1) (x-\alpha) (x-\beta) (x-\gamma)$,
which is $y^2 = H_1(x) H_2(x) H_3(x)$, where
$H_1(x) = x$, $H_2(x) = (x-1)(x-\alpha)$ and $H_3(x) = (x-\beta)(x-\gamma)$.
Suppose that the kernel of the dual isogeny consists
of the identity, $\{ \infty, (0,0) \}$, $\{ (1,0), (\alpha,0) \}$
and $\{ (\beta,0), (\gamma, 0) \}$. The points $D_1, D_2$
of order~$4$ on
the Jacobian of~$\calC$ must map to points of order~$2$ outside
the above mentioned kernel;
say that $D_1$ and $D_2$ map respectively to $\{ (1,0), \infty \}$ and
$\{(\beta,0),\infty\}$. 
As an aside, 
we note that $\{ (1,0), \infty \}$ and $\{(\beta,0),\infty\}$ have a
nontrivial Weil pairing, so the original points $D_1,D_2$ do not
generate an isotropic subgroup of the $4$-torsion in the Jacobian of
$\calC$, even though their doubles $E_1, E_2$ do generate an isotropic
subgroup of the $2$-torsion.


From the standard maps on p.106 of~\cite{CasselsFlynnBook},
the rationality of $D_1$ and $D_2$
forces all of $1$, $(1-\beta)(1-\gamma)$, $\beta$ and
$(\beta-1)(\beta-\alpha)$ to be squares. This is solved by:
\begin{equation}\label{alphabetagammasoln}
\alpha = \ua^2 + b^2 - \ua^2 b^2,\ \beta = b^2,\ 
\gamma = b^2 \uc^2 - \uc^2 + 1.
\end{equation}
We shall now increase the generality by only requiring that
$E_1,E_2,D_1$ are defined over the ground field, but not necessarily~$D_2$.
We let $a = \ua^2$, $c = \uc^2$, so that
\begin{equation}\label{alphabetagammasolnvar}
\alpha = a + b^2 - a b^2,\ \beta = b^2,\ \gamma = b^2 c - c + 1,
\end{equation}
where $a,b,c$ are arbitrary members of some ground field~$K$,
which is not of characteristic~$2$.
If we define the $h_{ij}$ by $H_j = H_j(x) = h_{j2} x^2 + h_{j1} x + h_{j0}$,
then we recall that the general formula for the isogenous curve~$\calC$
is
\begin{equation}\label{richeloteqn}
y^2 = \Delta \bigl( H_2' H_3 - H_2 H_3' \bigr)
             \bigl( H_3' H_1 - H_3 H_1' \bigr)
             \bigl( H_1' H_2 - H_1 H_2' \bigr),
\end{equation}
where $\Delta = \hbox{det} ( h_{ij} )$. If we apply this to our
specific $H_1,H_2,H_3$, we see that our original curve (after
absorbing the factor $(b-1)^2 (b+1)^2/ac$ into $y^2$ by a quadratic
twist of the $y$-coordinate) has the form
\begin{equation}\label{eqnforgenus2C}
\calC : y^2 =  g\bigl( (f+1) x^2 - 2 g x + b^2 f - d e \bigr)
               \bigl( x^2 - b^2 d \bigr)
               \bigl( x^2 + e \bigr),
\end{equation}
\noindent where
\begin{equation}\label{defndefg}
\begin{split}
d &= b^2 c - c + 1,\\
e &= a b^2 - a - b^2,\\
f &= a + c - 1,\\
g &= b^2 c + a.
\end{split}
\end{equation}
We
require
that the discriminant of~$\calC$ is nonzero, which
is equivalent to the condition that
$2abcdefg(a-1)(b^2-1)(c-1)$ is nonzero.
Let $\J = \hbox{Jac}(\calC)$, the Jacobian of~$\calC$.
Let~$E_0$ be the identity element in~$\J$, and let
$E_1 = \{ (g + (b^2-1) \sqrt{acf})/(a+c), 0),
          (g - (b^2-1) \sqrt{acf})/(a+c), 0) \}$,
$E_2 = \{ (b \sqrt{d}, 0), (-b\sqrt{d},0) \}$
and $E_3 = \{ (\sqrt{-e},0), (-\sqrt{-e},0) \}$
be the points of order~$2$ on~$\J$ corresponding to the three 
quadratic factors given on the right hand side of~(\ref{eqnforgenus2C}).
These are all defined over the ground field~$K$.
For any $D \in \J$, let $k(D) = [k_1(D),k_2(D),k_3(D),k_4(D)]$
denote the image of~$D$ in the above embedding of the Kummer surface.
Let~$D_1,D_2,D_3$ be points of order~$4$ such that 
$2 D_1 = E_1$, $2 D_2 = E_2$ and $2 D_3 = E_3$ (chosen so
that $D_3 = D_1 + D_2$).
It can be checked directly (see the file~\cite{FlynnMaple}) that~$D_1$
is defined over the ground field~$K$. Furthermore, $D_2,D_3$ 
are defined over $K(\sqrt{a},\sqrt{c})$.

\subsection*{Diagonalising addition by the order two elements $E_1,E_1$.}
We know (see p.22 of~\cite{CasselsFlynnBook})
that addition by any point of order~$2$ gives a linear map
on the Kummer surface. We simultaneously diagonalise addition
by~$E_1$ and~$E_2$ (as described in the file~\cite{FlynnMaple})
using the following linear change of basis for the Kummer surface.
\begin{equation}\label{kummerchangebasis}
\begin{pmatrix}
l_1\\
l_2\\
l_3\\
l_4
\end{pmatrix}
=
Q
\begin{pmatrix}
k_1\\
k_2\\
k_3\\
k_4
\end{pmatrix}
\end{equation}
\noindent where $Q = (Q_{ij})$ is the $4\times 4$ matrix with entries:
\begin{equation}\label{entriesofQ}
\begin{split}
 Q_{11} &= 2 g b^2 e (b^4 c^2-2 b^2 c^2+2 b^2 c+c^2+a-c),\\
 Q_{12} &= -2 g^2 b^2 e,\\
 Q_{13} &= 2 g (a^2 b^2+a b^2 c-b^4 c-a^2-2 a b^2-a c+a),\\
 Q_{14} &= 1,\\
 Q_{21} &= -2 g d (a^2 b^4-2 a^2 b^2-a b^4+b^4 c+a^2+2 a b^2),\\
 Q_{22} &= 2 g^2 d,\\
 Q_{23} &=
   -2 g (a b^4 c+b^4 c^2-a b^2 c-b^4 c-b^2 c^2+2 b^2 c+a),\\
 Q_{24} &= 1,\\
 Q_{31} &= 2 g b^2 (a^2 b^4 c+a b^4 c^2-2 a^2 b^2 c-2 a b^4 c
    -2 a b^2 c^2\\
        &\ \ \ \ \ \ \ \ \ \ \
         +a^2 c+4 a b^2 c+a c^2+b^4 c-2 a c+a),\\
 Q_{32} &= -2 b^2 g^2,\\
 Q_{33} &= 2 g (a b^4 c-2 a b^2 c+a b^2+a c+b^2 c),\\
 Q_{34} &= -1,\\
 Q_{41} &= -2 g d e (b^4 c+a),\\
 Q_{42} &= 2 g^2 d e,\\
 Q_{43} &= -2 g (-b^4 c^2+a^2 b^2+b^2 c^2-a^2-a b^2-b^2 c),\\
 Q_{44} &= -1.
\end{split}
\end{equation}
For any~$D\in \J$, we let 
$l(D) = [ l_1(D), l_2(D), l_3(D), l_4(D) ]$.
We first note that, after this linear change in coordinates,
the equation of the Kummer surface is considerably simplified:
\begin{equation}\label{kummereqnl}
\begin{split}
&\bigl( b g (a c l_1 l_2-f l_3 l_4) \bigr)^2 =\\
&\  \bigl( b^2 (a c (f l_1^2-e l_2^2)+c e f l_3^2-a f l_4^2) \bigr)
    \bigl( a c (d l_1^2+b^2 f l_2^2)-f (a d l_3^2+b^2 c l_4^2) \bigr).
\end{split}
\end{equation}
As above, we let~$E_0$ denote the identity element, and 
let~$E_1,E_2,E_3$ be the points of order~$2$ above.
Then: $l(E_0) = [1, 1, -1, -1]$, $l(E_1) = [1, 1, 1, 1]$, 
$l(E_2) = [1, -1, -1, 1]$ and and $l(E_3) = [1, -1, 1, -1]$. 
Addition by these points
maps
a general $[l_1,l_2,l_3,l_4]$
to, respectively: $[l_1,l_2,l_3,l_4]$, 
$[l_1,l_2,-l_3,-l_4]$, $[l_1,-l_2,l_3,-l_4]$
and $[l_1,-l_2,-l_3,l_4]$. Further, if~$D_1,D_2,D_3$ are the
points of order~$4$ given above, 
then $l(D_1) = [b, 1, 0, 0]$, $l(D_2) = [0, 1, 0, -\sqrt{a}]$
and $l(D_3) = [1, 0, 0, \sqrt{c}]$.

\subsection*{Simplified biquadratic forms on the Kummer surface.}
There is a result on Jacobians of genus two curves analogous to that
mentioned previously for elliptic curves (see Theorem~3.4.1
of~\cite{CasselsFlynnBook}) that, for points $D,E$ on~$\J$,
the $4\times 4$ matrix
$\bigl( k_i(D+E) k_j(D-E) + k_j(D+E) k_i(D-E) \bigr)$ is
projectively equal to a $4\times 4$ matrix of forms
which are biquadratic in $[ k_1(D), k_2(D), k_3(D), k_4(D) ]$ 
and $[ k_1(E), k_2(E), k_3(E), k_4(E) ]$.
If we perform the linear change in coordinates to the $l$-coordinates,
we see that these have a simpler form. 

Indeed, if we let $[u_1,u_2,u_3,u_4] = [l_1(D),l_2(D),l_3(D),l_4(D)]$ 
and $[v_1,v_2,v_3,v_4] = [l_1(E),l_2(E),l_3(E),l_4(E)]$
then there are biquadratic forms 
\begin{equation}\label{genus2biquad}
B_{ij} = B_{ij}\bigl( [u_1,u_2,u_3,u_4], [v_1,v_2,v_3,v_4] \bigr)
\end{equation}
such that the $4 \times 4$ matrices $\bigl( B_{ij} \bigr)$
and $\bigl( l_i(D+E) l_j(D-E) + l_j(D+E) l_i(D-E) \bigr)$
are projectively equal. The~$B_{ij}$ are derived and given
explicitly in the file~\cite{FlynnMaple}.

\subsection*{Embedding the Jacobian into $\Projective^3 \times \Projective^3$.}
We shall now embed our Jacobian~$\J$ into 
$\Projective^3 \times \Projective^3$, using the embedding:
\begin{equation}\label{genus2P3xP3}
\begin{split}
& D \mapsto \Bigl( [u_1,u_2,u_3,u_4], [y_1,y_2,y_3,y_4] \Bigr)\\
&= \Bigl( l(D), l(D+D_1) \Bigr)\\
&= \Bigl( [ l_1(D),l_2(D),l_3(D),l_4(D) ], 
	  [ l_1(D+D_1),l_2(D+D_1),l_3(D+D_1),l_4(D+D_1) ] \Bigr).
\end{split}
\end{equation}
Our aim for the rest of this section will be to describe
the defining equations and group law for this embedding.

We note that there are the following linear maps on this embedding
\begin{equation}\label{mapsonembedding}
\begin{split}
& D \mapsto -D \\
& : \bigl( [u_1,u_2,u_3,u_4], [y_1,y_2,y_3,y_4] \bigr)
\mapsto \bigl( [u_1,u_2,u_3,u_4], [y_1,y_2,-y_3,-y_4] \bigr),\\
& D \mapsto D + E_1\\
& : \bigl( [u_1,u_2,u_3,u_4], [y_1,y_2,y_3,y_4] \bigr)
\mapsto \bigl( [u_1,u_2,-u_3,-u_4], [y_1,y_2,-y_3,-y_4] \bigr),\\
& D \mapsto D + E_2\\
& : \bigl( [u_1,u_2,u_3,u_4], [y_1,y_2,y_3,y_4] \bigr)
\mapsto \bigl( [u_1,-u_2,u_3,-u_4], [y_1,-y_2,y_3,-y_4] \bigr),\\
& D \mapsto D + E_3\\
& : \bigl( [u_1,u_2,u_3,u_4], [ y_1,y_2,y_3,y_4] \bigr)
\mapsto \bigl( [u_1,-u_2,-u_3,u_4], [ y_1,-y_2,-y_3,y_4] \bigr),\\
& D \mapsto D + D_1\\
& : \bigl( [u_1,u_2,u_3,u_4], [y_1,y_2,y_3,y_4] \bigr)
\mapsto \bigl( [y_1,y_2,y_3,y_4], [u_1,u_2,-u_3,-u_4] \bigr),\\
& D \mapsto D - D_1\\
& : \bigl( [u_1,u_2,u_3,u_4], [ y_1,y_2,y_3,y_4] \bigr)
\mapsto \bigl( [ y_1,y_2,-y_3,-y_4], [u_1,u_2,u_3,u_4] \bigr),\\
\end{split}
\end{equation}
from which it follows that
\begin{equation}\label{swapmap}
\begin{split}
& D \mapsto - D - D_1\\
& : \bigl( [u_1,u_2,u_3,u_4], [y_1,y_2,y_3,y_4] \bigr)
\mapsto \bigl( [y_1,y_2,y_3,y_4], [u_1,u_2,u_3,u_4] \bigr),\\
\end{split}
\end{equation}
which swaps the~$u_i$ and~$y_i$. The above mappings generate a group
of mappings which is isomorphic to~$C_2 \times D_8$.

\subsection*{Using the biquadratic forms to obtain some of
the defining equations.}
Since $l(D_1) = [b, 1, 0, 0]$, we can substitute $v_1=b, v_2=1, v_3=0, v_4=0$
into the~$B_{ij}$,
to give $M = \bigl( l_i(D+D_1) l_j(D-D_1) + l_j(D+D_1) l_i(D-D_1) \bigr)$,
when we find that (as derived in the file~\cite{FlynnMaple}) 
\begin{equation}\label{genus2biquadD1}
\begin{split}
\bigl[ M_{1,1}, & M_{1,2}, M_{2,2}, M_{3,3}, M_{3,4}, M_{4,4} \bigr]\\
= \bigl[ &b^2 (a c (-f u_1^2+e u_2^2)-c e f u_3^2+a f u_4^2),\\
         &-b g (a c u_1 u_2-f u_3 u_4),\\
         &-a c (d u_1^2+b^2 f u_2^2)+f (a d u_3^2+b^2 c u_4^2),\\
         &a b^2 c (c u_1^2+a u_2^2-f u_3^2-u_4^2),\\
         &a b c g (u_1 u_2-u_3 u_4),\\
	 &a c (a d u_1^2-b^2 c e u_2^2+d e u_3^2-b^2 f u_4^2) \bigr]
\end{split}
\end{equation}
and all other entries are zero.

Furthermore, $\bigl( l_i(D+D_1) l_j(D-D_1) + l_j(D+D_1) l_i(D-D_1) \bigr)$
is the same as 
$\bigl( l_i(D+D_1) l_j(D+D_1+E_1) + l_j(D+D_1) l_i(D+D_1+E_1) \bigr)$.
If we now use that $y_i = l_i(D+D_1)$ and the fact that addition by~$E_1$
induces $[l_1,l_2,l_3,l_4] \mapsto [l_1,l_2,-l_3,-l_4]$, 
we see that our matrix is
also projectively equal to~$N$, which satisfies
\begin{equation}\label{genus2biquadD1var}
\begin{split}
&\bigl[ N_{1,1}, N_{1,2}, N_{2,2}, N_{3,3}, N_{3,4}, N_{4,4} \bigr]\\
&= \bigl[ y_1^2, y_1 y_2, y_2^2, -y_3^2, -y_3 y_4, -y_4^2 \bigr]. 
\end{split}
\end{equation}
Since~(\ref{genus2biquadD1}) and (\ref{genus2biquadD1var})
are projectively equal, we see that we now have a number of equations
satisfied by $\bigl( [u_1,u_2,u_3,u_4], [y_1,y_2,y_3,y_4] \bigr)$,
namely anything of the form
\begin{equation}\label{someeqns}
M_{i_1,j_1} N_{i_2,j_2} - M_{i_2,j_2} N_{i_1,j_1} = 0.
\end{equation}
for any $i_1,j_1,i_2,j_2 \in \{ 1,2,3,4 \}$.
For example, the following must be satisfied:
\begin{equation*}
a b c g (u_1 u_2-u_3 u_4) y_1 y_2 
= b g (a c u_1 u_2-f u_3 u_4) y_3 y_4.
\end{equation*}
However, we can see from
Theorem~\ref{generatorsInEachBidegree}
that, unlike the elliptic curves situation, we do not yet have
a complete set of defining equations. So, we shall now proceed
to the equations for the group law since it will turn out that
the group law equations will allow us to deduce the missing
defining equations of our variety.

\subsection*{The group law for our $\Projective^3 \times \Projective^3$
embedding.}
We can now derive the group law,
which will soon be described in Theorem~\ref{theoremGroupLawGenus2}. 
We know from Theorem~\ref{PQAddition} that the projective matrix
$\bigl( l_i(D+E) l_j(D-E+D_1) \bigr)$ is the same as a matrix whose entries 
are linear combinations of the terms of the form
$l_{i_1}(D) l_{i_2}(E) l_{i_3}(D+D_1) l_{i_4}(E+D_1)$.
If we let, for $i \in \{ 1,2,3,4 \}$,
\begin{equation}\label{genus2uyvz}
u_i = l_i(D),\ 
y_i = l_i(D+D_1),\
v_i = l_i(E),\
z_i = l_i(E+D_1),
\end{equation}
then we may regard $\bigl( [u_1,u_2,u_3,u_4], [y_1,y_2,y_3,y_4] \bigr)$
and $\bigl( [v_1,v_2,v_3,v_4], [z_1,z_2,z_3,z_4] \bigr)$ as two arbitrary 
points on the $\Projective^3 \times \Projective^3$ embedding of our Jacobian.
We know that there is a matrix~$\bigl( A_{ij} \bigr)$,
where each 
\begin{equation*}
A_{ij} = 
A_{ij}\Bigl( \bigl( [u_1,u_2,u_3,u_4], [y_1,y_2,y_3,y_4] \bigr),
    \bigl( [v_1,v_2,v_3,v_4], [z_1,z_2,z_3,z_4] \bigr) \Bigr)
\end{equation*}
is a linear combination of terms of the
form $u_{i_1} v_{i_2} y_{i_3} z_{i_4}$, with the property
that $\bigl( A_{ij} \bigr) = \bigl( l_i(D+E) l_j(D-E+D_1) \bigr)$.
These coefficients can be determined just from
the linear equations in the coefficients arising from considering
the effects of
(i)
translations by~$E_1,E_2,E_3$ and~$D_1$,
(ii)
swapping~$D$ and~$E$,
as well as the cases when:
(iii)
$D$ is general and $E$ is any of the $2$-torsion points,
(iv)
$E$ is general and $D$ is any of the $2$-torsion points,
and
(v)
$D = E = D_1$. These are derived in the file~\cite{FlynnMaple}
and give the following result.
\begin{equation}\label{Agenus2}
\begin{split}
A_{11} &= -a b c (u_1 y_1 (d v_1 z_1+b e v_2 z_2) 
        + b e u_2 y_2 (v_1 z_1-b v_2 z_2))\\
       &\ \  -b f (e u_3 y_3 (d v_3 z_3-b v_4 z_4) 
        - b u_4 y_4 (e v_3 z_3+b v_4 z_4)),\\
A_{12} &= a b^2 c f (u_1 y_2 (v_1 z_2-b v_2 z_1) 
        - u_2 y_1 (b v_1 z_2-v_2 z_1))\\
       &\ \  +b^2 f^2 (-u_3 y_4 (v_3 z_4-b v_4 z_3) 
        + u_4 y_3 (b v_3 z_4-v_4 z_3)),\\
A_{13} &= a b f (u_1 y_3 (d v_3 z_1-b v_4 z_2) 
        - u_3 y_1 (d v_1 z_3-b v_2 z_4))\\
       &\ \  +a b^2 f (-u_2 y_4 (v_3 z_1-b v_4 z_2) 
        + u_4 y_2 (v_1 z_3-b v_2 z_4)),\\
A_{14} &= b^2 c f (u_1 y_4 (b v_4 z_1+e v_3 z_2) 
        - u_4 y_1 (b v_1 z_4+e v_2 z_3))\\
       &\ \  +b^2 c e f (u_2 y_3 (v_4 z_1-b v_3 z_2) 
        - u_3 y_2 (v_1 z_4-b v_2 z_3)),\\
A_{22} &= a c (d u_1 y_1 (v_1 z_1-b v_2 z_2) 
        - b u_2 y_2 (d v_1 z_1+b e v_2 z_2))\\
       &\ \  +f (d u_3 y_3 (e v_3 z_3+b v_4 z_4) 
        + b u_4 y_4 (d v_3 z_3-b v_4 z_4)),\\
A_{23} &= b^2 c f (-u_1 y_4 (v_4 z_1-b v_3 z_2) 
        + u_4 y_1 (v_1 z_4-b v_2 z_3))\\
       &\ \  +b^2 c f (u_2 y_3 (b v_4 z_1+e v_3 z_2) 
        - u_3 y_2 (b v_1 z_4+e v_2 z_3)),\\
A_{24} &= a d f (-u_1 y_3 (v_3 z_1-b v_4 z_2) 
        + u_3 y_1 (v_1 z_3-b v_2 z_4))\\
       &\ \  +a b f (u_2 y_4 (d v_3 z_1-b v_4 z_2) 
        - u_4 y_2 (d v_1 z_3-b v_2 z_4)),\\
A_{33} &= a b c (u_1 y_1 (d v_3 z_3-b v_4 z_4) 
        + b u_2 y_2 (e v_3 z_3+b v_4 z_4))\\
       &\ \  -a b c (u_3 y_3 (d v_1 z_1+b e v_2 z_2) 
        - b u_4 y_4 (v_1 z_1-b v_2 z_2)),\\
A_{34} &= a b^2 c f (-u_1 y_2 (v_3 z_4-b v_4 z_3) 
        + u_2 y_1 (b v_3 z_4-v_4 z_3))\\
       &\ \  +a b^2 c f (u_3 y_4 (v_1 z_2-b v_2 z_1)
        - u_4 y_3 (b v_1 z_2-v_2 z_1)),\\
A_{44} &= a c (d u_1 y_1 (e v_3 z_3+b v_4 z_4) 
        - b e u_2 y_2 (d v_3 z_3-b v_4 z_4))\\
       &\ \  -a c (d e u_3 y_3 (v_1 z_1-b v_2 z_2) 
        + b u_4 y_4 (d v_1 z_1+b e v_2 z_2)),
\end{split}
\end{equation}
where $A_{ij}$ for $i > j$ are defined by
\begin{equation}\label{Aijremaining}
\begin{split}
&A_{ij}\Bigl( \bigl( [u_1,u_2,u_3,u_4], [y_1,y_2,y_3,y_4] \bigr), 
     \bigl( [v_1,v_2,v_3,v_4], [z_1,z_2,z_3,z_4] \bigr) \Bigr)\\
=
&A_{ji}\Bigl( \bigl( [y_1,y_2,y_3,y_4], [u_1,u_2,u_3,u_4] \bigr), 
     \bigl( [v_1,v_2,v_3,v_4], [z_1,z_2,z_3,z_4] \bigr) \Bigr),
\end{split}
\end{equation}
due to the fact that $A_{ij}(D,E) = A_{ji}(-D-D_1,E)$.
We see that any
nonzero
column of $\bigl( A_{ij} \bigr)$ gives 
the $u$-coordinates of~$D+E$;
this corresponds to a choice of $j$ with $l_j(D-E+D_1) \neq 0$.

By the same reasoning as in the previous section, if we define
\begin{equation}\label{genus2Jdefn}
\begin{split}
&J_{ij}\Bigl( \bigl( [u_1,u_2,u_3,u_4], [y_1,y_2,y_3,y_4] \bigr), 
     \bigl( [v_1,v_2,v_3,v_4], [z_1,z_2,z_3,z_4] \bigr) \Bigr)\\
&= A_{ji}\Bigl( \bigl( [u_1,u_2,u_3,u_4], [y_1,y_2,y_3,y_4] \bigr), 
     \bigl( [v_1,v_2,v_3,v_4], [z_1,z_2,-z_3,-z_4] \bigr) \Bigr),
\end{split}
\end{equation}
then
$J_{ij} = l_i(D+E+D_1) l_j(D-E)$, and
any
nonzero
column of $\bigl( J_{ij} \bigr)$ gives 
the $y$-coordinates of~$D+E$.


The above discussion gives our desired description of the group law for our
$\Projective^3 \times \Projective^3$ embedding.

\begin{theorem}
\label{theoremGroupLawGenus2}
Let $D$ and $E$ be given as elements of
$\Projective^3 \times \Projective^3$,
as above.  Then the image of $D+E$ in
$\Projective^3 \times \Projective^3$
is given by any column of $\bigl( A_{ij} \bigr)$
in the first $\Projective^3$,
together with any column of $\bigl( J_{ij} \bigr)$ in
the second $\Projective^3$.
More generally, we can take linear combinations of columns: for any
choice of $c_1,\dots,c_4$ and $c'_1,\dots,c'_4$ for which
$\sum_j c_j l_j(D-E+D_1) \neq 0$ and
$\sum_j c'_j l_j(D-E) \neq 0$, we have that $D+E$ is represented by
$([\sum_j c_j A_{ij}]_i, [\sum_j c'_j J_{ij}]_i)
  \in \Projective^3 \times \Projective^3$.
\end{theorem}

\subsection*{A complete set of defining equations.}
Let us return to the defining equations, which we shall
shortly be in a position to describe completely
in Theorem~\ref{theoremdefeqns}. 
To make upcoming expressions more succinct, we shall
now define several quadratic forms.
For $i\in \{1,2,3,4,5,6\}$, let $r_i = r_i(u_1,u_2,u_3,u_4)$
be the following quadrics in~$u_1,u_2,u_3,u_4$,
that we have already encountered before as certain $M_{i,j}$
in~\eqref{genus2biquadD1}:
\begin{equation}\label{ridefn}
\begin{split}
 r_1 &= b^2 (a c (-f u_1^2+e u_2^2)-c e f u_3^2+a f u_4^2),\\
 r_2 &= -b g (a c u_1 u_2-f u_3 u_4),\\
 r_3 &= -a c (d u_1^2+b^2 f u_2^2)+f (a d u_3^2+b^2 c u_4^2),\\
 r_4 &= a b^2 c (c u_1^2+a u_2^2-f u_3^2-u_4^2),\\
 r_5 &= a b c g (u_1 u_2-u_3 u_4),\\
 r_6 &= a c (a d u_1^2-b^2 c e u_2^2+d e u_3^2-b^2 f u_4^2).
\end{split}
\end{equation}
For $i\in \{1,2,3,4,5,6\}$, let $s_i$ be exactly the same quadric 
in $y_1,y_2,y_3,y_4$; that is, define:
\begin{equation}\label{siri}
s_i = r_i(y_1,y_2,y_3,y_4).
\end{equation}
We recall that we previously found a number of quartic forms
satisfied by $u_1,u_2,u_3$, $u_4,y_1,y_2,y_3,y_4$. The~$u_i$ and the~$y_i$
each satisfy the Kummer surface equation~(\ref{kummereqnl}), giving a quartic 
form purely in the~$u_i$, and another purely in the~$y_i$,
that is to say, forms of bidegrees~$(4,0)$ and~$(0,4)$.
We also previously noted the projective equality in
the arrays given in~(\ref{genus2biquadD1}) and~(\ref{genus2biquadD1var}),
which gives forms of bidegree~$(2,2)$.
These do not so far give a complete
set of defining equations. What we now also have available is that 
the columns of the matrix~$\bigl( A_{ij} \bigr)$ given in~(\ref{Agenus2}),
are projectively equal, and we may use this for any specified
choice of~$E$ to give further quartics. 
These quartics arise as the $2 \times 2$ minors
of~$(A_{ij})$, and they are of bidegree~$(2,2)$, due to the fact that
each entry of the matrix is itself of bidegree~$(1,1)$ once $E$ is fixed.
We merely need to apply
this for the choice~$E = D_2$ in order to obtain the remaining
forms of bidegree~$(2,2)$.
We may further derive from these
every $u_i (u_2 s_1 - u_1 s_2)$, for $i\in \{ 1,2,3,4\}$ and
so deduce that $u_2 s_1 - u_1 s_2$ must be satisfied.
We may similarly deduce that $u_2 s_2 - u_1 s_3$, $u_4 s_4 - u_3 s_5$, 
$u_4 s_5 - u_3 s_6$, $y_2 r_1 - y_1 r_2$, $y_2 r_2 - y_1 r_3$,
$y_4 r_4 - y_3 r_5$, $y_4 r_5 - y_3 r_6$ are all satisfied.
At this point, we are able to obtain a complete set of defining equations,
as described in the next theorem.
\begin{theorem}\label{theoremdefeqns}
Let~$K$ be any field, not of characteristic~$2$, and let
$a,b,c \in K$. Let~$\calC$ be as defined in~(\ref{eqnforgenus2C}),
where~$d,e,f,g$ are as in~(\ref{defndefg}), 
with $2abcdefg(a-1)(b^2-1)(c-1)$ nonzero, and let~$\J$ be the
Jacobian variety of~$\calC$. Let~$E_1,E_2,E_3 \in \J(K)$ be the
points of order~$2$ and~$D_1\in \J(K)$ the point of order~$4$,
such that~$2 D_1 = E_1$, described immediately after~(\ref{defndefg}).
For any~$D\in \J(K)$, let $l(D) = [ l_1(D), l_2(D), l_3(D), l_4(D) ]$
be the embedding of the Kummer surface given in~(\ref{kummerchangebasis}). 
Embed the Jacobian variety into $\Projective^3 \times \Projective^3$
according to the embedding $\bigl( l(D), l(D+D_1) \bigr)$ given
in~(\ref{genus2P3xP3}), and let
$\bigl( [u_1,u_2,u_3,u_4], [y_1,y_2,y_3,y_4] \bigr)$ be a member
of this embedding. Furthermore, let~$r_i$ and~$s_i$ be the
quadratic forms given in~(\ref{ridefn}) and~(\ref{siri}),
for $i\in \{ 1,2,3,4,5,6 \}$. Then the following is a set
of defining equations.
\begin{equation}\label{genus2defeqns}
\begin{split}
 & r_2^2 - r_1 r_3,\  s_2^2 - s_1 s_3,\\
 & u_2 s_1 - u_1 s_2,\ u_2 s_2 - u_1 s_3,\ 
   u_4 s_4 - u_3 s_5,\ u_4 s_5 - u_3 s_6,\\
 & y_2 r_1 - y_1 r_2,\  y_2 r_2 - y_1 r_3,\  
   y_4 r_4 - y_3 r_5,\  y_4 r_5 - y_3 r_6,\\
 & r_1 y_3^2 + r_4 y_1^2,\  r_1 y_4^2 + r_6 y_1^2,\  
   r_2 y_3 y_4 + r_5 y_1 y_2,\\
 & a (b u_2 y_1-u_1 y_2) (b u_4 y_3-u_3 y_4) 
    - (e u_3 y_2+b u_4 y_1) (b u_2 y_3-u_1 y_4),\\
 & c (e u_3 y_3+b u_4 y_4) (b u_2 y_2-u_1 y_1) 
    - f (b u_2 y_4-u_1 y_3) (b u_4 y_2-u_3 y_1).
\end{split}
\end{equation}
\end{theorem}
\begin{proof} It is checked in the file~\cite{FlynnMaple}
that all of the above are satisfied, and that they
generate independent equations as follows:
$1$~of bidegree~$(4,0)$, $1$~of bidegree~$(0,4)$,
$16$~of bidegree~$(3,1)$, $16$~of bidegree~$(1,3)$, 
$36$~of bidegree~$(2,2)$,
$4$~of bidegree~$(2,1)$ and $4$~of bidegree~$(1,2)$.
The way that the independence of a collection of $N$ equations is checked
in~\cite{FlynnMaple} is by finding a collection of $N$ monomials
appearing in the equations, and showing that the corresponding
$N\times N$ matrix giving the coefficient of the $i$th monomial in the
$j$th equation has a determinant that is divisible only by factors of
the nonzero discriminant $2abcdefg(a-1)(b^2-1)(c-1)$.  Thus the
equations that we have found are independent for all values of the
parameters that we are considering.
From Theorem~\ref{generatorsInEachBidegree}
we see that we must therefore
have a complete set of defining equations.
\end{proof}
At this stage, we now have a complete description of
our variety; (\ref{genus2defeqns}) gives a set of defining
equations, and the group law is given by any column of the 
matrix~$\bigl( A_{ij} \bigr)$ given in~(\ref{Agenus2}) together with 
any column of~$\bigl( J_{ij} \bigr)$ given in~(\ref{genus2Jdefn}).
It is apparent how much nicer both the defining equations and
description of the group law are here, compared with the
$\Projective^{15}$ embedding given in~\cite{CasselsFlynnBook},
where the defining equations were given as~$72$ quadrics,
and the defining equations on the Jacobian variety were too
enormous to be given explicitly in general.

We also note that the linear maps in~(\ref{mapsonembedding})
and~(\ref{swapmap})
give rise to
a
number of symmetries in our defining equations.
We see that the defining equations
on the third line of~(\ref{genus2defeqns}) are the images
of those on the second line of~\eqref{genus2defeqns}
under the transformation of~(\ref{swapmap}).
So, if we wish, this gives a still more succinct description
in terms of a smaller number of defining equations, together
with their orbits under~(\ref{mapsonembedding}) and~(\ref{swapmap})

\subsection*{Twists of our abelian surface.}
We note that, since the effect of negation, described
in~(\ref{mapsonembedding}), negates $y_3$ and $y_4$, we may create
a twist of our abelian surface by replacing these with
$\sqrt{\kk_1} y_3, \sqrt{\kk_1} y_4$, and this will still be defined
over the ground field~$K$, for the same reasons as described
in the previous section for elliptic curves.
Similarly, $D \mapsto -D - E_1$ negates $u_3,u_4$
and $D \mapsto D + E_3$ negates $u_2,u_4,y_2,y_4$, so we have
the following twists. 
\begin{definition} 
Let~$\J$ be as given in Theorem~\ref{theoremdefeqns}. For any
nonsquare $\kk_1,\kk_2,\kk_3 \in K$, define~$\J^{(\kk_1,\kk_2,\kk_3)}$ 
to be the abelian
surface, defined over~$K$, whose defining equations are
the same
as those
given in Theorem~\ref{theoremdefeqns}, except
that $u_1,u_2,u_3,u_4$ are replaced by
$u_1,\sqrt{\kk_3}u_2, \sqrt{\kk_2}u_3, \sqrt{\kk_2}\sqrt{\kk_3}u_4$,
respectively, and
$y_1,y_2,y_3,y_4$
are
replaced by 
$y_1, \sqrt{\kk_3}y_2, \sqrt{\kk_1} y_3, \sqrt{\kk_1}\sqrt{\kk_3} y_4$, 
respectively.
\end{definition}

\section{Conditions for non-degeneracy of the group law}
\label{section5}

Our goal in this section is to find conditions on the parameters (analogous
to the condition for Edwards curves that~$d$ is nonsquare)
which will imply that the group law is universal; this will 
be stated in Corollary~\ref{genus2universal}
(a consequence of Theorem~\ref{alwaysnonzero}).


Our strategy is to search for suitable $c_1, \dots, c_4$ for which the
sum $\sum_j c_j l_j(F)$ is nonzero for all $K$-rational divisors
$F \in \AV$ (where the~$l_j$ are as defined
in~(\ref{kummerchangebasis})).  
By Theorem~\ref{theoremGroupLawGenus2}, with such a
choice of $\{c_j\}$ (and the same choice of $c'_j = c_j$), the expressions
$\sum_j c_j l_j(D+E+D_1)$ and $\sum_j c_j l_j(D-E)$ will never be zero when
$D$ and $E$ are themselves $K$-rational. This will yield a universal
group law.

We will illustrate three attempts to find such a linear combination
$s = \sum_j c_j l_j$ which does not vanish at any $K$-rational point of
$\AV$.  The first two attempts were instructive near misses, and the
third attempt was successful in identifying a concrete set of
conditions on the parameters which would ensure the nonvanishing of
$s$ on $K$-rational points.

The vanishing locus of $s$ over the algebraic closure $\Kbar$ of $K$
is best understood in terms of linear series.  Recall from
Section~\ref{section2} that $s$ is a section of the line bundle
$\LL^2$.  As in the proof of Lemma~\ref{lemmaTheta1Nonvanishing}, let
$\Theta$, the theta-divisor, be the vanishing locus of
$\theta_{[1],0}$; then $\LL \isomorphic \OO_\AV(\Theta)$, and $\Theta$
is isomorphic to $\calC$.  Then the vanishing loci for the different
choices of $s$ are the divisors on $\AV$ belonging to the linear
series $|2\Theta|$.  For a generic choice of $s$ (over $\Kbar$), its
vanishing locus is a smooth curve of genus~$5$; this follows from the
adjunction formula, which in the case of an abelian surface says that
a smooth genus~$g$ curve $\mathcal{X} \subset \AV$ has self-intersection
$\mathcal{X} \cdot \mathcal{X} = 2g-2$, since the canonical bundle on
$\AV$ is trivial.  So for a typical $K$-rational choice of $s$, we
expect that the $K$-rational points where $s$ vanishes are the
$K$-rational points of a curve of genus~$5$; this is the phenomenon we
observe in our first attempt below.

In our second and third attempts below, we start from the observation
that for $p \in \AV(\Kbar)$, the divisor $(\Theta+p) \union (\Theta-p)$
belongs to $|2\Theta|$.  Here the two translates $\Theta\pm p$ of
$\Theta$ intersect in two points, because $\Theta \cdot \Theta = 2$.
(The intersection points might not be distinct.)
For certain choices of $p$, we can hope that each of the two
irreducible components $\Theta + p$ and $\Theta - p$
is defined over a common quadratic extension of $K$, and
that the two components are conjugate to each other.  In that
situation, the only $K$-rational points on the union are potentially
the intersection points between the two components.

When the two intersection points coincide at a point $q$, then this
point is $K$-rational, and is the only $K$-rational vanishing point of
$s$.
This corresponds to our second attempt below.  In our third attempt, we
identify a situation and choice of $s$ that does not vanish at any
$K$-rational points.  We believe, but have not checked the details,
that in this situation the two points of intersection
are individually defined over a quadratic extension of $K$, but not
over $K$ itself.  The construction in our third attempt has a similar
flavor to the constructions in Section~4.2
of~\cite{AreneKohelRitzenthaler} and Section 2.1
of~\cite{AreneCosset}, which treat the case of the embedding of $\AV$
into $\Projective^{15}$.

\subsection*{Two near miss attempts.}
For our first attempt, let us imagine that~$u_1 = 0$. Suppose also that the
following are not squares in the ground field:
$a,c,d,-e,-cef,af,$ $cd,adf,-ae,cf$.
Then,
as shown in~\cite{FlynnMaple},
one can deduce directly
from the defining equations that $u_2,u_3,u_4,y_1,y_2,y_3,y_4$
are all nonzero and that one can set $u_2=1, y_1=1$,
and use the defining equations to eliminate $u_3,y_2$;
the defining equations then become equivalent to the affine curve:
\begin{equation}\label{genus5a}
\begin{split}
u_4^2 (-f y_4^2+a c)/(-b^2 f y_4^2+a c d) &= 
-y_4^2 a c e (c-1)^2,\\
e y_3^2 (-f y_4^2+a c^2) &= -a^2 c (-y_4^2+c).
\end{split}
\end{equation}
This is birationally equivalent to the following genus~$5$
affine curve:
\begin{equation}\label{genus5b}
\begin{split}
 -a c e (f X^2-a c) (b^2 f X^2-a c d) &= Y^2,\\
  -c e (X^2-c) (f X^2-a c^2) &= Z^2.
\end{split}
\end{equation}
In summary, provided that $a,c,d,-e,-cef,af,cd,adf,-ae,cf$
are non-square in the ground field~$K$ and provided
that the curve~(\ref{genus5b}) has no $K$-rational points,
then any point $\bigl( [u_1,u_2,u_3,u_4], [y_1,y_2,y_3,y_4] \bigr)$
will satisfy that~$u_1$ and~$y_1$ are nonzero, so that
our projective variety is the same as the affine
variety in $U_2 = u_2/u_1, U_3 = u_3/u_1, U_4 = u_4/u_1$ 
and $Y_2 = y_2/y_1, Y_3 = y_3/y_1, Y_4 = y_4/y_1$, and furthermore
there is a universal group law given by
the first column of
matrix~$\bigl( A_{ij} \bigr)$ given in~(\ref{Agenus2}) together with
the first column of~$\bigl( J_{ij} \bigr)$ given in~(\ref{genus2Jdefn}).
This could happen over~$\Q$ but, for a given curve, could not
happen for arbitrarily large finite fields, since the above
genus~$5$ curve would eventually acquire points over these.
Of course, even if the above conditions are not satisfied,
these matrices will always give a complete description of the
group law; it is merely that one then requires different columns
for certain additions, rather than having any particular
column apply universally.

As a second attempt, consider the situation when~$k_3=0$,
which corresponds to~$D\in \J(K)$ which is either
of the form $\{ (x,y), (0,b \sqrt{a c d e g (d e - b^2 f)})\}$
or is the identity. 
Provided that $a c d e g (d e - b^2 f)$ is nonsquare in~$K$,
we see that~$D$ can only be defined over~$K$ if it is the
identity element, and this is the only intersection between~$k_3=0$
and our variety~$\J$. The condition~$k_3=0$, after the linear
change in coordinates~(\ref{kummerchangebasis}),(\ref{entriesofQ}),
corresponds to $a d u_1 - b^2 c e u_2 - d e u_3 + b^2 f u_4$, and
so this will only be zero when~$D$ is the identity.
For any $\bigl( [u_1,u_2,u_3,u_4], [y_1,y_2,y_3,y_4] \bigr)$,
we see that this linear combination of the $u$-coordinates
will only be zero when~$D$ is the identity, and the same linear
combination of the~$y$-coordinates will only be zero
when~$D = -D_1$. If we take our group law to be the same
linear combination of the columns of~$A$ together with the
same linear combination of the columns of~$J$ then, for any given~$D$,
this will only be degenerate when $E = D$ or $E = D+D_1$.
This condition can never give a universal group law, but it
does reduce the exceptional cases to these two values of~$E$,
corresponding to the fact that the curve which geometrically
describes such cases only has two $K$-rational points.

\subsection*{Third and successful attempt.}
We shall obtain a universal group law under conditions merely
that certain expressions are squares and others are non-squares
over our ground field~$K$; this will be stated in 
Corollary~\ref{genus2universal} (a consequence 
of Theorem~\ref{alwaysnonzero}).
In order to obtain the situation we desire 
in our third attempt, we first impose
the conditions that~$a$ and~$cf$ are squares in~$K$. The first
condition forces~$D_2$ to be defined over the field~$K(\sqrt{c})$, 
with conjugation corresponding to negation, and the second condition
forces the first quadratic factor of~$\calC$ in~(\ref{eqnforgenus2C})
to have $K$-rational roots, giving two Weierstrass points.
The conditions that~$a$ and~$cf = c(a+c-1)$ are squares in~$K$ can 
be parametrised as follows. We first set $a = s^2$ and~$cf = c(a+c-1) = u^2$,
and deduce $c(s^2+c-1) = u^2$; we then use $(s,u) = (1,c)$ as our
basepoint and define the parameter $\fa = (u-c)/(s-1)$.
Solving for~$s$ gives $s = (2 c \fa - \fa^2 - c)/(-\fa^2 + c)$.
Hence $a = \bigl( (2 c \fa - \fa^2 - c)/(-\fa^2 + c) \bigr)^2$,
which gives that $c f = \delta^2$, where 
$\delta = c(\fa^2 - 2 \fa + c)/(\fa^2 - c)$.
We can now think of our family $\J = \J_{\fa,b,c}$ as being
parametrised by $\fa,b,c \in K$.
One of the $K$-rational Weierstrass points on~$\calC$ is then~$(x_0,0)$,
where $x_0 = (\fa^2 + 2 \fa b^2 c  - 2 \fa c + c)/(\fa^2 + c)$. 
Suppose that $D \in \J(K)$ has the form $\{ (x,y), (x_0,0) \} + D_2$.
Since the conjugation $\sigma : \sqrt{c} \mapsto -\sqrt{c}$ negates~$D_2$
this would force $\{ (x,y), (x,-y)^\sigma \} = E_2$, which we can
hope to prevent, by imposing mild constraints on the parameters.
After describing such~$D$ in terms of our embedding, we find that
it corresponds (the details are in the file~\cite{FlynnMaple})
to the condition $c u_1 - \delta u_3 = 0$.
This motivates taking the intersection of $c u_1 - \delta u_3 = 0$
and our variety, hoping that there are arithmetic conditions on
the parameters which prevent this intersection having
a $K$-rational point. This approach turned out to be successful,
as described in the following theorem, where the condition 
is merely that three expressions in the parameters are nonsquares in~$K$,
similar to the nonsquare-$d$ condition for Edwards curves.
\begin{theorem}\label{alwaysnonzero}
Let~$\J = \J_{a,b,c}$ be as given in Theorem~\ref{theoremdefeqns}, defined
over a field~$K$, not of characteristic~$2$.
Let $d,e,f,g$ be as defined in~(\ref{defndefg}).
Let 
\begin{equation}\label{asoln}
a = \Bigl( \frac{\fa^2 - 2 \fa c + c}{\fa^2 - c} \Bigr)^2,
\end{equation}
for some~$\fa \in K$, so that we may think of~$\J = \J_{\fa,b,c}$
now as a family in terms of the parameters~$\fa,b,c \in K$.
Then~$c f = \delta^2$, where $\delta = c(\fa^2 - 2 \fa + c)/(\fa^2 - c)$,
and~$a = \rho^2$, where $\rho = (\fa^2 - 2 \fa c + c)/(\fa^2 - c)$.

Suppose that $c, c d, g (g - b^2 (c-1))$ are nonsquares in~$K$.
Then~$c u_1 - \delta u_3$ and~$c y_1 - \delta y_3$ are nonzero for every
$\bigl( [u_1,u_2,u_3,u_4],[y_1,y_2,y_3,y_4] \bigr) \in \J (K)$.
\end{theorem}
\begin{proof}
First note that, since~$c$ is nonsquare, the expression $\fa^2 - c$
is nonzero.
The defining equations given in~(\ref{genus2defeqns}) include
$y_2 r_2 - y_1 r_3 = 0$, $y_4 r_4 - y_3 r_5 = 0$ and 
$r_2 y_3 y_4 + r_5 y_1 y_2 = 0$.
If we now take the sum of: $r_5^2 y_1$ times the first of these,
$-r_2^2 y_4$ times the second of these, and $-r_2 r_5$ times the third
of these, we obtain
\begin{equation}\label{r2r3r4r5eqn}
r_2^2 r_4 y_4^2 = -r_3 r_5^2 y_1^2.
\end{equation}
Imagine that $u_1 - (\delta/c) u_3 = 0$.
Substituting $u_1 = (\delta/c) u_3$ into the above factors, we find that
\begin{equation}\label{r2r3r3r5soln}
\begin{split}
r_2 &= 
- b g \delta u_3 \bigl( a u_2 - \frac{\delta}{c} u_4 \bigr),\\
r_3 &=
- b^2 \delta^2 ( \rho u_2 + u_4 ) ( \rho u_2 - u_4 ),\\
r_4 &=
b^2 c \rho^2 ( \rho u_2 + u_4 ) ( \rho u_2 - u_4 ),\\
r_5 &=
b c g \rho^2 u_3 \bigl( \frac{\delta}{c} u_2 - u_4 ).
\end{split}
\end{equation}
We now see that the two sides of~(\ref{r2r3r4r5eqn}),
if nonzero,
have quotient~$c$
modulo squares. We are assuming that~$c$ is nonsquare in~$K$, so it
follows that both sides of~(\ref{r2r3r4r5eqn}) must be zero.
The factors $b,c,g,\delta,\rho$ are all factors of the
discriminant of~$\calC$, which we are assuming to be nonzero.
It follows that one of the other factors of the left hand
side must be zero, namely one of: $\rho u_2 - u_4$, $\rho u_2 + u_4$,
$u_3$, $y_4$ or~$a u_2 - (\delta/c) u_4$ must be zero.

Consider the case $\rho u_2 - u_4 = 0$. If we substitute
both $u_1 = (\delta/c) u_3$ and $u_2 = u_4/\rho$ into
the defining equation $r_2^2 - r_1 r_3$,
this
gives 
$4 b^2 g^2 c f \fa^2 (c-1)^2 u_3^2 u_4^2/(\fa^2-c)^2 = 0$,
so that either~$u_3=0$ or~$u_4=0$, which forces either
$u_1 = u_3 = 0$ or $u_2 = u_4 = 0$. For the subcase when $u_1 = u_3 = 0$,
then the defining equations $y_2 r_1 - y_1 r_2$ and $y_4 r_5 - y_3 r_6$
give nonzero constants times $u_4^2 y_2$ and $u_4^2 y_3$, respectively;
but $u_4 \not= 0$ (since otherwise all $u$-coordinates would be zero),
so that now $y_2 = y_3 = 0$; putting this into the defining
equation $s_2^2 - s_1 s_3$ gives the factors
$y_4^2 - c y_1^2$ and $b^2 \delta^2 y_4^2 - c d \rho^2 y_1^2$;
the fact that~$c$ and $cd$ are nonsquares then forces $y_1 = y_4 = 0$,
so that all $y$-coordinates are zero, a contradiction.
The subcase $u_2 = u_4 = 0$ gives the same contradiction, using
the same definining equations. We deduce that the case
$\rho u_2 - u_4 = 0$ is impossible.

The case when $\rho u_2 + u_4 = 0$ is also incompatible with
$c$ and $cd$ being nonsquares, using the same defining equations
as the previous case.

The remaining cases $u_3$, $y_4$ and $a u_2 - (\delta/c) u_4$
are shown to be nonzero in a similar style (it is the last 
of these which uses the condition that $g (g - b^2 (c-1))$
is non-square); we have
put the details in the file~\cite{FlynnMaple}.

So we now have, in all cases, a contradiction arising from 
our initial assumption that $u_1 - (\delta/c) u_3 = 0$.
It follows that $c u_1 - \delta u_3$ is always nonzero, as required.
Since $[y_1,y_2,y_3,y_4]$ are the $u$-coordinates
of $D+D_1 \in \J(K)$, it follows that $c y_1 - \delta y_3$
is also always nonzero.
\end{proof}
We observe here that, as long as the conditions of 
Theorem~\ref{alwaysnonzero} are satisfied,
we can treat
the elements of~$\J(K)$ in terms of affine coordinates
in ${\bf A}^3(K) \times {\bf A}^3(K)$. Specifically,
we may represent any member
of~$\J(K)$ by $\bigl( (U_2,U_3,U_4), (Y_2,Y_3,Y_4) \bigr)$,
where each $U_i = u_i/(c u_1 - \delta u_3)$ and 
each $Y_i = y_i/(c y_1 - \delta y_3)$.

The existence of a universally nonzero linear combination
of coordinates now gives a universal group law, as follows.
\begin{corollary}\label{genus2universal}
Let~$\J = \J_{\fa,b,c}$ satisfy the same conditions as 
in Theorem~\ref{alwaysnonzero}, namely that
$c, c d, g (g - b^2 (c-1))$ are nonsquares in~$K$.
Let~$A$ and $J$ be the matrices defined in~(\ref{Agenus2})
and~(\ref{genus2Jdefn}), respectively. Then
\begin{equation}\label{genus2universaleqns}
\begin{split}
\Bigl( &\bigl[ c A_{11} - \delta A_{13}, c A_{21} - \delta A_{23},  
c A_{31} - \delta A_{33}, c A_{41} - \delta A_{43} \bigr],\\  
&\bigl[ c J_{11} - \delta J_{13}, c J_{21} - \delta J_{23},  
c J_{31} - \delta J_{33}, c J_{41} - \delta J_{43} \bigr] \Bigr),
\end{split}
\end{equation}
gives a universal group law on~$\J(K)$.
\end{corollary}
\begin{proof}
The equations in~(\ref{genus2universaleqns}) are the
linear combination: $c$ times the first column minus $\delta$ times
the third column, for the matrices~$A$ and~$J$, respectively.
The entries of the first array are 
just $l_i(D+E) \bigl( c l_1(D-E+D_1) - \delta l_3(D-E+D_1) \bigr)$, 
for $i \in \{ 1,2,3,4 \}$. Since $D-E+D_1 \in \J(K)$, we know
from Theorem~\ref{alwaysnonzero} that $c l_1(D-E+D_1) - \delta l_3(D-E+D_1)$
is nonzero, and so the elements of this array are not all zero,
and they gives the $u$-coordinates of~$D+E$. Similarly the
elements of the second array are not all zero, and they give
the $y$-coordinates of~$D+E$, as required.
\end{proof}

\subsection*{Satisfiability of the conditions on the parameters.}
We should also comment on the satisfiability of the
condition that $c, cd, g(g - b^2(c-1))$ are nonsquares in~$K$.
A sufficient condition (which is equivalent for a finite field) is
for~$d$ to be square, and for~$c$ and $g(g - b^2(c-1))$ to be nonsquares.
The condition for~$d$ to be square is:
$b^2 c - c + 1 = x^2$, for some~$x$.
Regarding this as a conic in $b,x$, we may use the basepoint
$(b,x) = (1,1)$, and define the parameter $t = (x-1)/(b-1)$.
After solving for~$b$, we can write $b = (t^2+c-2t)/(t^2-c)$
in terms of our parameter~$t$.
At this stage, we can think of $\fa,t,c$ as our parameters,
and we now only require that~$c$ and~$g(g - b^2(c-1))$ are 
nonsquares in~$K$. 
Note that $g (g - b^2 (c-1))$ modulo squares is the same
as the following polynomial in our parameters $\fa,t,c$.
\begin{equation}\label{eq:nonsquarecond}
\begin{split}
& 2 (t^2 w^2 + c t^2 - 4 c t w + c w^2 + c^2)\\
& (c t^2 w^2 +c^2 t^2 + 4 c^2 t w + c^2 w^2 - 4 c t^2 w - 4 c t w^2\\
& + t^2 w^2 + c^3 - 4 c^2 t - 4 c^2 w + c t^2 + 4 c t w + c w^2 + c^2)\\
& (2 c^2 t^4 w^2 - 2 c t^4 w^3 + t^4 w^4 - 4 c^3 t^2 w^2 - 2 c^2 t^4 w
 + 4 c^2 t^2 w^3 - 2 t^3 w^4\\ 
& + 2 c^4 w^2 + 4 c^3 t^2 w - 2 c^3 w^3 + c^2 t^4 - 4 c^2 t^2 w^2 
 + c^2 w^4 + 4 c t^3 w^2\\
&  - 2 c t w^4 + 2 t^2 w^4 - 2 c^4 w - 2 c^2 t^3 + 4 c^2 t w^2 
- 4 c t^2 w^2 + c^4 - 2 c^3 t + 2 c^2 t^2).
\end{split}
\end{equation}
This has discriminant (with respect to t):
\begin{equation}\label{eq:discnonsquarecond}
\begin{split}
& 2^{40} c^{28} (c-1)^{20} (2 c w^2 - 2 c w - w^2 + c)^2 
(2 c^2 - 2 c w + w^2 - c)^2\\ 
& (2 c w - w^2 - c)^{24} (-w^2 + c)^{24} (w^2 + c - 2 w)^4,
\end{split}
\end{equation}
and the coefficient of the highest power of~$t$ is:
\begin{equation}\label{eq:highpowert}
2 (w^2 + c) (c w^2 + c^2 - 4 c w + w^2 + c)
(2 c^2 w^2 - 2 c w^3 + w^4 - 2 c^2 w + c^2).
\end{equation}
Assuming that our field~$K$ contains nonsquares, take~$c$
to be any fixed nonsquare in~$K$. Now let~$w$ be any member
of~$K$ for which~(\ref{eq:discnonsquarecond}) 
and~(\ref{eq:highpowert}) are nonzero, which means
avoiding at most~$16$ values in~$K$, since
automatically $-w^2 + c$ is nonzero. Let $\phi(t)$ denote 
the polynomial in~$t$ obtained by substituting these
values of~$c,w$ into~(\ref{eq:nonsquarecond}).
The will be a degree~$8$ polynomial in~$t$ with no repeated
roots, and our only remaining requirement is to choose
values for~$t$ such that $\phi(t)$ is nonsquare in~$K$
(as well as avoiding any values of~$t$ for which the
disciminant of our original curve is zero).

It is now clear that there are plentiful examples of
our conditions being satisfied when~$K$ is a number
field or a finite field (for sufficiently large~$p$). 
The curve $y^2 = \phi(t)$ is of genus~$3$ over~$K$.
For example, when~$K$ is a number field, by 
Faltings' Theorem~\cite{Faltings} this curve
has only finitely many points, and so we need only avoid 
finitely many values of~$t$. When~$K = \F_q$ is a finite field
with~$q$ elements, the Hasse-Weil bounds (see Chapter~3 of~\cite{Moreno})
tells us that the number of points over~$K$ on
the curve $y^2 = \phi(t)$ is in the range from $q+1-2g\sqrt{q}$
to $q+1+2g\sqrt{q}$, where here $g=3$ is the genus of this curve;
hence roughly half of the values of~$t \in K$ will give nonsquare values
for~$\phi(t)$. It follows from the above discussion that there
will be examples in arbitrarily large finite fields
for which the conditions are satisfied.
As an explicit example, let $K = \F_{1201}$.
Then the conditions are satisfied for $\fa = 6, b = 7, c = 11$,
since then $c = 11, cd = 1015$ and $g(g-b^2(c-1)) = 202$
are all nonsquares in~$\F_{1201}$.


\subsection*{Data Availability}

Data sharing is not applicable to this article as no datasets were
generated or analysed during the current study.

\subsection*{Auxiliary Files}

The supplementary information file~\cite{FlynnMaple} includes details
of the calculations performed using the computer algebra software
Maple. The file is available from the website of the first-named
author, as described in the references, and is also available as an
ancillary file from \texttt{arxiv:2211.01450}.
There is also a shortened version of this file~\cite{FlynnShort}
which gives only the assignments of the main objects (such as
the diagonalising change in basis, the defining equations and
the group law), which can be used in any algebra package.

\subsection*{Conflict of Interest}

The authors certify that there is no actual or potential conflict of
interest in relation to this article.

%


\end{document}